\documentclass[11pt,reqno]{amsart}
\usepackage{hyperref}
\usepackage{amssymb, amsmath, amsthm, amsfonts}
\usepackage{verbatim}
\usepackage{bbm}
\usepackage{enumerate}
\usepackage{dsfont}
\usepackage[mathscr]{eucal}
\usepackage{tikz}
\usepackage[normalem]{ulem}

\usetikzlibrary{patterns}

\usepackage{todonotes}
\usepackage{comment}
\usepackage{color}
\definecolor{gray}{gray}{0.5}

\def\OPB{{\mathrm{Op}_{B_1,B_2}}}
\def\simpleone{{\mathrm S}_{B_1}}
\def\simpledual{{\mathrm S}_{B_2^*}}
\newcommand{\vertiii}[1]{{\left\vert\kern-0.25ex\left\vert\kern-0.25ex\left\vert #1 
\right\vert\kern-0.25ex\right\vert\kern-0.25ex\right\vert}}
    
\newcommand{\Eann}{\mathcal E_{\mathrm{ann}}}
\def\av{\mathrm{av}}

\def\tLa{{\widetilde \Lambda}}

\def\lc{\lesssim}

\def\sD{\mathscr{D}}
\def\sH{\mathscr{H}}
\def\sA{\mathscr{A}}
\def\sR{\mathscr{R}}
\def\sV{\mathscr{V}}

\newcommand{\sS}{\mathscr S}

\def\eps{\varepsilon}
\def\bbone{{\mathbbm 1}}
\newcommand{\dil}{{\text{\rm Dil}}}

\newcommand{\sC}{\mathscr C}
\newcommand{\sL}{\mathscr L}

\newcommand{\Lamaxga}{\Lambda^{*}}

\newcommand{\ci}[1]{_{{}_{\!\scriptstyle{#1}}}}

\newcommand{\floor}[1]{\lfloor #1 \rfloor }

\newcommand{\Be}{\begin{equation}}
\newcommand{\Ee}{\end{equation}}

\newcommand{\Bm}{\begin{multline}}
\newcommand{\Em}{\end{multline}}

\def\lac{\mathrm{lac}}

\newcommand\Ch{Ch.}
\newcommand{\dyad}{{\mathrm{dyad}}}
\def\intslash{\rlap{\kern  .32em $\mspace {.5mu}\backslash$ }\int}
\def\qsl{{\rlap{\kern  .32em $\mspace {.5mu}\backslash$ }\int_{Q_x}}}

\def\Im{\operatorname{Im\,}}

\def\F{\mathcal F}

\def\lc{\lesssim}

\def\floor#1{{\lfloor #1 \rfloor }}
\def\emph#1{{\it #1 }}
\def\diam{{\text{\rm  diam}}}

\def\Id{\mathrm{I}}

\def\ga{\gamma}

\def\cf{{\it cf}}

\def\loc{{\mathrm{loc}}}

\def\meas{{\mathrm{meas}}}

\def\dist{{\mathrm{dist}}}

\def\supp{{\mathrm{supp}}}

\newcommand{\dilq}[2]{{#1}{#2}}

\newcommand{\tr}[1]{\dilq{3}{#1}}

\def\inn#1#2{\langle#1,#2\rangle}
\def\biginn#1#2{\big\langle#1,#2\big\rangle}
\def\jp#1{{\langle#1\rangle}}

\def\ga{\gamma}             
\def\eps{\varepsilon}
\def\ep{\epsilon}
\def\ka{\kappa}
\def\la{\lambda}             \def\La{\Lambda}
\def\vphi{\varphi}
\def\om{\omega}              \def\Om{\Omega}
\def\vth{\vartheta}

\def\fD{{\mathfrak {D}}}

\def\fG{{\mathfrak {G}}}

\def\fQ{{\mathfrak {Q}}}
\def\fR{{\mathfrak {R}}}
\def\fS{{\mathfrak {S}}}

\def\fa{{\mathfrak {a}}}
\def\fb{{\mathfrak {b}}}

\def\fz{{\mathfrak {z}}}

\def\bbC{{\mathbb {C}}}
\def\bbD{{\mathbb {D}}}
\def\bbE{{\mathbb {E}}}

\def\bbN{{\mathbb {N}}}

\def\bbR{{\mathbb {R}}}

\def\bbZ{{\mathbb {Z}}}

\def\cB{{\mathcal {B}}}
\def\cC{{\mathcal {C}}}

\def\cE{{\mathcal {E}}}
\def\cF{{\mathcal {F}}}

\def\cI{{\mathcal {I}}}

\def\cK{{\mathcal {K}}}

\def\cM{{\mathcal {M}}}

\def\cS{{\mathcal {S}}}
\def\cT{{\mathcal {T}}}

\def\cV{{\mathcal {V}}}
\def\cW{{\mathcal {W}}}

\def\emph#1{{\it #1}}
\def\textbf#1{{\bf #1}}
\def\Sp{{\mathrm{Sp}}}

\def\beq{\begin{equation}}
\def\endeq{\end{equation}}

\def\bs{\begin{split}}
\def\es{\end{split}}

\theoremstyle{plain}
\newtheorem{thm}{Theorem}[section]
\newtheorem{prop}[thm]{Proposition}

\newtheorem{lem}[thm]{Lemma}
\newtheorem{cor}[thm]{Corollary}
\newtheorem{definition}[thm]{Definition}
\newtheorem{claim}[thm]{Claim}

\newtheorem*{thm*}{Theorem}
\newtheorem*{conj*}{Conjecture}
\newtheorem*{openproblem*}{Open Problem}

\theoremstyle{remark}
\newtheorem{rem}[thm]{Remark}

\newtheorem*{remarka}{Remark}
\newtheorem*{remarksa}{Remarks}
\newtheorem*{definitiona}{Definition}

\numberwithin{equation}{section}

\subjclass[2020]{42B15, 42B20, 42B25}
\keywords{Sparse domination, Fourier multipliers, maximal functions, square functions, variation norm.}

\begin{document}

\title
[Multi-scale sparse domination]
{Multi-scale sparse domination}
\author[D. Beltran \ \ \ \ \ \  \ J. Roos \ \ \ \ \ \  \ A. Seeger ] {David Beltran  \ \ \ \ Joris Roos \ \ \ \ Andreas Seeger }

\address{David Beltran: Department of Mathematics, University of Wisconsin-Madison, 480 Lincoln Dr, Madison, WI-53706, USA}
\curraddr{Department of d'An\'alisi Mathem\'atica, Universitat de Valencia, Dr. Moliner 50, 46100 Burjassot, Spain}
\email{david.beltran@uv.es}

\address{Joris Roos: Department of Mathematics and Statistics,  University of Massachusetts Lowell,
1 University Ave. 
Lowell, MA 01854, USA, and 
School of Mathematics, The University of Edinburgh, James Clerk Maxwell Building, Peter Guthrie Tait Rd, Edinburgh EH9 3FD, UK}
\email{joris\_roos@uml.edu}

\address{Andreas Seeger: Department of Mathematics, University of Wisconsin-Madison, 480 Lincoln Dr, Madison, WI-53706, USA}
\email{seeger@math.wisc.edu}

\begin{abstract}

We prove a  bilinear form sparse domination theorem that applies to many multi-scale operators beyond Calder\'on--Zyg\-mund theory,   and also  establish  necessary conditions.  Among the  applications, we cover  large classes of Fourier multipliers, maximal functions, square functions and variation norm operators.
\end{abstract}


\maketitle

\setcounter{tocdepth}{4}
\tableofcontents

\newpage

\section{Introduction}
Sparse domination results have received
considerable  interest in recent years since the fundamental work of Lerner on Calder\'on--Zygmund operators \cite{LeCZ, LeA2}, which provided an alternative proof  of
the $A_2$-theorem \cite{Hytonen}. The original Banach space domination result was refined and streamlined to a pointwise result \cite{CAR2014,lerner-nazarov,Lac2015,LeNew}, but it is the concept of sparse domination in terms of bilinear (or multilinear) forms \cite{bernicot-frey-petermichl, culiuc-diplinio-ou} that has allowed to extend the subject to many operators in harmonic analysis beyond the scope of Calder\'on--Zygmund theory. Among other examples, one may find the bilinear Hilbert transform \cite{culiuc-diplinio-ou},  singular integrals with limited regularity assumptions \cite{conde-alonso-etal, BeneaBernicot, LernerRev2019},  Bochner--Riesz operators \cite{benea-bernicot-luque, LaceyMenaReguera}, spherical maximal functions \cite{laceyJdA19}, singular Radon transforms \cite{cladek-ou,roberlin,hu-thesispaper},  pseudo-differential operators \cite{beltran-cladek}, maximally modulated singular integrals \cite{DPL14, Beltran-Thesis}, non-integral square functions \cite{BBR}, and  variational operators \cite{DiPlinioDoUraltsev, FrancaSilvaZorinKranich, BeneaMuscalu}, as well as results in a discrete setting (see for instance \cite{KLdiscrete,CKL,AHR}). 

Many operators in  analysis have a multiscale structure, either on the space or frequency side. We consider sums 
\[T=\sum_{j\in\bbZ} T_j,\]
where the Schwartz kernel of $T_j$ 
is  supported in a $2^j$ neighborhood of the diagonal and where suitable 
 rescalings of the individual operators $T_j$ and their adjoints satisfy uniform $L^p\to L^q$ bounds. Moreover we assume  that all  partial sums  $\sum_{j=N_1}^{N_2}  T_j$ satisfy uniform $L^p \to L^{p,\infty}$ and $L^{q,1} \to L^q$ bounds.
 The goal  of this paper 
is to show   bilinear form  $(p, q')$-sparse  domination  results (with 
$q'=q/(q-1)$  the dual exponent) and investigate to which extent our assumptions are necessary.
  We prove such results under  a very mild additional   regularity assumption on the rescaled pieces;   
for a  precise statement see Theorem \ref{mainthm}  below. 
To increase applicability, we cover  vector-valued situations, thus  consider functions with values in a Banach space $B_1$ and operators that map simple $B_1$-valued functions to  functions with values in a Banach space $B_2$.
Our results apply to many classes of operators beyond Calder\'on--Zyg\-mund theory, and cover 
general classes of convolution operators with weak assumptions on the dyadic frequency localizations, together with  associated maximal functions, square functions, variation norm operators, and more. 
See Theorem \ref{thm:MEintro} for a particularly clean result on translation invariant maximal functions.
We shall formulate 
the results with respect to cubes in the standard Euclidean geometry 
but there are no fundamental  obstructions to extend them to other geometries involving nonisotropic dilations (see \textit{e.g.} \cite{cladek-ou}). Our approach to sparse domination extends ideas in the papers by  Lacey \cite{laceyJdA19} on spherical maximal functions and by R. Oberlin \cite{roberlin} on singular Radon transforms to more general situations.

We now  describe the framework for our main theorem and first review  basic definitions. For a Banach space $B$ 
let $\mathrm S_B$ be the space of all $B$-valued simple functions on $\bbR^d$ with compact support, i.e. all functions of the form $f=\sum_{i=1}^N a_i \bbone_{E_i}$ where $a_i\in B$ and $E_i$ are Lebesgue measurable subsets of $\bbR^d$ contained in a compact set. For Banach spaces $B_1$, $B_2$ we 
consider  the space $\OPB$ of  linear  operators $T$ mapping functions in $\simpleone$ to weakly measurable $B_2$-valued functions (see \textit{e.g.} \cite{hytonen-etal} for an exposition of Banach-space  integration theory) with the property that $x\mapsto \inn{Tf(x)}{\la} $ is locally integrable  for any bounded linear functional $\la \in B_2^*$.  If $T\in \OPB$, then the 
 integral 
 $$\biginn{Tf_1}{f_2}=\int_{\bbR^d}  \inn{Tf_1(x)}{f_2(x)}_{(B_2, B_2^*)} dx $$ is well-defined 
 for all $f_1\in \simpleone$ and $f_2\in \simpledual$. 
 For a Banach space $B$ and $p,r\in[1,\infty]$ we define the Lorentz space $L^{p,r}_B$ as the space of strongly measurable functions $f:\bbR^d\to B$ so that the function $x\mapsto |f(x)|_B$ is in the scalar Lorentz space $L^{p,r}$ (and we endow $L^{p,r}_B$ with the topology inherited from $L^{p,r}$). In particular, $L^{p}_B=L^{p,p}_B$ coincides with the standard Banach space valued $L^p$ space as defined in \cite{hytonen-etal}, up to equivalence of norms. If $p\in (1,\infty)$ and $r\in[1,\infty]$, then $L^{p,r}_B$ is normable and we write $\|\cdot\|_{L^{p,r}_B}$ to denote the norm induced by the norm on scalar $L^{p,r}$  defined via the maximal function of the nonincreasing rearrangement \cite{hunt}.

In the definition of sparse forms it is convenient to work with a dyadic lattice $\fQ=\cup_{k\in \bbZ}\fQ_k$ of cubes, in the sense of Lerner and Nazarov \cite[\S2]{lerner-nazarov}.  A prototypical example is  when the cubes in the $k$-th generation $\fQ_k$ are given  by 
\[ 
\fQ_k=\begin{cases} \{ 2^{-k} \fz + [-\tfrac 13 2^{-k},\tfrac 13 2^{-k+1})^d: \fz \in \bbZ^d \} &\text{ if $k$ is odd,}
\\
\{ 2^{-k} \fz + [-\tfrac 13 2^{-k+1},\tfrac 13 2^{-k})^d: \fz \in \bbZ^d \} &\text{ if $k$ is even,}
\end{cases}
\] but many other choices are possible. Notice in this example the cubes in $\fQ_k$ have side length $2^{-k}$.
This family satisfies the three axioms of a {\it dyadic lattice} in \cite{lerner-nazarov}. 
We briefly review the definition. $\fQ$ is a dyadic lattice if  
\begin{enumerate}
\item[(i)] every child of a cube $Q\in \fQ$ is in $\fQ$, 
\item[(ii)] every two cubes $Q$, $Q'$ have a common ancestor in $\fQ$, and 
\item[(iii)] every compact set in $\bbR^d$ is contained in a cube in $\fQ$.
\end{enumerate}
For each dyadic lattice there is an $\alpha\in [1,2)$ such that all cubes $Q\in \fQ$ are of side length $\alpha 2^{-k}$ for some $k\in \bbZ$. Fixing $k$ we then call the cubes of side length $\alpha 2^{-k}$ the  $k$-th generation cubes in $\fQ$.
If $Q\in \fQ$  we can, for every $l\ge 0$, tile $Q$ 
into disjoint subcubes  $Q$ of side length equal to $2^{-l}$ times the side length of $Q$. We denote this family by $\sD_l(Q)$ and let $\sD(Q)=\cup_{l\ge 0}\sD_l(Q)$, the family of all dyadic subcubes of $Q$. Then for every $Q\in \fQ$ we have $\sD(Q)\subset \fQ$. Note that because of condition (iii) the standard dyadic lattice is not a dyadic lattice in the above sense.

\begin{definitiona}
Let $0<\gamma<1$. A collection $\fS \subset \fQ$ is {\it $\ga$-sparse} if for 
every $Q\in \fS$ there is a  measurable subset $E_Q\subset Q$ such that $|E_Q|\ge \ga |Q|$ and such that the sets on the family $\{E_Q : Q \in \fS\}$ are pairwise disjoint.
\end{definitiona}
We next review the concept of sparse domination.
Given a cube $Q$, $1 \leq p < \infty$ and a $B$-valued strongly measurable locally integrable function $f$ we use the notations
\[\av_Qf= |Q|^{-1}\int_Q f(x) dx,  
\qquad \jp f_{Q,p,B} = \Big(|Q|^{-1}\int_Q |f(x)|_B^p dx\Big)^{1/p}  \] for the average of $f$ over $Q$ and the $L^p$ norm on $Q$ with normalized measure, thus $\jp f_{Q,p,B}=(\av_Q|f|_B^p)^{1/p}$.
For an operator $T \in \mathrm{Op}_{B_1,B_2}$  we say that pointwise sparse domination \cite{CAR2014, lerner-nazarov} by $L^p$-averages holds if for every $f\in\simpleone$
there are at most $3^d$ sparse families $\fS_i(f) $ such that \Be\label{pointwisesparse} |Tf(x)|_{B_2}\le C \sum_{i=1}^{3^d}  \sum_{Q\in \fS_i(f)} \langle f\rangle _{Q,p,B_1}\bbone_{Q}(x)  \quad\text{ for a.e. }x \Ee and we  denote by $\|T\|_{\mathrm{sp}_\ga(p,B_1,B_2) }$ the  infimum over all $C$ such that \eqref{pointwisesparse} holds for some collection of $3^d$ $\gamma$-sparse families depending on $f$.
 
 For many operators it is not possible to obtain pointwise sparse domination  and the concept  of  sparse domination of bilinear forms, which  goes back to \cite{bernicot-frey-petermichl} and \cite{culiuc-diplinio-ou}, is an appropriate substitute. 
Given a $\gamma$-sparse collection  of  cubes $\fS$ and $1 \leq p_1, p_2 < \infty$, one defines 
an  associated sparse $(p_1,p_2)$-form acting on  pairs $(f_1,f_2)$ where $f_1$ is 
a simple $B_1$-valued function and 
$f_2$ is 
a simple $B_2^*$-valued function. It
is given by 
\Be\label{sparseform}
\La^{\fS}_{p_1,B_1,p_2,B_2^*}(f_1,f_2) 
 = \sum_{Q\in \fS} |Q|  \langle f_1\rangle_{Q,{p_1},{B_1}} 
 \langle f_2\rangle_{Q,{p_2},{B_2^*}}, 
 \Ee 
and will be abbreviated by $\La^{\fS}_{p_1,p_2}(f_1,f_2) $ if the choice of $B_1, B_2^*$ is clear from context. The form \eqref{sparseform} acts a  bi-sublinear form on 
$(|f_1|_{B_1}, |f_2|_{B_2^*})$. All sparse forms are dominated by a 
maximal form 
 \begin{align}\label{eqn:maxform}
  \Lamaxga_{ {p_1},{B_1}, {p_2},{B_2^*}} (f_1,f_2)
  &= \sup_{\fS:  \text{$\ga$-sparse}} \La^{\fS}_{p_1,B_1,p_2,B_2^*}(f_1,f_2) ,
 \end{align} 
 again also abbreviated by $\Lamaxga_{{p_1}, {p_2}} (f_1,f_2)$ if the choice of $B_1, B_2^*$ is clear from the context.
 The maximal form may  not be   a sparse form itself but, obviously,  for every $f_1,f_2$ there exists a sparse family $\fS(f_1,f_2)$ such that $\Lambda^{\fS(f_1,f_2)} _{p_1,B_1,p_2,B_2^*}(f_1,f_2)\ge \frac 12 \La^*_{p_1,B_1,p_2,B_2^*} (f_1,f_2)$ (\cf.   \cite{LaceyMena2017}, \cite{CuliucDiPlinioOu2017} for more explicit constructions). 
 Note from \eqref{sparseform} that for each pair of simple functions $(f_1,f_2)$,
 \[\Lamaxga_{ {p_1},{B_1}, {p_2},{B_2^*}} (f_1,f_2)\le \gamma^{-1} \|f_1\|_\infty \|f_2\|_\infty \,  \mathrm{meas}(\supp f_1\cup \supp f_2)<\infty.\]
  
 We say that $T\in \OPB$  satisfies a  sparse $(p_1,p_2)$ bound if there is a constant $C$ so that 
 for all  $f_1\in \simpleone$ and $f_2\in \simpledual$  the inequality 
\Be \label{sparseineq} \big|\inn{Tf_1} {f_2}\big| \, 
\le \, C \Lamaxga_{p_1,B_1,p_2,B_2^*} (f_1,f_2)
\Ee
is satisfied.  The best constant in \eqref{sparseineq}  
defines a norm $\|\cdot\|_{\Sp_\gamma({p_1},{B_1}, {p_2},{B_2^*})} 
$ on a subspace of  $\OPB$. 
Thus  $\|T\|_{\Sp_\ga({p_1},{B_1};{p_2},{B_2^*})}$ is given by 
\begin{equation}\label{sparse norm}
\sup\Big\{
\frac {|\inn{Tf_1}{f_2}|}{\Lamaxga_{{p_1},{B_1} ,{p_2},{B_2^*}} (f_1,f_2)}: f_1\in\simpleone, \, f_2\in \simpledual, \,\,   f_i\neq 0,\, i=1,2\Big\},
\end{equation}
where $f_i \neq 0$ means that $f_i(x)\neq 0$ on a set of positive measure.  It is then immediate that 
$
\|T\|_{\Sp_\ga(p_1,B_1,p_2, B_2^*)} \le \|T\|_{\mathrm{sp}_\ga(p_1,B_1, B_2)}$ for $p_2\ge 1$.  %
It can be shown that the  space of operators in $\OPB$
for which \eqref{sparseineq} holds for all $f_1$, $f_2$ with a finite $C$ does not depend on $\gamma$. We denote this space by $\Sp(p_1, B_1, p_2, B_2^*)$ or simply $\Sp(p_1, p_2)$  if the choices of $B_1, B_2^*$ are clear from context.
The norms
$\|\cdot\|_{\Sp_\ga({p_1},{B_1},{p_2},{B_2^*})}$, $0<\gamma<1$,  are equivalent norms on $\Sp(p_1,B_1,p_2,B_2^*)$.
Moreover, if $B_1$, $B_2$ are separable Banach spaces and $p_1<p<p_2'$, then all operators in $\Sp({p_1},{B_1},{p_2}, {B_2^*})$ extend to bounded operators from $L^p_{B_1}$ to $L^p_{B_2}$. %

\subsection{The main result}%
For a function $f$ define  $\dil_t f(x)= f(tx)$. For an operator $T$ define the dilated operator  $\dil_{t} T$ by
 \[ \dil_t T = \dil_t\circ T\circ\dil_{t^{-1}}.\]
Note that if $T$ is given by a Schwartz kernel  $(x,y)\mapsto K(x,y)$, then the Schwartz kernel of $\dil_t T$ is given by $(x,y)\mapsto t^dK(tx, ty)$.
 
\noindent \textbf{Basic assumptions.} Let $\{T_j\}_{j \in \bbZ}$ be a family of operators in $\OPB$. We shall make the following assumptions.

\textit{Support condition.}  For all $f\in \simpleone$,
  \Be \label{support-assu} \supp\; (\dil_{2^j}T_j) f \subset \{x \in \bbR^d: \mathrm{dist}(x,\supp\,  f )\le 1\}.\Ee
 This means that if  $T_j$ is given by integration against a Schwartz kernel $K_j$, then $K_j$ lives on  a $2^j$-neighborhood of the diagonal.\hspace{.3cm}
 
\textit{Weak type $(p,p)$ condition.} For all integers $N_1 \leq N_2$, the sums $\sum_{j=N_1}^{N_2} T_j$  are of weak type $(p,p)$, with uniform bounds, 
\begin{subequations}
 \begin{equation}\label{bdness-wt}  \sup_{N_1\le N_2} \Big\|\sum_{j=N_1}^{N_2}T_j\Big\|_{L^p_{B_1}\to L^{p,\infty}_{B_2}}\le A(p).
 \end{equation}
 
 \textit{Restricted strong type $(q,q)$ condition.} For all integers $N_1 \leq N_2$, the sums $\sum_{j=N_1}^{N_2} T_j$  are of restricted strong type $(q,q)$, with uniform bounds, 
 \begin{equation}\label{bdness-rt}\sup_{N_1\le N_2} \Big\|\sum_{j=N_1}^{N_2}T_j\Big\|_{L^{q,1}_{B_1}\to L^q_{B_2}}\le A(q).
 \end{equation} 
 \end{subequations}
 
\textit{Single scale $(p,q)$ condition}. The operators $T_j$ satisfy the uniform improving bounds 
 \begin{equation}\label{p-q-rescaled}
\sup_{j \in \bbZ} \|\dil_{2^j} T_j \|_{L^p_{B_1}\to L^q_{B_2}} \le A_\circ(p,q).
\end{equation} 

\textit{ Single scale $\varepsilon$-regularity conditions}. For some $\varepsilon >0$ the operators $T_j$ and the adjoints $T_j^*$ satisfy
\begin{subequations} \label{p-q-rescaled-reg}
 \begin{align}\label{p-q-rescaled-reg-a}
&\sup_{|h|\le 1}|h|^{-\eps}\sup_{j \in \bbZ} \|(\dil_{2^j} T_j )\circ\Delta_h\|_{L^p_{B_1}\to L^q_{B_2}} \le B,
\\ \label{p-q-rescaled-reg-b}
&\sup_{|h|\le 1}|h|^{-\eps}\sup_{j \in \bbZ} \|(\dil_{2^j} T_j ^*)\circ\Delta_h\|_{L^{q'}_{B_2^*}\to L^{p'}_{B_1^*}} \le B,
\end{align} 
\end{subequations} 
where \begin{equation}\label{defDeltah}
\Delta_h f(x):=f(x+h) -f(x).
\end{equation}

 The above hypotheses %
 assume certain boundedness assumptions in Lebes\-gue or Lorentz spaces of vector-valued functions; it is then implied  that all operators $T_j$  map simple  $B_1$-valued functions to  $B_2$-valued functions which are strongly measurable with respect to Lebesgue measure. We formulate our main result  for $1<p\le q<\infty$ and refer to Appendix \ref{appendixB} for variants with $p=1$ or $q=\infty$. 
 
 \begin{thm}\label{mainthm}
 Let $1<p\le q<\infty$. Let $\{T_j\}_{j\in\bbZ}$ be a family of operators in $\mathrm{Op}_{B_1,B_2}$ such that
\begin{enumerate}
    \item[$\circ$] the support condition \eqref{support-assu} holds,
    \item[$\circ$] the weak type $(p,p)$ condition \eqref{bdness-wt} holds,
    \item[$\circ$] the restricted strong type $(q,q)$ condition \eqref{bdness-rt} holds,
    \item[$\circ$] the single scale $(p,q)$ condition \eqref{p-q-rescaled} holds,
    \item[$\circ$] the single scale $\varepsilon$-regularity conditions \eqref{p-q-rescaled-reg-a}, \eqref{p-q-rescaled-reg-b} hold.
\end{enumerate}
Define
\Be\label{eqn:cCdef}\mathcal{C}= A(p)+A(q)+A_\circ(p,q)  \log \big(2+
\tfrac{B}  {A_\circ(p,q)}\big).\Ee
Then,
for all integers $N_1, N_2$ with $N_1\le N_2$,
\Be\label{vectsparsebound}
\Big\|\sum_{j=N_1}^{N_2}T_j\Big \|_{\Sp_\ga(p,{B_1},{q'},{B_2^*})} \lesssim_{p,q,\eps,\gamma,d} \mathcal{C}.
\Ee
 \end{thm}

 The estimate \eqref{vectsparsebound} implies, via a  linearization technique  (\cf. Lemma \ref{lem:proofofcormain}) 
 the following variant 
 which leads to a sparse domination result for maximal functions, square functions and variational operators, see \Ch  \ref{sec:squarefct-etc}. Instead of $T_j~\in~ \mathrm{Op}_{B_1,B_2}$ we use the more restrictive assumption that the $T_j$  map functions in $\simpleone$ to locally integrable $B_2$-valued functions.  We let $L^1_{B_2, \loc}$  be the space of all strongly  measurable $B_2$-valued  functions which are Bochner integrable over compact sets.
 
 \begin{cor}\label{mainthm-cor}
Let $1 < p \leq q < \infty$.  Let $\{T_j\}_{j \in \bbZ}$ be a family of operators, with  $T_j:\simpleone\to L^1_{B_2,\loc}$, and satisfying the assumptions of  Theorem \ref{mainthm}. Let $\cC$ be as in \eqref{eqn:cCdef}. Then for all  $f\in \simpleone$, all $\bbR$-valued nonnegative
measurable %
functions $\omega$, and all integers $N_1, N_2$ with $N_1\le N_2$,
\Be\label{norms}
\int_{\bbR^d} \Big|\sum_{j=N_1}^{N_2} T_j f(x) \Big|_{B_2} \omega(x) dx 
\lesssim_{p,q,\varepsilon, \gamma, d} \cC \, \Lamaxga_{p,B_1,q',\bbR} (f,\om).
\Ee 
 \end{cor}

\begin{remarksa}

 (i)  We emphasize that the implicit constants in \eqref{vectsparsebound} and \eqref{norms} are dependent on the input constants in \eqref{bdness-wt}, \eqref{bdness-rt}, \eqref{p-q-rescaled}, \eqref{p-q-rescaled-reg-a}, \eqref{p-q-rescaled-reg-b} but otherwise not dependent  on the specific choices of the Banach spaces $B_1$, $B_2$.
 In some  applications this enables us to perform certain approximation arguments, where for example the Banach spaces are replaced by finite-dimensional subspaces of large dimension.

 (ii) We note that  for operators $T_j$ which commute with translations Condition \eqref{p-q-rescaled-reg-b} is implied by  Condition \eqref{p-q-rescaled-reg-a}.

 (iii) The H\"older-type regularity assumption \eqref{p-q-rescaled-reg} for the operator norm can be further weakened.
 In  applications this will often  be used  for the situation that an operator $T$ is split into a sum $\sum_{\ell\ge 0} T^{\ell}$ where each $T^{\ell}=\sum_j T^\ell_j$ satisfies the assumptions with  $A(p), A(q), A_\circ(p,q)=O(2^{-\ell\ep'})$ for some $\ep'>0$ and $B=2^{\ell M}$, for a possibly very large  $M$. The conclusion will then say that $\|T^{\ell}\|_{\Sp_\ga  (p,q')}= O(\ell 2^{-\ell\ep'} )$, which can be summed in $\ell$, leading to a sparse bound for $T$. 
 
 (iv)  In this paper we  are mainly interested in applications beyond the  Calder\'on--Zyg\-mund theory and %
 focus on the case $p>1$ and $q<\infty$.   
 There are some elements in our proof such as the property of   $L^{p,\infty}$ being  the dual space of $L^{p',1}$ for which there is no analog  for $p=1$ and similarly the failure of a  suitable notion of restricted strong type for $q=\infty$; hence Theorem \ref{mainthm}  does not immediately apply to the situations where $p=1$ or $q=\infty$. Nevertheless one can formulate variants of the theorem  which cover these missing cases.  We treat them  in Appendix 
 \ref{appendixB}; indeed they  are close  to results already covered in other works, in particular 
  \cite{conde-alonso-etal}.

 (v) The role of the simple functions is not essential in Theorem \ref{mainthm}, and the sparse bound can be extended to other classes of functions under appropriate hypotheses; see Lemma \ref{lem:density}.
 
 (vi) We use the Banach space valued formulation only to increase applicability. We emphasize  that we make no specific assumptions on the Banach spaces in our formulation of   Theorem \ref{mainthm} (such as $\mathrm{UMD}$ in the theory of Banach space valued  singular integrals).  In applications to Banach space valued singular integrals, such assumptions are  made     only because they may   be needed to verify $L^p$-boundedness  hypotheses  but they are not needed to establish  the implication in Theorem \ref{mainthm}.  %
\end{remarksa}

\subsection{Necessary conditions} %
Under the additional assumption that $T_j : \simpleone \to L^1_{B_2, \loc}$, together with $p < q$, one has that the weak type $(p,p)$ condition  \eqref{bdness-wt} and the restricted strong type condition \eqref{bdness-rt} are necessary for the conclusion of Theorem \ref{mainthm} to hold.
Moreover, if we strengthen the support condition \eqref{support-assu} assuming that the Schwartz kernels of $T_j$ are not only supported in $\{|x-y|\lc 2^j\}$ but actually in $\{|x-y|\approx 2^j\}$, then we can also show that the single scale $(p,q)$ condition \eqref{p-q-rescaled} is necessary. 

We also have an analogous statement for Corollary \ref{mainthm-cor}. Indeed, as the corollary is proved via the implication
\Be \notag \eqref{vectsparsebound}\implies \eqref{norms}, \Ee 
see Lemma \ref{lem:proofofcormain} below, we will simply formulate the necessary conditions for the conclusion in Corollary \ref{mainthm-cor}, which will also imply those in Theorem \ref{mainthm}.

To be precise in the general setting, let us formulate the following assumption on a family of operators $\{T_j\}_{j \in \bbZ}$.

{\it Strengthened support condition.}  There are $\delta_1>\delta_2>0$ such that for all $j \in \bbZ$ and all $f \in \simpleone$ \begin{multline} \label{eqn:strengthenedsupp}\supp(\dil_{2^j} T_j f) \subset \{x: \delta_1\le \dist (x,\supp f) \le 1\} , 
\\ \text{ whenever } 
\diam(\supp f) \le\delta_2. \end{multline}
If the $T_j $ are given by a Schwartz kernel $K_j$, then the condition is satisfied provided that   \[\supp(K_j) \subset \{(x,y): (\delta_1-\delta_2)  2^j\le |x-y|\le 2^j\} . \]

\begin{thm}\label{thm:necessarycor}
Suppose that $1<p<q<\infty$. Let  $\{T_j\}_{j \in \bbZ}$ be a family of operators, with 
$T_j:\simpleone\to L^1_{B_2,\loc}$, and  satisfying the support condition \eqref{support-assu}.
Assume the conclusion of Corollary \ref{mainthm-cor},  that is, there exists $\sC>0$ such that for all $N_1$, $N_2$ with $N_1\le N_2$, all  $f_1\in \mathrm{S}_{B_1}$, and all  nonnegative simple functions $\om$ 
\Be \notag \label{eqn:corollary-assumption}\int_{\bbR^d} \Big|\sum_{j=N_1}^{N_2}T_jf(x) \Big| \,\omega(x)  \,dx\le \sC \, \Lamaxga_{p,B_1,q', \bbR}(f,\om) .\Ee Then
\begin{enumerate}
\item[(i)] 
Conditions \eqref{bdness-wt} and  \eqref{bdness-rt} hold, i.e.,  there is a constant $c>0$ only depending on $d,p,q,\gamma$ such that for all $N_1, N_2$ with $N_1\le  N_2$, 
\begin{equation} \notag
 \Big\|\sum_{j=N_1}^{N_2} T_j\Big \|_{L^{p}_{B_1}\to L^{p,\infty}_{B_2}}
\le c \,\sC,  \qquad 
\Big\|\sum_{j=N_1}^{N_2} T_j\Big \|_{L^{q,1}_{B_1}\to L^q_{B_2}} \le c\,\sC\,.
\end{equation}

\item[(ii)] If, in addition, the $T_j$ satisfy 
the strengthened support condition \eqref{eqn:strengthenedsupp} 
then condition \eqref{p-q-rescaled} holds, i.e., there is a constant $c>0$ 
only depending on $d,p,q,\gamma$ such that 
\begin{equation}\notag
\sup_{j \in \bbZ} \|\dil_{2^j} T_j\|_{L^p_{B_1}\to L^q_{B_2}} \le c\, \sC.
\end{equation}
\end{enumerate} 
\end{thm}

\begin{remarksa}
(i) Note that in Theorem \ref{thm:necessarycor} there are no additional assumptions on the Banach spaces. The a priori assumption  $T_j:\simpleone\to L^1_{B_2,\loc}$ enters in the proof of necessary conditions for both Theorem \ref{mainthm} and Corollary \ref{mainthm-cor}.

(ii) There is an alternative version for necessary conditions for Theorem \ref{mainthm} where one a priori assumes merely that the $T_j$ belong to $ \mathrm{Op}_{B_1,B_2}$ (i.e. $T_j f$ is only a priori weakly integrable for $f\in\simpleone$),  but where one  imposes the assumption that  $B_2$ is reflexive. See Theorem \ref{thm:necessary} below.

(iii) We have no necessity statement regarding the  regularity conditions \eqref{p-q-rescaled-reg}  in Theorem 
\ref{mainthm}, or Corollary \ref{mainthm-cor}. However, these conditions enter in the conclusion  of both Theorem \ref{mainthm}  and Corollary \ref{mainthm-cor} only in a logarithmic way (see \eqref{eqn:cCdef}), hence  the gap between necessity and sufficiency  appears to be small. Note  that the necessary and sufficient conditions  are formulated for a uniform statement on a family of  operators $\{\sum_{j=N_1}^{N_2} T_j\}_{N_1,N_2} $ but, with the generality of our current formulation, we are unable to prove a necessary condition for sparse domination for a specific   operator in this family. %
Nevertheless,  the formulation allows us to  show necessary conditions for several specific maximal operators, variation  norm operators  and other vector-valued variants, in particular those considered in \S\ref{sec:maxfctandellr}, \S\ref{sec:variationnorms}, \S\ref{mainthmtrunc} 
and \S\ref{subsec:max op Fourier}.

(iv)  The constant $c$ in the conclusion of Theorem \ref{thm:necessarycor}
is independent of the particular pair of Banach spaces $B_1$, $B_2$. This is significant for applying the theorem to families of maximal and variational operators where for the necessity conditions one can replace the  spaces $\ell^\infty$, $L^\infty$, $V^r$ by finite-dimensional subspaces of large dimension.

(v) Since $\|T\|_{\Sp_\ga(p_1,B_1,p_2, B_2^*)} \le \|T\|_{\mathrm{sp}_\ga(p_1,B_1, B_2)}$ for $p_2\ge 1$, %
the necessary conditions in Theorem \ref{thm:necessarycor} can also be used to prove the impossibility of pointwise sparse domination for many of the operators considered in this paper. 
\end{remarksa}

\subsection{An application to maximal functions}  \label{sec:maxfct-intro}
We illustrate  Remark (iii) above with a brief discussion about 
maximal operators associated to a distribution $\sigma$ compactly supported in $\bbR^d \backslash \{0\}$  (for example a measure), for which we have necessary conditions for sparse bounds.
 Denote by $\sigma_t= t^{-d}\sigma(t^{-1}\cdot)  $  the $t$-dilate in the sense of distributions. For a  dilation set $E\subset (0,\infty)$ we consider the maximal operator 
\Be \label{eqn:MEsigma} M_E^\sigma f(x)=\sup_{t\in E} |f*\sigma_t(x)|. \Ee

The maximal function is a priori well defined as measurable function if $f$ is in the Schwartz class; alternatively we may just restrict to countable $E$ (see \S\ref{sec:MEmax} for comments why this is not a significant restriction).

For the formulation of our theorem we also need the rescaled local operators $M^\sigma_{E_j}$ with 
\begin{equation}\label{def:Ej}
E_j = (2^{-j} E)\cap [1,2]. 
\end{equation}
A  model case is given when $E$ consists of all dyadic dilates of a set in $[1.2]$, {\em i.e.}
\[E= \bigcup_{j\in \bbZ} 2^j E^\circ \quad \text{ with }  E^\circ \subset [1,2]. \]
In this case
\[M^\sigma_{E_j}= M^\sigma_{2^{-j}E \cap [1,2]}  = M^\sigma_{E^\circ} \quad \text{ for all $j\in \bbZ$.} \]

\begin{definitiona} 
The \emph{Lebesgue exponent set} 
of the pair $(\sigma, E)$, denoted by $\mathscr{L}(\sigma,E)$, consists   of all $(1/p,1/q)$ for which 
\Be \label{eqn:apriori}
\|M^\sigma_E\|_{L^p\to L^{p,\infty} } + \|M^{\sigma}_{E}\|_{L^{q,1}\to L^q} +\sup_{j\in \bbZ} \|M^\sigma_{E_j} \|_{L^p\to L^q} <\infty.
\Ee

The
{\it sparse exponent set} of $M_E$, denoted by $\Sp[M_E^\sigma] $
consists of all pairs $(1/p_1, 1/p_2)$ with $1/p_2\ge 1/p_1$ for which there is $0<\gamma<1$ and a constant $C$ such that
\[ \int_{\bbR^d} M_E^\sigma f(x)\om(x) dx \le C\,\Lamaxga_{p_1,p_2} (f,\om) 
\]
for all simple $f$ and simple nonnegative $\omega$.

Let $\eps>0$. We let $\Eann(\la)$ be the space of tempered distributions whose Fourier transform is supported in $\{\xi:\la/2<|\xi|<2\la\}$. We say that the pair $(\sigma, E)$ satisfies  an \emph{$\eps$-regularity condition}   if  there exists $C\ge 0$, and an exponent  $p_0\ge 1$ 
such that for all $\la>2$, $j\in \bbZ$,   we have
\begin{equation}  \label{eq:eps reg cond max fn}
\|M_{E_j}^\sigma f\|_{p_0}\le C \la^{-\eps}  \|f\|_{p_0} \quad   \text{  for all $f\in \cS\cap \Eann(\la)$.}
\end{equation}

\end{definitiona} 

\begin{remarka} 
The usual lacunary maximal operator   correspond to the case where 
$E=\bbZ$. 
Under this assumption, $M_E^\sigma$ satisfies an $\eps$-regularity condition for some $\eps>0$  if and only if there is an $\eps'>0$ such that 
\[\widehat \sigma(\xi)=O(|\xi|^{-\eps'} ).\]
Moreover the condition $\sup_{j\in \bbZ} \|M^\sigma_{E_j} \|_{L^p\to L^q} <\infty$  is, in this case, equivalent with the $L^p$ improving inequality
\[\|\sigma*f\|_q\lc \|f\|_p\] for all $f\in L^p$.

\end{remarka}

Denote  by $\mathrm{Int}(\Om)$ 
the interior 
of a planar set $\Om$.
 Define $\Phi:\bbR^2\to \bbR^2$ by 
\[\Phi(x,y)=(x,1-y).\]
We will show that, under the assumption of an $\eps$-regularity condition for some $\eps>0$, 
the interiors of $\mathscr{L}(\sigma,E)$ and $\Sp[M^\sigma_E ]$ are in unique correspondence under $\Phi$ (see Figure 1). That is, 
\Be\label{eqn:interiors}  \mathrm{Int}(\mathrm{Sp}[M_E^\sigma ])=\Phi(\mathrm{Int}(\mathscr{L}(\sigma,E))); \Ee
this can be deduced as a consequence of Corollary \ref{mainthm-cor} and Theorem \ref{thm:necessarycor}.
The next theorem contains a slightly more precise statement.

		

	\begin{figure}
	\begin{tikzpicture}[scale=4]
  	
 	\newcommand{\drawconvexregion}{
 		\coordinate (Q1) at (0,0); 
 		\coordinate (Q2) at (.85,.85);
 		\coordinate (Q3) at (.7, .3); 
		\coordinate (Q4a) at (.40, .080); 
 		\coordinate (Q4) at (.5, .09);
 				\coordinate (Q5) at (.024, .04);
		\filldraw [pattern=north west lines, pattern color=gray, dashed] (Q1) -- (Q2) -- (Q3) [rounded corners=20] --
 		(Q4) [rounded corners=0] -- cycle;
 		-- (Q5)  -- cycle;
 	}

	\draw (0,0) [thick,->] -- (0,1.1) node [left] {$\frac1q$};
	\draw (0,0) [thick,->] -- (1.1,0) node [below] {$\frac1p$};
	\draw [loosely dashed,opacity=.5] (0,1) -- (1,1) -- (1,0);
	\draw [loosely dashed, opacity=.5] (0,0) -- (1,1);
	\drawconvexregion
	
	\begin{scope}[xshift=50]
		\begin{scope}[yshift=28.5, yscale=-1]
			\draw (0,1) [thick,->] -- (0,-.1) node [left] {$\frac1{q'}$};
			\draw (0,1) [thick,->] -- (1.1,1) node [below] {$\frac1p$};
        	\draw [loosely dashed, opacity=.5] (0,0) -- (1,1);
			\draw [loosely dashed,opacity=.5] (0,0) -- (1,0) -- (1,1);
			\drawconvexregion
		\end{scope}
		
	\end{scope}	
	\end{tikzpicture}
	\caption{
	Example for 
	$\mathscr{L}(\sigma,E)$ (left) and $\mathrm{Sp}[M_E^\sigma]$ (right). It may occur that the closure of  $\mathscr{L}(\sigma,E)$ is not a polygonal region, see for example \cite{RoosSeeger}.}
\end{figure}
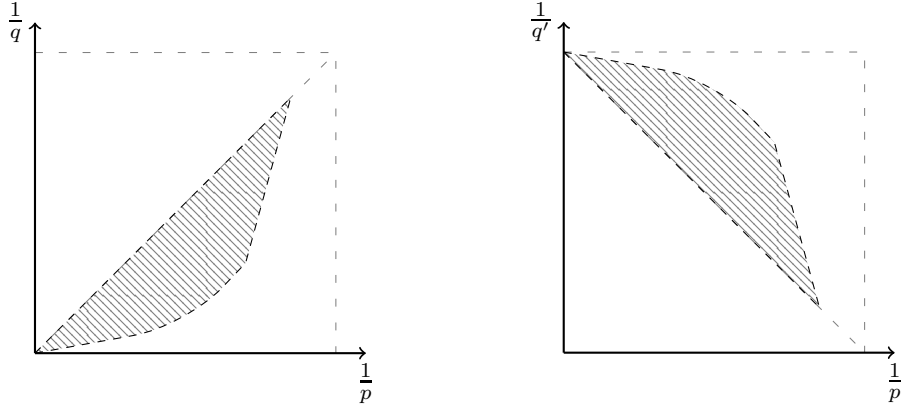

\begin{thm}\label{thm:MEintro} %
Suppose that $\sigma$ is a  compactly supported distribution supported in $\bbR^d\setminus\{0\}$, and suppose that $(\sigma,E) $  satisfies the $\eps$-regularity condition \eqref{eq:eps reg cond max fn} for some $\eps>0$.  
Let  $1< p< q<\infty .$ 
Then the  following implications hold: 
\begin{align}\label{eqn:type-sparse-implication}
(\tfrac 1p,\tfrac 1q)\in \mathrm {Int}(\mathscr{L}(\sigma,E)) 
&
\Longrightarrow 
(\tfrac 1p,\tfrac 1{q'})\in \mathrm{Sp}[M_E^\sigma ]\,,
\\
\label{eqn:sparse-type-implication}
(\tfrac 1p,\tfrac 1q)\in \mathscr{L}(\sigma,E) &
\Longleftarrow 
(\tfrac 1p,\tfrac 1{q'})\in \mathrm{Sp}[M_E^\sigma ]\,.
\end{align}
\end{thm}

\begin{remarksa}  %

(i) The correspondence \eqref{eqn:interiors} is an immediate  consequence of Theorem \ref{thm:MEintro}.

 (ii) If $\sigma$ is as in Theorem \ref{thm:MEintro} then  similar statements characterizing the sparse exponent set  hold for variation norm operators.
See  the statement of Propositions \ref{prop:VEsigma}. 

(iii) In the case of $\sigma$ being the surface measure on the unit sphere one  recovers as a special case the results  by Lacey \cite{laceyJdA19} on the lacunary and full spherical maximal functions. 
\end{remarksa}

\subsection{Fourier multipliers}
Given a bounded function $m$ we consider the convolution operator $\cT$ given on Schwartz functions $f:\bbR^d\to \bbC$ by 
\Be \label{multiplierdef}\widehat {\cT f}(\xi) =m(\xi)\widehat f(\xi),\qquad \xi\in \widehat \bbR^d,\Ee
i.e. $\cT f=\cF^{-1}[m]*f$ where  $\cF^{-1}[m]$ is the Fourier  inverse of $m$ in the sense of tempered distributions. If $1\le p<\infty$, we say that $m\in M^p$ if $\cT$ extends to a bounded operator on $L^p$ and we define $\|m\|_{M^p} $ to be the $L^p\to L^p$ operator norm of $\cT$. A similar definition applies to $p=\infty$; however one replaces $L^\infty$ by the space $C_0$ of continuous functions that vanish at $\infty$ (i.e. the closure of the Schwartz functions in the $L^\infty$ norm). By duality we have $M^p=M^{p'}$ for $1/p'=1-1/p$. Moreover, $M^2=L^\infty$, $M^p\subset L^\infty$ and $M^1$ is the space of Fourier transforms of finite Borel measures. Similarly, if $1 \leq p, q < \infty$, we say that $m \in M^{p,q}$ if $\mathcal{T}$ is bounded from $L^p$ to $L^q$ and we define by $\|m\|_{M^{p,q}}$ to be the $L^p\to L^q$ operator norm of $\mathcal{T}$. For these and other simple facts on Fourier multipliers see \cite{hormander1960} or \cite{stein-weiss}.

Let $\phi$ be a nontrivial radial $C^\infty_c$ function compactly supported in $\widehat{\bbR}^d\setminus\{0\}$. 
A natural single scale assumption would be to assume a uniform $M^{p_0}$ bound for the pieces  $\phi(t^{-1}\cdot)m$ which is equivalent by  dilation-invariance 
to the condition
\Be\label{eqn:Mp-single-scale} 
\sup_{t>0} \|\phi m(t\cdot)\|_{M^{p_0}} <\infty. \Ee
Inequality \eqref{eqn:Mp-single-scale} is  a necessary and sufficient condition for $\cT$ to be bounded  on the homogeneous  Besov spaces $\dot B^s_{p_0,q}$, for any $s\in\bbR,0<q\le \infty$; see \cite{stein-zygmund}, \cite[\S2.6]{triebel1983}.
However, it does {\it not} imply boundedness on the Lebesgue spaces, except on $L^2$. Indeed, Littman, McCarthy and Rivi\`ere \cite{lmr1968} and Stein and Zygmund \cite{stein-zygmund} give examples of $m$ satisfying 
\eqref{eqn:Mp-single-scale}  for a $p_0\neq 2$ for which $m\notin M^p$ for all $p\neq 2$.

The papers by Carbery \cite{carbery-revista} and by one of the authors \cite{See88} provide positive results under an additional dilation invariant regularity condition, 
\Be\label{eqn:Holder-single-scale} 
\sup_{t>0} \|\phi m(t\cdot)\|_{C^\eps} <\infty,\Ee
where $C^\eps$ is the standard H\"older space. Indeed, it is shown in \cite{carbery-revista,See88} that for $1<p_0<2$, $0<\eps<1$,
\Be \notag \label{eqn:localtoglobal}
\|m\|_{M^p} \le C(p,\eps) \sup_{t>0} \big( \|\phi m(t\cdot)\|_{M^{p_0}} +  \|\phi m(t\cdot)\|_{C^\eps}\big),  \quad p_0<p<p_0'.
\Ee
If the standard H\"older condition $\|\phi m(t\cdot)\|_{C^\eps}=O(1)$ is replaced by its $M^{p_0}$ variant,  $\sup_{t>0}\|\Delta_h[\phi m(t\cdot)]\|_{M^{p_0}}= O(|h|^\eps)$,  one obtains a  conclusion  for $p=p_0$. 
We will show that for fixed $p \in (p_0,p_0')$, the $L^p$-boundedness self-improves to a sparse domination inequality. 

\begin{thm}\label{thm:localtoglobal-sparse}
Let $1<p_0<2$, $0<\eps<1$, and  assume that \eqref{eqn:Mp-single-scale}  and \eqref{eqn:Holder-single-scale} hold. Then for every $p\in (p_0,2]$ there 
is a $\delta=\delta(p)>0$ such that $\cT\in \Sp(p-\delta,p'-\delta).$
\end{thm}

We note that, in view of the compact support,  for $p\le q$ the quantity $\|\phi m(t\cdot)\|_{M^{p,q}}$ can be bounded by $C\|\phi m(t\cdot)\|_{M^p}$ via Young's inequality. In Theorem \ref{thm:localtoglobal-sparse} the self-improvement to  a sparse bound is due to a tiny bit of regularity as hypothesized in \eqref{eqn:Holder-single-scale}. This together  with \eqref{eqn:Mp-single-scale} implies a mild regularity  condition for $\phi m(t\cdot)$  measured in the $M^{p,q}$ norm. If one seeks better results on the sparse bound in terms  of $q$  a further  specification of this regularity is needed. For this  we use the iterated difference operators  %
\[\Delta_h^M=\Delta_h\Delta_h^{M-1} \quad \text{ for } M\ge 2,\]
where $\Delta_h$ is as in \eqref{defDeltah}. With $\phi$ 
as above we get the following.
\begin{thm} \label{thm:localtoglobal-sparse-s} Let  $m\in L^\infty(\bbR^d)$ and $\cT$ as in \eqref{multiplierdef}.
Let    $1<p\leq q<\infty$. Assume    that there exists $s>d(1/p-1/q)$ and an  $M\in \bbN$ such that
\Be \label{eqn:MpqHolder}
\sup_{t>0} \sup_{|h|\leq 1} |h|^{-s}\big \|\Delta_h^M[\phi m(t\cdot)]\big\|_{M^{p,q}}  <\infty.
\Ee
Then $\cT \in \Sp(p,q')$.
\end{thm}
One should always take $M>s$. Indeed, note that if $M<s$, then \eqref{eqn:MpqHolder} implies $m\equiv 0$. We note that the $L^p \to L^q$ conditions \eqref{p-q-rescaled}, \eqref{p-q-rescaled-reg}
in Theorem \ref{mainthm} correspond in the instance of convolution operators to an $M^{p,q}$ condition of derivatives of order $s>d(1/p-1/q)$ on the localizations of the Fourier multiplier. 
Also, for fixed $s>d(1/p-1/q)$, if \eqref{eqn:MpqHolder} holds with some $M\geq s$, then it holds for all integers $M>s$.  
For an illustration of this and the broad scope of Theorem \ref{thm:localtoglobal-sparse-s},  see the  discussion on singular Radon transforms in   \S\ref{sec:SingRadonFM} and on various classes of  Fourier multipliers related to oscillatory multipliers in \S\ref{sec:genclassesFM} and to radial multipliers in  \S\ref{sec:rad-mult}.

Theorems \ref{thm:localtoglobal-sparse} and \ref{thm:localtoglobal-sparse-s} will be deduced  in
\S\ref{sec:proofoflocaltoglobal-sparse-s}
from the   more precise, but also more technical Theorem \ref{thm:sparsemult} which expresses the regularity via dyadic decompositions of $\cF^{-1}[\phi m(t\cdot)]$.
Moreover, there we will cover a version involving Hilbert space valued functions which is useful for sparse domination results for  objects such as Stein's square function associated with Bochner--Riesz means.

\subsection{Application to weighted norm inequalities} \label{weighted-section}
It is well known that sparse domination implies a number of weighted inequalities in the context of Muckenhoupt and reverse H\"older classes of weights, and indeed this serves as a first motivation for the subject; see the lecture notes by Pereyra \cite{PeyreyraLN} for more information.
Here we  just cite a general  result about this connection which can be directly applied to all of our results on sparse domination and is due to 
Bernicot, Frey and Petermichl \cite{bernicot-frey-petermichl}. 
 Recall the definition of the Muckenhoupt class $A_t$ consisting of weights for which \[[w]_{A_t}= \sup_Q \jp{w}_{Q,1}\jp{w^{-1}}_{Q,t'-1} <\infty,\]  
 and the definition of the reverse H\"older class $\mathrm{RH}_s$  consisting of weights for which 
 \[[w]_{\mathrm{RH}_s} =\sup_Q \frac{\jp{w}_{Q,s}}{\jp{w}_{Q,1}} <\infty.\] 
 {In both cases the supremum is taken over all cubes $Q$ in $\bbR^d$.}
 \begin{prop}[\cite{bernicot-frey-petermichl}]
  If $T\in \Sp(L^p_{B_1},L^{q'}_{B_2^*})$, then one has the weighted norm inequality
 \begin{multline*}\Big(\int_{\bbR^d} |Tf(x)|^r_{B_2} w(x) \,dx\Big)^{\frac 1r}    \lc \|T\|_{\Sp_\ga({p_1},{B_1};{p_2},{B_2^*})} \times \\ ([w]_{A_{r/p}}[w]_{\mathrm{RH}_{(q/r)'}})^{\alpha} \Big(\int_{\bbR^d} |f(x)|^r_{B_1} w(x) \,dx\Big)^{\frac1r}%
 \end{multline*}
 for all $w\in A_{r/p}\cap \mathrm{RH}_{(q/r)'}$ and $p < r < q$,  where $\alpha:=\max(\tfrac{1}{r-p}, \tfrac{q-1}{q-r})$.
 \end{prop}
 
 We refer to   \cite[\S6]{bernicot-frey-petermichl} for more information and a detailed exposition.  See also \cite{FreyNieraeth} for other weighted norm inequalities.%

\subsection{Structure of the paper and notation} %
We begin  addressing
necessary conditions, and prove Theorem \ref{thm:necessarycor} in \S\ref{sec:nec}. {In \S\ref{sec:single scale} we review useful preliminary facts needed in the proof of Theorem \ref{mainthm} regarding the single scale regularity conditions; in particular, an alternative form for the regularity conditions in \eqref{p-q-rescaled-reg}.  The proof of Theorem \ref{mainthm} is
presented in \S \ref{genthmpf}. 
The main part of the argument consists of an induction step, which is contained in \S \ref{proofofclaim}. The implication that yields Corollary \ref{mainthm-cor} from Theorem \ref{mainthm} is given in \S\ref{sec:mainthm-part-two}. 
In \S\ref{sec:squarefct-etc} we apply Corollary \ref{mainthm-cor} to deduce sparse domination results for maximal functions, square functions and variation norm operators, as well as Cotlar-type operators associated to truncations of operators. In the case of maximal functions, the assumptions of Theorem \ref{mainthm} can be slightly weakened, and we present this in \S \ref{sec:more max}. Theorems \ref{thm:localtoglobal-sparse} and \ref{thm:localtoglobal-sparse-s} are  proved in \S \ref{sec:Fourier-multipliers}. Finally, in \S \ref{sec:applications} we apply our main theorems to several specific examples, including the  proof of Theorem \ref{thm:MEintro} in \S\ref{sec:MEmax}.  Moreover,  we give several applications of  Theorem \ref{thm:localtoglobal-sparse-s} to specific classes of multipliers. For completeness, we include several appendices. Appendix \ref{basicsect} covers some basic facts  on  sparse domination. Appendix \ref{appendixB} covers versions of the main Theorem \ref{mainthm} for $p=1$ and/or $q=\infty$. 
Some basic facts on Fourier multipliers needed in  \S \ref{sec:Fourier-multipliers} 
are covered in Appendix \ref{app:multipliers}.

\medskip

\noindent{\it Notation.}  The notation $A\lesssim B$ will be used to denote that $A\le C\cdot B$, where the constant $C$ may change from line to line. Dependence of $C$ on various parameters may be denoted by a subscript or will be clear from the context. We use $A\approx B$ to denote that  $A\lc B$ and $B\lc A$.
 
 We shall use the definition
 $\widehat{f}(\xi)= \cF f(\xi)= \int_{\bbR^d} e^{-i\inn{y}{\xi} }f(y) dy$ for the  Fourier transform on $\bbR^d$. %
 We let $\cF^{-1}$ be the inverse Fourier transform and use the notation $m(D) f= \cF^{-1}[m\widehat f]$.
 We denote by $\cS\equiv \cS(\bbR^d)$ the space of Schwartz functions on $\bbR^d$,
by $\cS'$ the space of tempered distributions on $\bbR^d$, and 
by $\Eann(\la)$ the space of all $f\in \cS'$ such that the Fourier transform $\widehat f$ is supported in the open annulus $\{\xi\in \hat \bbR^d: \la/2<|\xi|<2\la\}$. 

For a $d$-dimensional rectangle $R=[a_1,b_1]\times\dots\times[a_d,b_d]$ we denote  by $x_R$ the center of $R$,  i.e. the points with coordinates $x_{R,i}=(a_i+b_i)/2$, $i=1,\dots, d$. If $\tau>0$, we denote by $\dilq{\tau}{R}$ to be the $\tau$-dilate of $R$ with respect  to its center, i.e.
\[\dilq{\tau}{R}=\Big\{ x\in\bbR^d\,:\, x_R+\frac{x-x_R}{\tau}\in R\Big\}. \]
We shall use many spatial or frequency decomposition throughout the paper:
\begin{itemize}
\item[$\circ$] $\{\lambda_k\}_{k \geq 0}$, $\{\widetilde \la_k\}_{k\ge 0}$ are specific families of Littlewood--Paley type operators that can be used for a reproducing formula \eqref{eqn:resofid}; they are compactly  supported and have  vanishing moments (\cf.  \S \ref{sec:resofid});
\item[$\circ$] $\{\Psi_\ell\}_{\ell \geq 0}$ is an inhomogeneous dyadic decomposition in $x$-space, compactly supported where $|x|\approx 2^\ell$ if $\ell>0$ (\cf. \S \ref{sec:more-precise});
\item[$\circ$] $\{\eta_\ell\}_{\ell \geq 0}$ is an inhomogeneous dyadic frequency decomposition
so that $\widehat {\eta}_\ell$ is supported where $|\xi|\approx 2^\ell$ if $\ell>0$ (\cf. \S\ref{sec:singlescalereg}, \S\ref{sec:anotherSRT}). 

\end{itemize}
\noindent Similarly, we shall use the following bump functions:
\begin{itemize}
    \item[$\circ$]  $\phi$ is a radial $C^\infty(\widehat{\bbR}^d)$ function supported in $|\xi| \approx 1$ and not identically zero (\cf. \S\ref{sec:more-precise});
    \item[$\circ$] $\theta$ is a radial $C^\infty(\bbR^d)$ function supported in $|x| \lesssim 1$ with vanishing moments and such that $\widehat{\theta}(\xi) >0$ in $|\xi| \approx 1$ (\cf.  \S\ref{sec:more-precise});
    \item[$\circ$] $\beta$ is any nontrivial $C_c^\infty(\widehat{\bbR})$ function with compact support in $(0,\infty)$ (\cf.   \S \ref{sec:rad-mult}).
\end{itemize}

\medskip

\subsubsection*{Acknowledgements} We would  like to  thank Richard Oberlin and Luz Roncal for informative conversations regarding sparse domination. We thank the referee for various  valuable suggestions and in particular for recommending to formulate  versions of our results  for the cases $p=1$ and $q=\infty$.
D.B. was partially supported by  NSF grant DMS-1954479. 
J.R.  and A.S. were  partially supported by NSF grant DMS-1764295, and A.S. 
 was partially supported by a Simons fellowship. 
We thank Tuomas Hyt\"onen for pointing out an inaccuracy in the original formula \eqref{elldef}.

\smallskip

\noindent{\it Remark.} After we circulated  the first version of this paper,  Jos\'e M. Conde-Alonso, Francesco Di Plinio, Ioannis Parissis and Manasa N. Vempati kindly shared  their  preprint \cite{CDIV},
in which they  develop a metric theory of sparse domination on spaces of homogeneous type. There is a small overlap with our work, as \cite{CDIV}  covers classes of singular Radon transforms, in particular  non-isotropic versions of  results in \S \ref{sec:SingRadon}.

\section{Necessary conditions} \label{sec:nec}

In this section we prove
Theorem \ref{thm:necessarycor} and another partial converse for Theorem \ref{mainthm}, namely Theorem \ref{thm:necessary} below.

We begin 
with an  immediate and well known, but  significant estimate for the maximal sparse forms which will lead to simple necessary conditions. In what follows, let
$\cM$ denote   the Hardy--Littlewood maximal operator.

\begin{lem}\label{standardlemma}    The following hold for the maximal forms defined in \eqref{eqn:maxform}.
\begin{enumerate}
\item[(i)] For $f_1\in \simpleone$, $f_2\in \simpledual$, 
\Be\label{maxineq} 
\Lamaxga_{{p_1},{B_1},{p_2},{B_2^*}} (f_1,f_2)  \le\ga^{-1} \int_{\bbR^d} (\cM[|f_1|_{B_1}^{p_1}](x))^{1/p_1} 
( \cM[|f_2|_{B_2^*}^{p_2}](x))^{1/p_2}  dx.
\Ee

\item[(ii)]  
If $1\le p_1<p$, and $f_1\in L^{p,1}_{B_1}$, $f_2\in L^{p'}_{B^*_2}$, then 
\Be\label{Gbypp'lor1}\Lamaxga_{{p_1},{B_1}, {p'},{B_2^*}} (f_1,f_2) 
\lc_{p,p_1} \ga^{-1} \|f_1\|_{L^{p,1}_{B_1}}\|f_2\|_{L^{p'}_{B_2^*}},
\Ee

\item[(iii)]
If $1<p<p_2'$, and $f_1\in L^p_{B_1}$, $f_2\in L^{p',1}_{B^*_2}$, then
\Be\label{Gbypp'lor2}\Lamaxga_{p,{B_1}, {p_2},{B_2^*}} (f_1,f_2) 
\lc_{p,p_2}  \ga^{-1} \|f_1\|_{L^{p}_{B_1}}\|f_2\|_{L^{p',1}_{B_2^*}},
\Ee
\end{enumerate}
\end{lem}

\begin{proof}
For a $\ga$-sparse family of cubes we have
\begin{multline*}
\La^\fS_{p_1,B_1, p_2, B_2^*}(f_1,f_2)\le \sum_{Q\in \fS}
\frac{1}{\ga}\int_{E_Q} (\cM[|f_1|_{B_1}^{p_1}](x))^{1/p_1} 
( \cM[|f_2|_{B_2^*}^{p_2}](x))^{1/p_2}  dx
\end{multline*}
and \eqref{maxineq}  follows by the disjointness of  the sets $E_Q$ and taking supremum over all sparse families.

Now let $f_1\in \simpleone$, $f_2\in \simpledual$. For \eqref{Gbypp'lor1}  we use 
 \eqref{maxineq},  with $p_2=p'$, together with the fact that for $p_1<p$ the operator 
 $g\mapsto (\cM|g|_{B_1}^{p_1} )^{1/p_1}$ maps $L^{p,1}_{B_1}$ to itself; this follows by real interpolation from the fact that it maps $L^p$ to itself, for all $p>p_1$.
  We can now  estimate
     \begin{align*} \Lamaxga_{p_1,B_1,p', B_2^*} (f_1,f_2) &\lc \ga^{-1}  
     \| (\cM[|f_1|^{p_1}_{B_1}])^{1/p_1} \|_{L^{p,1}}
     \| (\cM[|f_2|^{p'}_{B_2^*}])^{1/p'} \|_{L^{p',\infty}}
     \\
     &\lc_{p.p_1}  \ga^{-1}  \|f_1\|_{L^{p,1}_{B_1}} \|f_2\|_{L^{p'}_{B_2^*}}.
 \end{align*}
 Since simple $B_1$-valued functions are dense in  $L^{p,1}_{B_1}$ and 
 simple $B_2^*$-valued functions are dense in $L^{p'}_{B_2^*}$  we get \eqref{Gbypp'} for all $f_1\in L^{p,1}_{B_1}$ and $f_2\in L^{p'}_{B_2^*}$, by a straightforward limiting argument.
 
  For \eqref{Gbypp'lor2} we argue similarly. 
 We use  
 \eqref{maxineq} with $p_1=p$, together with the fact that for $p_2<p'$ the operator  $g\mapsto (\cM|g|_{B_2^*}^{p_2} )^{1/p_2}$ maps $L^{p',1}_{B_2^*}$ to itself, and hence
     \begin{align*}
          \Lamaxga_{p,B_1, p_2,B_2^*} (f_1,f_2)&\lc \ga^{-1}    \| (\cM[|f_1|^p_{B_1}])^{1/p} \|_{L^{p,\infty}}
     \| (\cM[|f_2|^{p_2}_{B_2^*}])^{1/p_2} \|_{L^{p',1}}
     \\
     &\lc_{p,p_2} \ga^{-1} \|f_1\|_{L^p_{B_1}} \|f_2\|_{L^{p',1}_{B_2^*}} .  \qedhere
 \end{align*}
 \end{proof}

The estimates in  Lemma {\ref{standardlemma} immediate  yield estimates for the forms
$\inn{Tf_1}{f_2}$, since by the definition \eqref{sparse norm} 
\Be \notag \label{eqn:Tf1f2}  |\inn{Tf_1}{f_2}|\le \|T\|_{\Sp_\ga(p_1,B_1;p_2,B_2^*)} \,  \La^{\ga,*}_{p_1,B_1,p_2,B_2^*}(f_1,f_2).
\Ee

We shall now prove Theorem \ref{thm:necessarycor} in \S\ref{sec:locint}, and  a variant under reflexivity of $B_2$ in \S\ref{sec:reflexivity}.

\subsection{The local integrability hypothesis} \label{sec:locint}
If $Tf_1 \in L^1_{B_2,\loc}$, Lemma \ref{standardlemma} further yields  bounds for the $L^p_{B_2}$ or $L^{p,\infty}_{B_2}$ norms of $Tf_1$  via 
 a  duality result for scalar functions. 
 
\begin{lem}\label{lem:necessitywithL1loc}
Suppose $T:\simpleone \to L^1_{B_2,\loc}$. Then the following hold.
\begin{enumerate}
\item[(i)] 
If $1 \leq p_1 < p<\infty$ and if for all $f\in\simpleone$ and all $\bbR$-valued nonnegative simple functions $\om$
\[\int_{\bbR^d} |Tf(x)|\ci {B_2} \om(x) dx \le A_1 \Lamaxga_{p_1,B_1,p',\bbR} (f,\om), \]
then
$T$ extends to a bounded operator from $L^{p,1}_{B_1}$ to $L^{p}_{B_2}$ so that 
\Be\label{rtbdness-vs-sparse}
\|T\|_{L^{p,1}_{B_1}\to L^{p}_{B_2}} \lc_{p_1,p} \gamma^{-1} A_1. %
\Ee
\item[(ii)]  If $1< p<p_2'$ and if for all $f\in \simpleone$ and all $\bbR$-valued nonnegative simple functions $\om$
\[\int_{\bbR^d} |Tf(x)|\ci{B_2} \om(x) dx \le A_2 \Lamaxga_{p,B_1,p_2,\bbR} (f,\om), \]
then 
$T$ extends to a bounded operator from $L^{p}_{B_1}$ to $L^{p,\infty}_{B_2}$ so that
\Be\label{wtbdness-vs-sparse}
\|T\|_{L^p_{B_1}\to L^{p,\infty}_{B_2}} \lc_{p_2,p} \gamma^{-1} A_2. %
\Ee
\end{enumerate}
\end{lem}

\begin{proof}
We rely on Lemma \ref{standardlemma}. For part (i)  we use
\eqref{Gbypp'lor1} to estimate, for $f_2\in \mathrm{S}_{B_2}$,
\Be \notag \label{eqn:Tf1f2(i)}        
\int_{\bbR^d} |Tf(x)|\ci {B_2} \om(x) dx 
\lc \ga^{-1}  A_1
\|f\|_{L^{p,1}_{B_1}} \|\om\|_{L^{p'}}.
\Ee
By $L^p$ duality this implies an $L^p$ bound for the locally integrable scalar function $x\mapsto |Tf(x)|_{B_2} $  and consequently 
$Tf\in L^p_{B_2}$ with
\[\|Tf\|_{L^p_{B_2}} \lc \ga^{-1}  A_1
\|f\|_{L^{p,1}_{B_1}}\] 
and \eqref{rtbdness-vs-sparse} follows.

 For part (ii) we argue similarly. 
 We use  \eqref{Gbypp'lor2} 
 to estimate
 \Be \notag \label{eqn:Tf1f2(ii) }
 \int_{\bbR^d} |Tf(x)|\ci {B_2} \om(x) dx 
     \lc \ga^{-1} A_2
     \|f\|_{L^p_{B_1}} \|\om\|_{L^{p',1}} . 
     \Ee
     By the duality $(L^{p',1})^*=L^{p,\infty} $ for scalar functions for $1 < p < \infty$ \cite{hunt} 
    we get
         \[ \|Tf\|_{L^{p,\infty} _{B_2} } \lc \ga^{-1} A_2
         \|f_1\|_{L^p_{B_1}}\] 
     and \eqref{wtbdness-vs-sparse} follows.
     \end{proof}

\begin{cor} \label{cor:necL1loc}
Assume that $T:\mathrm{S}_{B_1}\to L^1_{B_2,\loc}$ and let $1 \leq p_1 < p < p_2'$.
    If for all $f \in \mathrm{S}_{B_1}$ and all $\bbR$-valued  nonnegative simple  functions $\om$ 
    \[\int_{\bbR^d} |Tf(x) |\ci{B_2} \om(x) dx \le A \, \Lamaxga_{p_1,B,p_2,\bbR}(f,\om), \] 
    then $T$ extends to a bounded operator from $L^p_{B_1}$ to $L^p_{B_2^{}}$ so that
\Be \label{bdness-vs-sparse}\|T\|_{L^p_{B_1}\to L^p_{B_2}}\lesssim_{p,p_1,p_2}  \ga^{-1} A.
\Ee

\end{cor}

\begin{proof}
Lemma \ref{lem:necessitywithL1loc} implies $T$ maps boundedly $L^{\tilde{p}_1}_{B_1} \to L^{\tilde{p}_1,\infty}_{B_2^{}}$ and $L^{\tilde{p}_2',1}_{B_1} \to L^{\tilde{p}_2'}_{B_2^{}}$ for any $p_1 < \tilde{p}_1 < p$ and $p< \tilde{p}_2' < p_2'$; the desired $L^p_{B_1} \to L^p_{B_2^{}}$ boundedness for $p_1 < p < p_2'$ then follows by interpolation. 

Alternatively, one could deduce this result directly from \eqref{maxineq}. Arguing as in the proofs of (ii) or (iii) in Lemma \ref{standardlemma}, by the Hardy--Littlewood theorem and \eqref{maxineq} one has
\Be\label{Gbypp'}\Lamaxga_{{p_1},{B_1}, {p_2},{B_2^*}} (f_1,f_2) 
\lesssim_{p,p_1,p_2} \ga^{-1} \|f_1\|_{L^p_{B_1}}\|f_2\|_{L^{p'}_{B_2^*}}
\Ee
for $1\le p_1<p<p_2'$. %
Then one can argue as in the proof  of Lemma \ref{lem:necessitywithL1loc},  to deduce \eqref{bdness-vs-sparse} from \eqref{Gbypp'}.
\end{proof}

We next turn to the necessity of the condition \eqref{p-q-rescaled} in Corollary \ref{mainthm-cor}
and  Theorem \ref{mainthm}. In this generality, this  type of implication appears to be new in the sparse domination literature. It is inspired by the philosophy of adapting the counterexamples for $L^p\to L^q$ estimates to sparse bounds (see \textit{i.e.} the examples for spherical maximal operators in \cite{laceyJdA19}).
\begin{lem} \label{lem:pq-nec}  
Let  $\{T_j\}_{j \in \bbZ}$ be a family of operators, with
$T_j:\simpleone\to L^1_{B_2,\loc}$, and  satisfying the strengthened support condition \eqref{eqn:strengthenedsupp}. Let $1\le p<q\le \infty$
and  suppose that for all $f \in \mathrm{S}_{B_1}$ and all $\bbR$-valued nonnegative simple function $\om$, the estimate
\[ \int_{\bbR^d} |T_j f(x)|\ci{B_2} \om(x) dx \le \sC \, \Lamaxga_{p,B_1,q',\bbR}(f,\om) \] holds uniformly in  $j\in \bbZ$. 
Then \[\sup_{j \in \bbZ}\|\dil_{2^j }T_j \|_{L^p_{B_1}\to L^q_{B_2^{}}} \lc_{\gamma,d,\delta_1,\delta_2,p,q}  \sC.\]
\end{lem}

\begin{proof} Fix $j \in \bbZ$ and let $S= \dil_{2^j} T_j$. We first apply a scaling argment. Note that
 by assumption
\begin{align*} \int |S f(x)|\ci {B_2} \om(x) dx &= 
2^{-jd} \int \big|T_j[f(2^{-j} \cdot)] (x)\big|\ci{B_2} \om(2^{-j} x) dx
\\& \le \sC \,  2^{-jd} \Lamaxga_{p,B_1,q',\bbR} (f(2^{-j}\cdot), \om(2^{-j}\cdot) )
\end{align*}
If $\La_{p,B_1,q,\bbR} ^\fS$ is a sparse form  with a  $\gamma$-sparse collection of cubes  we form the collection
$\fS_j$ of dilated cubes $\{2^{-j} y: y\in Q \} $ where $Q\in \fS$. 
Then 
\[2^{-jd} \La^{\fS }_{p,B_1,q',\bbR} (f(2^{-j}\cdot), \om(2^{-j}\cdot) )
=\La^{\fS_j}_{p,B_1,q',\bbR} (f,\om) \] and therefore we get the estimate
\Be \label{eq:assumption nec normalised}
\int|S f(x)|\ci{B_2} \om(x) dx \le \sC \Lamaxga_{p,B_1,q',\bbR}(f,\om) .
\Ee

Suppose that $b$ is the smallest positive integer  such that
$$2^{-b} \le d^{-1/2}\min\{ \delta_1/2, \delta_2\}.$$
 For $\fz\in \bbZ^d$ let  
 \[Q_\fz= \{x: 2^{-b} \fz_i\le x_i< 2^{-b}(\fz_{i}+1), \,\,i=1,\dots, d \}\] and let $f_\fz=f\bbone_{Q_\fz}$. 
Let $R_\fz$ the cube of side length $3$ centered at $2^{-b}\fz$. Then $Sf_\fz$ is supported in $R_\fz$. We decompose $R_\fz$ into $3^d 2^{bd}$ 
cubes $R_{\fz,\nu}$ of side length $2^{-b}$, here $\nu\in  \cI_{\fz}$ with $\# \cI_{\fz} =3^d 2^{bd} $. 

Fix $\fz,\nu$ and a simple nonnegative function  $\om$ with $\|\om \|_{L^{q'}}\le 1$.  We first prove that for $\nu\in \cI_{\fz}$ 
\Be\label{eqn:dualityfixedz}
\int |Sf_\fz(x)|\ci{B_2} \omega (x) \bbone_{R_{\fz,\nu}}(x)  dx   \lc \sC \, \|f_\fz\|_{L^p_{B_1}}
\| \om \|_{L^{q'}}.
\Ee
In this argument we shall  not use strong measurability of $Sf_\fz$. 
By \eqref{eq:assumption nec normalised} we have
\[\int |Sf_\fz(x)|\ci{B_2} \omega (x) \bbone_{R_{\fz,\nu}}(x)  dx  
 \le \sC \, \Lamaxga_{p,B_1,q', \bbR} (f_{\fz}, \om \bbone_{R_{\fz,\nu}} ) \]
and therefore we find a sparse family $\fS_{\fz,\nu} $ such that
\Be \label{eqn:concretesparse} 
\int |Sf_\fz(x)|\ci{B_2} \omega (x) \bbone_{R_{\fz,\nu}}(x)  dx  
\le 2 \sC \sum_{Q\in \fS_{\fz,\nu}} |Q|\jp{f_{\fz}}_{Q,p} \jp{ \om \bbone_{R_{\fz,\nu}}}_{Q,q'}.  \Ee
By the strengthened support condition, \eqref{eqn:strengthenedsupp},
\Be\label{eqn:suppsep}
Sf_{\fz}\bbone_{R_{\fz,\nu}} \neq 0 
\implies \dist( Q_{\fz}, R_{\fz,\nu} ) \ge \delta_1-2^{-b}\sqrt d.\Ee
Assuming that the left-hand side is not $0$ in \eqref{eqn:concretesparse}, and in view of \eqref{eqn:suppsep}, we see that for a cube $Q\in \fS_{\fz,\nu}$ we have, recalling that $\delta_1\ge\sqrt{d}\,2^{-b+1}$,
\begin{equation}\label{eqn:lowerbdQ}\notag
\def\arraystretch{1.4}
\left.
\begin{array}{r}
      Q_\fz\cap Q\neq \emptyset \\ 
R_{\fz,\nu}\cap Q \neq  \emptyset
\end{array}\right\} \implies  
\diam(Q)\ge \delta_1-2^{-b} \sqrt d \ge  2^{-b}\sqrt d.
\end{equation}
Hence all cubes that contribute to the sum in  \eqref{eqn:concretesparse} have side length $\ge 2^{-b}$. %
Denote the cubes in $\fS_{\fz,\nu}$ with side length in $[2^\ell, 2^{\ell+1})$ by $\fS_{\fz,\nu}(\ell) $ and note that for every $\ell \ge -b$ there are at most  $C(d)$ many cubes that contribute. 
Hence we may estimate
\begin{align*}
&\sum_{Q\in \fS_{\fz,\nu}} |Q|\jp{f_{\fz}}_{Q,p} \jp{\om\bbone_{R_{\fz,\nu}}}_{Q,q'}
\\&\le \sum_{\ell\ge -b} 
\sum_{Q\in \fS_{\fz,\nu}(\ell)} |Q|^{\frac 1q-\frac 1p} \Big(\int_Q|f_{\fz}(y)|_{B_1}^p dy\Big)^{\frac 1p} \Big(\int_Q |\om(x) \bbone_{R_{\fz,\nu}}(x) |_{}^{q'} dx\Big)^{\frac{1}{q'}} 
\\ &\lc_d\sum_{\ell\ge -b} 2^{\ell d(\frac 1q-\frac 1p)} \|f_\fz\|_{L^p_{B_1}} \|\om\|_{L^{q'}_{} } \lc_{b,d,p,q} \|f_\fz\|_{L^p_{B_1}} 
\end{align*}
where we used the assumption 
$q>p$ to sum in $\ell$.
This establishes \eqref{eqn:dualityfixedz}.

By duality 
combined with \eqref{eqn:dualityfixedz} we have
\Be\label{Sffz-Lq-bd} 
\|S f_{\fz} \|_{L^q_{B_2} (R_{\fz,\nu})}  \lc 
\sup_{\substack{ \om\in {\mathrm S}_{\bbR} \\ \|\om\|_{L^{q'}_{}}\le 1}}  \int |Sf_\fz (x)|\ci{B_2} |\om(x)| \bbone_{R_{\fz,\nu }}(x) dx   \lc \sC \, \|f_\fz\|_{L^p_{B_1}}.
\Ee
Considering this for various $\nu \in \cI_\fz $ we get
\[
\|S f_{\fz} \|_{L^q_{B_2^{}}} 
\lc\sum_{\nu\in \cI_\fz}  \|S f_{\fz} \|_{L^q_{B_2^{}} (R_{\fz,\nu})} 
\lc\sum_{\nu\in \cI_\fz} \sC \, \|f_\fz\|_{L^p_{B_1}}
\lc_{d,\delta_1,\delta_2} \sC \, \|f_\fz\|_{L^p_{B_1}}.\]
Then
\begin{align}\label{eqn:fzlocalization}
\| S f\|_{L^q_{B_2^{}}}&=\Big\|\sum_{\fz \in \bbZ^d} S f_\fz\Big\|_{L^q_{B_2^{}}} \lc C_d  2^{bd} \Big(\sum_{\fz \in \bbZ^d} \|S f_\fz\|_{L^q_{B_2^{}}}^q\Big)^{1/q}
\\  \notag &\lc_{d,\delta_1,\delta_2}
\sC \, \Big(\sum_{\fz\in\bbZ^d} \|f_\fz\|_{L^p_{B_1}}^q\Big)^{1/q} \lc
\sC \, \Big(\sum_{\fz\in\bbZ^d} \|f_\fz\|_{L^p_{B_1}}^p\Big)^{1/p} \lc\sC \, \|f\|_{L^p_{B_1}}. \qedhere
\end{align}
\end{proof}
\noindent Theorem \ref{thm:necessarycor} now follows from Lemmata \ref{lem:necessitywithL1loc} and \ref{lem:pq-nec}.

\subsection{The reflexivity hypothesis}\label{sec:reflexivity}
In this section we prove a version of Theorem \ref{thm:necessarycor} where  we drop the a priori assumption on $T_j$ sending $\simpleone$ to $L^1_{B_2,\loc}$ and thus we can no longer  assume the strong $B_2$ measurability of $Tf$.  We still get a partial converse to Theorem \ref{mainthm} if we assume that the Banach space $B_2$ is reflexive.

\begin{thm}\label{thm:necessary}
Let $B_2$ be reflexive and let  $1<p<q<\infty$. Let   $\{T_j\}_{j \in \bbZ}$ be a family of operators in 
$\mathrm{Op}_{B_1,B_2}$ satisfying the support condition \eqref{support-assu}.
Assume the conclusion of Theorem \ref{mainthm},  that is, there exists $\sC>0$ such that for all $N_1$, $N_2$ with $N_1\le N_2$
and every   $f_1\in \mathrm{S}_{B_1}$, $f_2\in \mathrm S_{B_2^*} $,
\[\Big|\biginn{\sum_{j=N_1}^{N_2}T_jf_1}{f_2} \Big|\le \sC \, \Lamaxga_{p,B_1,q', B_2^*}(f_1,f_2) .\] Then 
\begin{enumerate}
\item[(i)] 
Conditions \eqref{bdness-wt} and  \eqref{bdness-rt} hold, i.e.,  there is a constant $c>0$ only depending on $d,p,q,\gamma$ such that for all $N_1, N_2$ with $N_1\le  N_2$,
\begin{equation} \notag
 \Big\|\sum_{j=N_1}^{N_2} T_j\Big \|_{L^{p}_{B_1}\to L^{p,\infty}_{B_2}}
\le c \,\sC,  \qquad 
\Big\|\sum_{j=N_1}^{N_2} T_j\Big \|_{L^{q,1}_{B_1}\to L^q_{B_2}} \le c\,\sC\,.
\end{equation}

\item[(ii)] If, in addition, the $T_j$ satisfy 
the strengthened support condition \eqref{eqn:strengthenedsupp} 
then condition \eqref{p-q-rescaled} holds, i.e., there is a constant $c>0$ 
only depending on $d,p,q,\gamma$ such that 
\begin{equation}\notag
\sup_{j \in \bbZ} \|\dil_{2^j} T_j\|_{L^p_{B_1}\to L^q_{B_2}} \le c\, \sC.
\end{equation}
\end{enumerate} 
\end{thm}

In the vector valued setting of  Theorem \ref{thm:necessary} 
we need to use a more abstract duality argument which requires some care  because of a potential  lack of strong local integrability.  We briefly discuss the issue of duality.

Let $B$ be a Banach space. Recall that for $1\le p\le \infty$, $1/p+1/p'=1$, %
the space $L^{p'}_{B^*}$ is embedded in $(L^p_{B})^*$ via the canonical isometric homomorphism. 
In the scalar case  this isometry is also  surjective when $1\le p<\infty$, and the proof of this fact relies on  the Radon--Nikodym theorem.  In the vector-valued case  the surjectivity is equivalent with the dual space $B^*$ having the  {\it Radon--Nikodym property} (RNP) with respect to Lebesgue measure
(see \cite[Chapter 1.3.b]{hytonen-etal} for the formal definition).
Thus under  this assumption we have an  identification of the dual of $L^p_B$ with $L^{p'}_{B^*}$. To summarize, 
\begin{equation}\label{hyp:RNP}  B^* \in \mathrm{  RNP}  \iff 
 (L^p_{B})^* = L^{p'}_{B^*}, \quad 1\le p<\infty.
\end{equation}
Similarly, the Radon--Nikodym property for $B^*$ also implies
\Be \notag \label{eqn:Lordual} (L^{p,r}_{B})^* = L^{p',r'}_{B^*}, \quad  1<p<\infty,  \quad 1\le r<\infty;\Ee
this is not stated in \cite{hytonen-etal} but follows by a similar  argument as in the scalar case \cite[(2.7)]{hunt}, essentially with  the exception of the application of the Radon--Nikodym property in place of the scalar Radon--Nikodym theorem.
For a detailed discussion of the Radon--Nikodym properties and its applications we  refer to  \cite[Chapter 1.3.c]{hytonen-etal}.
The class of  spaces which have the Radon--Nikodym property with respect to  all $\sigma$-finite measure spaces  includes  all {\it reflexive} spaces and also all spaces that have a {\it separable dual} (\cf.  \cite[Theorem 1.3.21]{hytonen-etal}). If $B$ is reflexive, so is $B^*$, and therefore 
 \eqref{hyp:RNP} holds for reflexive spaces $B$.

Under the assumption that the double dual $B_2^{**}$ satisfies the Radon--Nikodym property, we can show that the sparse bound implies that $Tf$ can be identified with a $B_2^{**}$ strongly measurable function. This leads to a satisfactory conclusion {under the stronger assumption} that  $B_2$ is reflexive.

\begin{lem} \label{lem:necessitywith**}
Assume that $T\in \mathrm{Op}_{B_1,B_2} $  and that  $B_2^{**}$ satisfies the Radon--Nikodym property. Then the following hold.
\begin{enumerate}
\item[(i)] 
If $1 \leq p_1 < p<\infty$ and  $T\in \Sp_\ga(p_1,B_1,p',B_2^*) $ then 
$T$ extends to a bounded operator from $L^{p,1}_{B_1}$ to $L^{p}_{B_2^{**}}$ so that
\Be\label{rtbdness-vs-sparse**}
\|T\|_{L^{p,1}_{B_1}\to L^{p}_{B_2^{**}}} \lc_{p_1} \gamma^{-1} \|T\|_{\Sp_\ga({p_1},{B_1},{p'},{B_2^*})}.
\Ee
\item[(ii)]  If $1< p<p_2'$ and  $T\in \Sp_\ga(p,B_1,p_2,B_2^*) $ then 
$T$ extends to a bounded operator from $L^{p}_{B_1}$ to $L^{p,\infty}_{B_2^{**}}$ so that
\Be\label{wtbdness-vs-sparse**}
\|T\|_{L^p_{B_1}\to L^{p,\infty}_{B_2^{**}}} \lc_{p_2} \gamma^{-1} \|T\|_{\Sp_\ga({p},{B_1},{p_2},{B_2^*})}.
\Ee
\end{enumerate}
\end{lem}

\begin{proof}

We rely, 
as in the proof of Lemma \ref{lem:necessitywithL1loc}, on Lemma \ref{standardlemma}.

For part (i), we use \eqref{Gbypp'lor1} to
obtain \[ |\inn{Tf_1}{f_2}|\lc \gamma^{-1}  \|T\|_{\Sp_\ga({p},{B_1};{p_2},{B_2^*})} \|f_1\|_{L^{p,1}_{B_1}} \|f_2\|_{L^{p'}_{B_2^*}} .
\]
This inequality establishes the form $f_2\mapsto \inn{Tf_1}{f_2}$  as a linear functional on $L^{p'}_{B_2^*}$. Since $B_2^{**}$ has the Radon--Nikodym property and thus %
$(L^{p'}_{B_2^*})^*= L^{p}_{B_2^{**}}$, 
we can identify $Tf_1$ as a member of $L^{p}_{B_2^{**}}$. 
Since  
 \[  \|T\|_{L^{p,1}_{B_1}\to L^{p}_{B_2^{**}}}
     =
     \sup_{\|f_1\|_{L^{p,1}_{B_1}}\le 1} \sup_{\|f_2\|_{L^{p'}_{B_2^*}}\le 1} |\inn{Tf_1}{f_2}|,
     \]
     we have established \eqref{rtbdness-vs-sparse**}.

    Similarly, for part (ii) we  use \eqref{Gbypp'lor2} 
 to obtain 
 \[ |\inn{Tf_1}{f_2}|\lc \gamma^{-1}  \|T\|_{\Sp_\ga({p},{B_1};{p_2},{B_2^*})} \|f\|_{L^{p}_{B_1}} \|g\|_{L^{p',1}_{B_2^*}} .
\]
 Since  $L^{p,\infty}_{B_2^{**}}$ can be identified with 
 $(L^{p',1}_{B_2^*} )^*$ we then get 
 \[  \|T\|_{L^{p}_{B_1}\to L^{p,\infty}_{B_2^{**}}}
     =
     \sup_{\|f_1\|_{L^{p}_{B_1}}\le 1} \sup_{\|f_2\|_{L^{p',1}_{B_2^{*}}}\le 1} |\inn{Tf_1}{f_2}|
     \] and obtain \eqref{wtbdness-vs-sparse**}.
\end{proof}

\begin{lem} \label{lem:pq-nec**}  %
Assume that $B_2^{**} $ satisfies the Radon--Nikodym property.
Let $\{T_j\}_{j \in \bbZ} $ be  a family of  operators in $\mathrm{Op}_{B_1.B_2} $, 
satisfying   the strengthened support condition \eqref{eqn:strengthenedsupp}. Suppose that $1\le p<q\le \infty$
and  
\[\sup_{j \in \bbZ} \|T_j\|_{\Sp_\ga(p,B_1, q', B_2^*)}\le \sC.\]
Then \[\sup_{j \in \bbZ}\|\dil_{2^j }T_j \|_{L^p_{B_1}\to L^q_{B_2^{**}}} \lc_{\gamma,d,\delta_1,\delta_2,p,q}  \sC.\]
\end{lem}
\begin{proof} 
We let 
$S=\dil_{2^j }T_j$, 
$R_{\fz}$, $R_{\fz,\nu}$, $\nu\in \cI_{\fz}$  as in the proof of Lemma \ref{lem:pq-nec}. The proof of  \eqref{eqn:dualityfixedz} can be modified just with appropriate notational changes, such as replacing expressions as one the left -hand side of  \eqref{eqn:dualityfixedz} with  \[\lambda_{f,\fz,\nu} (g):=\inn{Sf_\fz}{g\bbone_{R_{\fz,\nu}} } .\]
This leads to the inequality
\Be\label{eqn:dualityfixedz**}  |\inn{Sf_{\fz}}{g\bbone_{R_{\fz,\nu}}}| \lc \sC \|f_\fz\|_{L^p_{B_1}} \|g\|_{L^q_{B_2^*}} \Ee
in the place of \eqref{eqn:dualityfixedz}.
Inequality \eqref{eqn:dualityfixedz**} 
shows   that $\lambda_{f,\fz,\nu} $ is a continuous linear functional on the space $L^{q'}_{B_2^*} (R_{\fz,\nu} )$; recall that by assumption $1\le q'<\infty$. By the Radon--Nikodym property of $B_2^{**}$, the linear functional $\la_{f,\fz,\nu}$ is identified  with a  function $ Sf_\fz$  restricted to $R_{\fz,\nu}$, in the space  
$L^q_{B_2^{**}}   (R_{\fz,\nu}) $.  Hence we now get a variant of inequality  \eqref{Sffz-Lq-bd}, namely 
\begin{align*}
&\|S f_{\fz} \|_{L^q_{B_2^{**}}} 
\lc\sum_{\nu\in \cI_\fz}  \|S f_{\fz} \|_{L^q_{B_2^{**}} (R_{\fz,\nu})} 
\\ &\lc\sum_{\nu\in \cI_\fz} 
\sup_{\|g\|_{L^{q'}_{B_2^*}}\le 1}  |\inn{Sf_\fz}{g\bbone_{R_{\fz,\nu }}} \lc_{d,\delta_1,\delta_2} \sC \|f_\fz\|_{L^p_{B_1}}.\end{align*}
We finish  {as}  in \eqref{eqn:fzlocalization} to bound
$\| S f\|_{L^q_{B_2^{**}}}\lc\sC\|f_{L^p_{B_1}}$. 
\end{proof}

 \begin{proof}[Conclusion of the proof of Theorem  \ref{thm:necessary}]  
Since we are assuming  that $B_2$ is reflexive we have that
$B_2=B_2^{**}$ satisfies the Radon--Nikodym property. Hence now  the necessity of 
the $L^{q,1}_{B_1}\to L^q_{B_2} $ and $L^{p}_{B_1}\to L^{p,\infty}_{B_2} $ conditions 
follow from Lemma \ref{lem:necessitywith**},
and the necessity of the single scale $L^p\to L^q$ conditions 
  follows from  using   the assumption with $N_1=N_2$ and applying Lemma \ref{lem:pq-nec**}.
\end{proof}

\section{Single scale sparse domination}\label{sec:single scale}
We collect some preliminary  results which  are needed in the proof of Theorem \ref{mainthm}.
\subsection{A single scale estimate}
We state an elementary lemma which is used to establish the base  case in the induction proof of Theorem \ref{mainthm}.
 Recall that for a cube $Q$, we let $\tr{Q}$ denote the cube centered at the center of $Q$ with three times the side length of $Q$, which is also the union of $Q$ and its neighbors of the same side length.

\begin{lem}\label{single-scale-lem} 
Let $T_j \in \mathrm{Op}_{B_1,B_2}$ satisfy \eqref{support-assu} and \eqref{p-q-rescaled} for some exponents $p,q\in[1,\infty]$. Let $Q$ be a cube of side length $2^{j}$. Then for $f_1\in \simpleone$, $f_2\in \simpledual$,
\[
|\inn {T_j[f_1\bbone_Q]}{f_2}| \le 3^{d/q'}A_\circ(p,q) |Q|\jp {f_1}_{Q,p} \jp {f_2}_{\tr{Q},q'}.
\]
\end{lem}
\begin{proof} By the support property \eqref{support-assu}, 
$ T_j[f_1\bbone_Q]$ is supported in $\tr{Q}$.
By re-scaling we get from \eqref{p-q-rescaled} that \Be \notag \label{eqn:Tjrescale}\|T_j\|_{L^p\to L^q}\le 2^{-jd(1/p-1/q) }A_\circ(p,q)\Ee and thus 
\begin{align*}|\inn {T_j[f_1\bbone_Q]}{f_2}| & =
|\inn {T_j[f_1\bbone_Q]}{f_2\bbone_{\tr{Q}}}|  
\\
&\le A_\circ(p,q) 
2^{-jd(1/p-1/q)} \|f_1\bbone_Q\|_p \|f_2 \bbone_{\tr{Q}}\|_{q'} 
\\& =A_\circ(p,q)3^{d/q'} |Q| \jp{f_1}_{Q,p}\jp{f_2}_{\tr{Q},q'},
\end{align*}
as claimed.
\end{proof}
This implies a sparse bound for the single scale operators $T_j$; indeed the sparse collection is a disjoint collection of cubes.
\begin{cor} For $0<\ga\le 1$ and $1\le p,q\le\infty$,
$$\|T_j\|_{\Sp_\ga(p,q')} \le 3^{d(1/p+1/q')}  A_\circ(p,q).$$
\end{cor}
\begin{proof} 
We tile  $\bbR^d$ by a family $\fQ_j$ of dyadic cubes of side length $2^j$ and estimate
\begin{align*} 
|\inn{T_jf_1}{f_2}|&\le \sum_{Q\in \fQ_j} |\inn {T_j[f_1\bbone_Q]}{f_2}| \le A_\circ(p,q) 3^{d/q'}\sum_{Q\in \fQ_j}|Q|
\jp{f_1}_{Q,p}\jp{f_2}_{\tr{Q},q'} \\&\le A_\circ(p,q)
 3^{d(1/q'-1/p')} \sum_{Q\in \fQ_j}|\tr{Q}|
\jp{f_1}_{\tr{Q},p}\jp{f_2}_{\tr{Q},q'}.
\end{align*}
The family $\{\tr{Q}: Q\in \fQ_j\}$ can be split into $3^d$  subfamilies  consisting each of {\it disjoint} cubes of side length $3\cdot 2^j$. This implies the assertion (for every $0<\gamma\le 1$) since for every $\tr{Q}$ involved in each subfamily we can choose $E_{\tr{Q}}=\tr{Q}$.
\end{proof}

\subsection{A resolution of the identity} \label{sec:resofid}
It will be quite convenient to work with a  
 resolution of the identity  using Littlewood--Paley decompositions which are localized in space. We have 
\Be\label{eqn:resofid}\Id=\sum_{k=0}^\infty  \La_k\widetilde \La_k \Ee
which converges in the strong operator topology 
on $L^p_{B_1}(\bbR^d)$, $1\le p<\infty$.  Here 
$\La_k$, $\widetilde \La_k$ are convolution operators with convolution kernels $\la_k$, $\widetilde\la_k$
such that $\la_0\in C^\infty_c$ has support in $\{x : |x|<1/2\}$, $\int \la_0=1$, $\la_1= 2\la_0(2\cdot)-\la_0$ and $\widetilde \la_0, \widetilde\la_1\in \cS$ with $\int \tilde \la_0=1$, 
and $\int \widetilde \la_1 =0$. Moreover, for $k\ge 1$, 
$\la_k= 2^{(k-1)d} \la_1(2^{k-1} \cdot)$, 
$\widetilde \la_k= 2^{(k-1)d}\widetilde \la_1(2^{k-1} \cdot)$. 
For later applicability we may choose 
  $\la_1$, 
  so that
  \Be \notag \label{eqn:moments}\int \la_1(x) \prod_{i=1}^d x_i^{\alpha_i} \, dx =0 ,  \quad \text{ for } \sum_{i=1}^d{\alpha_i} \le 100 d \Ee 
   and the same for $\widetilde{\lambda}_1$. 

A proof of \eqref{eqn:resofid} with these specifications can be found in
 \cite[Lemma 2.1]{se-tao} (the calculation there shows that $\widetilde\la_0$, $\widetilde \la_1$ can be chosen with compact support as well).
For later use we let $P_k$ be the operator given by convolution with $2^{kd}\la_0(2^k\cdot)$ for $k \geq 0$, and also set  $P_{-1}=0$, and observe that by our construction 
\Be\label{eqn:difference} \Lambda_k=P_k-P_{k-1}, \quad \text{for } k\ge 1 .\Ee

\subsection{Single scale regularity} \label{sec:singlescalereg}
In our proof of Theorem \ref{mainthm} it will be useful to work with other versions of the regularity conditions \eqref{p-q-rescaled-reg} which are adapted to the dyadic setting. 
To formulate these, we fix a dyadic lattice of cubes $\fQ$. 
Let $\{\bbE_n\}_{n \in \bbZ}$ be the conditional expectation operators associated to the $\sigma$-algebra generated by  the subfamily $\fQ_n$ of cubes in $\fQ$ of side length in $[2^{-n}, 2^{1-n}),$ 
that is, $\bbE_n f(x) = \mathrm{av}_Q f$ for every $x\in Q$ with $Q\in \fQ_n$.
Define the martingale difference operator $\bbD_n$ by  \[\bbD_n=\bbE_{n}-\bbE_{n-1} \quad \text{ for } n \geq 1.\]
We also use the operators $\La_k$, $\widetilde \La_k$ in the  decomposition 
\eqref{eqn:resofid}.

\begin{lem}\label{lem:eps-reg}
Let $T\in\OPB$. 

(i) Let $1\le p\le q<\infty$,
$0<\vartheta<1/p$. Then 
\Be\label{eqn:reg1} \|T\bbE_0\|_{L^p_{B_1}\to L^q_{B_2}} +
\sup_{n>0} 2^{n\vartheta} \|T \bbD_{n}\|_{L^p_{B_1}\to L^q_{B_2}} 
\lc_\vartheta \sup_{k\ge 0} 2^{k\vth}\|T\Lambda_k\|_{L^p_{B_1}\to L^q_{B_2}}. 
\Ee

(ii) Let $1\le p\le q\le\infty$, $0<\vartheta<1$. Then 
\Be\label{eqn:diff-LP}
\sup_{k\ge 0} 2^{k\vth}\|T\Lambda_k\|_{L^p_{B_1}\to L^q_{B_2}} \approx_\vartheta 
\|T\|_{L^p_{B_1}\to L^q_{B_2}} +
\sup_{0<|h|<1} |h|^{-\vartheta} \|T\Delta_h \|_{L^p_{B_1}\to L^q_{B_2}}.  \Ee

(iii) Let $1< p\le q\le \infty$,
$0<\vartheta<1-1/p$. Then
  \Be   \label{eqn:reginfty}
\sup_{k\ge 0} 2^{k\vth}\|T\Lambda_k\|_{L^p_{B_1}\to L^q_{B_2}} 
\lc_\vartheta\|T\bbE_0\|_{L^p_{B_1}\to L^q_{B_2}} +
\sup_{n>0} 2^{n\vartheta} \|T \bbD_{n}\|_{L^p_{B_1}\to L^q_{B_2}}. 
\Ee
\end{lem}

An immediate consequence is the following.

\begin{cor}
 For $1<p\le q<\infty$ and $0<\vartheta<\min \{1/p, 1-1/p\}$,
 \begin{multline*} \|T\bbE_0\|_{L^p_{B_1}\to L^q_{B_2}} +
\sup_{n>0} 2^{n\vartheta} \|T \bbD_{n}\|_{L^p_{B_1}\to L^q_{B_2}} 
\approx_\vartheta\\  \|T\|_{L^p_{B_1}\to L^q_{B_2}} +
\sup_{0<|h|<1} |h|^{-\vartheta} \|T\Delta_h \|_{L^p_{B_1}\to L^q_{B_2}}.
 \end{multline*}
\end{cor}

\begin{proof}[Proof of Lemma \ref{lem:eps-reg}]  We rely on arguments  used before  in   considerations  %
of variational estimates \cite{JonesKaufmanRosenblattWierdl}, \cite{JonesSeegerWright}, of  basis properties of the Haar system in  spaces measuring smoothness \cite{GarrigosSeegerUllrich} and elsewhere.
We  use 
\[\|\La_k\|_{L^p_{B_1}\to L^p_{B_1}} =O(1),\quad \|\widetilde{\La}_k\|_{L^p_{B_1}\to L^p_{B_1}} =O(1),\quad
\|\bbE_n\|_{L^p_{B_1}\to L^p_{B_1}} =O(1)\] throughout the proof. 
Since $\vth>0$ we get from \eqref{eqn:resofid} 
\[\|T\bbE_0\|_{L^p_{B_1}\to L^q_{B_2}} 
\lc \sum_{k\ge0} \|T\Lambda_k\|_{L^p_{B_1}\to L^q_{B_2}}
\lc_\vth \sup_{k\ge 0} 2^{k\vartheta} \|T\Lambda_k\|_{L^p_{B_1}\to L^q_{B_2}}.
\]
To estimate $T\bbD_n$ we  will need 
 \Be\label{cancellation}
 \begin{aligned}
 \| \widetilde \La_k \bbD _n\|_{L^p_{B_1}\to L^p_{B_1}} \lc 
\min\{1, 2^{(k-n)/p} \},
\end{aligned}
\Ee
and only the case $k\le n$ needs a proof. A standard calculation using cancellation of $\bbD_n$ yields  \eqref{cancellation} for $p=1$ and the rest follows by interpolation.
Consequently we can estimate 
\begin{align*}
&2^{n\vartheta} \|T \bbD_n\|_{L^p_{B_1}\to L^q_{B_2}}
\le 2^{n\vartheta}\sum_{k\ge 0} 
\|T\La_k\widetilde \La_k \bbD_n\|_{L^p_{B_1} \to L^q_{B_2}}
\\&
\le 2^{n\vartheta}\sum_{k\ge 0} 
\|T\La_k\|_{L^p_{B_1} \to L^q_{B_2} } \|\widetilde \La_k\bbD_n\|_{L^p_{B_1} \to L^p_{B_1}}
\\
&
\lc \sum_{0\le k\le n} 2^{k\vartheta}\|T\La_k\|_{L^p_{B_1} \to L^q_{B_2} } 2^{-(n-k)(\frac 1p-\vartheta) }
+ \sum_{k>n} 2^{k\vartheta}\|T\La_k\|_{L^p_{B_1}  \to L^q_{B_2} }2^{-(k-n)\vartheta}\\
&
\lc \sup_{k\ge 0} 2^{k\vartheta}\|T\La_k\|_{L^p_{B_1} \to L^q_{B_2} },
\end{align*}
where we used $\vth<1/p$ for the first sum.  This proves \eqref{eqn:reg1}.

We now turn to \eqref{eqn:diff-LP} and estimate the left-hand side. By \eqref{eqn:difference} we can write
\[ \|T\La_k \|_{L^p_{B_1}\to L^q_{B_2}}\le 
\|T(\Id-P_k) \|_{L^p_{B_1}\to L^q_{B_2}} +
\|T(\Id-P_{k-1} ) \|_{L^p_{B_1}\to L^q_{B_2}}. \]
Note that, as $\int \lambda_0=1$,
\[(\Id-P_{k-1}) f(x) =\int 2^{(k-1)d} \la_0(2^{k-1} h)
\Delta_{-h} f(x) dh,\]
so 
\begin{align*}\|T(\Id-P_{k-1})f\|_{L^p_{B_1}\to L^q_{B_2}}
&= \int 2^{(k-1)d}|\la_0(2^{k-1} h)|  
\|T\Delta_{-h}\|_{L^p_{B_1}\to L^q_{B_2}}dh
\\&\lc \sup_{|h|\le 2^{-k}} \|T\Delta_{h}\|_{L^p_{B_1}\to L^q_{B_2}}
\end{align*}
and the same bound  for 
$\|T(\Id-P_{k})f\|_{L^p_{B_1}\to L^q_{B_2}}$. This establishes that the left-hand side is smaller than the right-hand side in \eqref{eqn:diff-LP}.
 
 We argue similarly for the converse inequality. %
 We  estimate
\begin{align*}\|T\|_{L^p_{B_1}\to L^q_{B_2}} 
&\le \sum_{k\ge 0} \|T\Lambda_k\|_{L^p_{B_1}\to L^q_{B_2}} 
\lc \sup_{k\ge 0} 2^{k\vartheta} \|T\Lambda_k\|_{L^p_{B_1}\to L^q_{B_2}} .\end{align*}
For the main terms 
\begin{align*}
  \|T\Delta_h\|_{L^p_{B_1}\to L^q_{B_2}}  \le \sum_{k\ge 0} \|T\La_k\|_{L^p_{B_1}\to L^q_{B_2}}\|\tLa_k\Delta_h\|_{L^p_{B_1}\to L^p_{B_1}}.
\end{align*} 
Now 
\Be\label{eqn:LaDeltah}
\|\tLa_k\Delta_h\|_{L^p_{B_1}\to L^p_{B_1}}\lc \|\la_k(\cdot+h)-\la_k(\cdot)\|_1 \lc \min\{1, 2^k|h|\}  \Ee
and therefore
\begin{align*}
|h|^{-\vartheta}\|T\Delta_h\|_{L^p_{B_1}\to L^q_{B_2}}
&\le\sum_{k\ge 0}\|T\La_k\|_{L^p_{B_1}\to L^q_{B_2}}2^{k\vartheta} (2^k|h|)^{-\vartheta} \min\{1, 2^k|h|\} 
\\ &\lc\sup_{k\ge 0} 2^{k\vartheta} \|T\La_k\|_{L^p_{B_1}\to L^q_{B_2}} 
\end{align*}
since $\sum_{{k\ge 0}} (2^{k}|h|)^{-\vartheta} \min\{1, 2^k|h|\} \lc_\vartheta 1$ if $0<\vartheta<1$. %
This completes the proof of \eqref{eqn:diff-LP}.

It remains to prove \eqref{eqn:reginfty}.
Setting $\bbD_0:=\bbE_0$ we observe  that $\Id=\sum_{n\ge 0} \bbD_n$ and $\bbD_n=\bbD_n\bbD_n$, and thus
\begin{align*}
    \|T\La_k\|_{L^p_{B_1}\to L^q_{B_2}}
    \le \sum_{n\ge 0} \|T\bbD_n\|_{L^p_{B_1}\to L^q_{B_2}}\|\bbD_n\La_k\|_{L^p_{B_1}\to L^p_{B_1}}.
\end{align*}
We use $\|\bbD_n\La_k\|_{L^p_{B_1}\to L^p_{B_1}}\lc 1$ for $n\ge k$ and 
\Be \label{eqn:DnLk}
\|\bbD_n\La_k\|_{L^p_{B_1}\to L^p_{B_1}}\lc 2^{(n-k)(1-\frac 1p)} \text{ for } n<k.
\Ee
This is clear for $p=1$ and by interpolation it suffices to show it for $p=\infty$.
Let $Q$ be a dyadic cube of side length $2^{-n+1} $. Let  $\mathrm{Ch}(Q) $ be the set of $2^d$ dyadic children of $Q$  (i.e. the dyadic sub-cubes of side length $2^{-n}$). Let 
\begin{align*} F_{Q,k}&=\{ x: \dist(x, \partial \tilde Q) \le 2^{-k} \text{ for some } \tilde Q\in \mathrm{Ch}(Q)\}.
\end{align*}
Then $|F_{Q,k-1} | \lc 2^{-n(d-1)}2^{-k} $.
Let $g_{Q,k} = f \bbone_{Q\setminus F_{Q,k}} $ and observe that by Fubini's theorem and the cancellation and support properties of $\lambda_k$
\[ \bbE_{n} \La_k g_{Q,k}(x)=0 \quad \text{ and } \quad \bbE_{n-1} \La_k g_{Q,k}(x)=0 \quad  \text{ for  } x\in Q.\]
Hence for $x\in Q$,
\begin{align*} |\bbD_n \La_k f(x)|&= |\bbD_n \La_k (f\bbone_{F_{Q,k}} )(x)| \\&\lesssim 2^{nd} \int \int |\la_k(w-y) | |f(y)| dw  \bbone_{F_{Q,k}}(y) dy
\\&\lc  2^{nd} |F_{Q,k-1}| \|f\|_\infty \lc 2^{n-k}\|f\|_\infty.
\end{align*}
This implies \eqref{eqn:DnLk} for $p=\infty$.

To finish we write 
\begin{align*}
    2^{k\vartheta} \|T\La_k\|_{L^p_{B_1}\to L^q_{B_2}}
    &\le \sum_{n\ge 0} \|T\bbD_n\|_{L^p_{B_1}\to L^q_{B_2}} 2^{k\vartheta} \min\{ 1, 2^{(n-k)(1-1/p)} \} \\&
    \lc \sup_{n \geq 0} 2^{n\vartheta} \|T\bbD_n\|_{L^p_{B_1}\to L^q_{B_2}} 
\end{align*}
where we used 
$\sum_{n\ge 0}  2^{(k-n)\vartheta} \min\{ 1, 2^{(n-k)(1-1/p)} \} \lc 1 $ provided that $0<\vartheta<1-1/p$.
This proves \eqref{eqn:reginfty}.
\end{proof}

In the proof of Theorem \ref{mainthm} we use the following Corollary.
\begin{cor}\label{cor:reg}
Let $1\le p\le q<\infty$ and $0<\vartheta<1/p$. Then 

\noindent (i) For any $n \geq 0$,
\[
\|T(\Id-\bbE_n)\|_{L^p_{B_1}\to L^q_{B_2}} \lc_{\vartheta} 2^{-n\vartheta} \sup_{0<|h|<1} |h|^{-\vartheta} \|T\Delta_h\|_{L^p_{B_1}\to L^q_{B_2}}.
\]
(ii) If $T_j$ is such that \eqref{p-q-rescaled-reg-a} holds then
\[\|T_j (\Id-\bbE_{n-j})\|_{L^p_{B_1}\to L^q_{B_2}} \lc_{\vartheta} B 2^{-n\vartheta} 2^{-jd(\frac 1p-\frac 1q)}. \]
\end{cor}

\begin{proof} We write $\Id= \bbE_n+\sum_{k=1}^\infty \bbD_{n+k}$, and thus
\begin{align*}
    \|T(\Id-\bbE_n)\|_{L^p_{B_1}\to L^q_{B_2}} & \le \sum_{k=1}^\infty \|T\bbD_{n+k}\|_{L^p_{B_1}\to L^q_{B_2}}
    \\& \lc_{\vartheta} \sum_{k=1}^\infty 2^{-(n+k)\vartheta} \sup_{0<|h|<1} |h|^{-\vartheta} \|T\Delta_h\|_{L^p_{B_1}\to L^q_{B_2}}
\end{align*}
by combining part \eqref{eqn:reg1}, \eqref{eqn:diff-LP} in the statement of  Lemma \ref{lem:eps-reg}. We sum and get the assertion. Part (ii) follows by rescaling and the hypotheses.
\end{proof}

 We finally discuss a formulation of the regularity condition which involves the Fourier support of the function and is therefore limited  to the case where  $B_1$ is a separable Hilbert space, here denoted by $\sH$. It is convenient to use a frequency decomposition %
\Be \label{eqn:resofidmod} f=\sum_{\ell\ge 0}  %
\eta_\ell* f,\Ee
with $\widehat \eta_0$ is supported in $\{\xi: |\xi|< 1\}$ such that $\widehat \eta_0(\xi)=1$ for $|\xi|\le \tfrac 34$, and with $\eta_\ell $ defined by 
$\widehat \eta_\ell(\xi)=\widehat \eta_0(2^{-\ell }\xi)-\widehat \eta_0(2^{1-\ell}\xi)$ for $\ell \geq 1$, {\em i.e.}  we have 
\Be\label{eqn:suppetaell}%
\supp(\widehat\eta_\ell)\subset \{ \xi :   2^{\ell-2}<|\xi|<  2^{\ell} \}, \quad \ell\ge 1.\Ee
Recall that $\Eann(\la)$ denotes  the space of tempered distributions whose Fourier transform is supported in $\{\xi:\la/2<|\xi|<2\la\}$.

\begin{lem} \label{lem:regbyFourier} Let $\sH$ be a separable Hilbert space and $T\in \mathrm{Op}_{\sH, B_2}$. Suppose that $T: L^p_\sH\to L^q_{B_2} $  satisfies 
\begin{equation}\label{eq:p-q hyp regbyFourier}
\|  T\|_{L^p_\sH\to L^q_{B_2}} \le A,
\end{equation}
and  for all $\la>2$ and
all  $\sH$-valued Schwartz functions $f\in \Eann(\la)$, 
\begin{equation}\label{eq:p-q hyp ann regbyFourier}
\|Tf\|_{L^q_{B_2} } \le A \la^{-\vartheta}  \|f\|_{L^p_\sH}.
\end{equation} 
Then
\[
\sup_{0<|h|<1} |h|^{-\vartheta} \|T\Delta_h \|_{L^p_{\sH}\to L^q_{B_2}} \lc_{\vartheta} A. \]
\end{lem}

\begin{proof}  %
 By the  assumptions \eqref{eq:p-q hyp regbyFourier} ($\ell=0$) and \eqref{eq:p-q hyp ann regbyFourier} ($\ell \geq 1$) we have  
 \[ \|T [\eta_\ell * \Delta_h f ]\|_{L^q_{B_2}}\le A 
 2^{(1-\ell)\vartheta } \|\eta_\ell *\Delta_h f\|_{L^p_\sH} \]
 Arguing as in \eqref{eqn:LaDeltah} we get
 \[ \|\eta_\ell* \Delta_h f\|_{L^p_\sH}  
 = \|\Delta_h\eta_\ell *f\|_{L^p_\sH}  
 \lc \min\{1, 2^\ell|h|\} \|f\|_{L^p_\sH}.\]  
 Thus using \eqref{eqn:resofidmod} we obtain 
 \[ \|T\Delta_h f\|_{L^q_{B_2} }\le  \sum_{\ell=0}^\infty  \| T[ \eta_\ell*\Delta_h  f ]\|_{L^q_{B_2}} \lc  A \sum_{\ell=0}^\infty 2^{-\ell \vartheta}  \min \{ 1, 2^{\ell}   |h| \}  \|f\|_{L^p_\sH}  
 \] and after summing in $\ell$ we arrive at $\|T\Delta_h f\|_{L^q_{B_2}} \lc_\vartheta |h|^{\vartheta} \|f\|_{L^p_\sH} $.
 \end{proof}

\section{Proof of the main result}\label{genthmpf}

\subsection{A modified version of sparse forms} We fix a dyadic lattice $\fQ$ in the sense of Lerner and Nazarov, where we assume that the side length of each cube in $\fQ$ is dyadic, {\em i.e.} of the form $2^k$ with $k\in\bbZ$. Also fix  $\gamma\in (0,1)$ and $1<p\leq q<\infty$. It will be convenient to use variants $\fG_{Q_0}\equiv \fG_{Q_0,\ga}$ of  the  maximal form  
$\Lamaxga_{ p,q'}$ 
defined in \eqref{eqn:maxform}. %
The presence of the triple cubes  in the new form allows one to exploit more effectively the support condition \eqref{support-assu}.

\begin{definition} Given a cube $Q_0\in \fQ$ 
let  
\[\fG_{Q_0}(f_1,f_2)= \sup\sum_{Q \in \fS}|Q|\jp{f_1}_{Q,p, B_1} \jp{f_2}_{\tr{Q},q',B_2^*}\] where the supremum is taken over all $\gamma$-sparse
 collections $\fS$ consisting of cubes in $\sD(Q_0)$. 
\end{definition} 

\noindent \textit{Notational convention.} From now on in this proof, the dependence on the Banach spaces $B_1$, $B_2^*$ will not be explicitly indicated, i.e. $\jp{f_1}_{Q,p}$ should be understood as $\jp{f_1}_{Q,p,B_1}$ and $\jp{f_2}_{Q,q'}$ should be understood as $\jp{f_2}_{Q,q', B_2^*}$.

The key step towards proving Theorem \ref{mainthm} is to establish a variant in which $\Lamaxga_{p, q'}$ is replaced by $\fG_{Q_0}$, that is,
\Be \label{goal la replaced}
\big| \biginn{\sum_{j=N_1}^{N_2} T_j f_1}{f_2}\big| 
\lc_{p,q,\varepsilon,d,\gamma} \cC \,  \fG_{Q_0}(f_1,f_2)
\Ee
for $f_1\in \mathrm{S}_{B_1}$, $f_2\in \mathrm{S}_{B_2^*}$ 
and a sufficiently large cube $Q_0\in \fQ$. The reader will notice that $\fG_{Q_0}$ does not define a sparse form, and we will show in \S \ref{sec:proofofthmgivenclaim} how to finish the proof of Theorem \ref{mainthm} given \eqref{goal la replaced}. 
The proof of 
\eqref{goal la replaced}  will be done by induction, which leads us to the following definition.

\begin{definition}   For $n=0,1,2,\dots$ let $U(n)$ be the smallest constant $U$ so that for all families of operators $\{T_j\}$ satisfying the assumptions of Theorem \ref{mainthm}, for all pairs  $(N_1, N_2)$ with $0\le N_2-N_1\le n$ and for all dyadic cubes $Q_0\in \fQ$ of side length $2^{N_2}$ we have 
$$\big| \biginn{\sum_{j=N_1}^{N_2} T_j f_1}{f_2}\big| 
\le U \fG_{Q_0}(f_1,f_2)
$$
whenever 
$f_1\in \mathrm{S}_{B_1}$ with $\supp(f_1)\subset Q_0$ and $f_2\in\mathrm{S}_{B_2^*}$.

\end{definition}

Thus, in order to show \eqref{goal la replaced}, it suffices to show that 
\[U(n) \lesssim_{p,q,\varepsilon,d,\gamma} \mathcal{C}\]
uniformly in $n \in \bbN_0$. This will be proven by induction on $n$. By Lemma \ref{single-scale-lem} we have the base case 
\Be\label{base-case} U(0) \le 3^{d/q'} A_\circ(p,q)\Ee 
and, more generally, 
$U(n)\le (n+1) 3^{d/q'} A_\circ(p,q)$, which shows the finiteness of the $U(n)$.  The proof then reduces to the verification of the following inductive claim.

\begin{claim} \label{claim}There is a constant
$c=c_{p,q,\eps,d,\ga}$ such that for all 
$n>0$,
\[
U(n) \le \max \{ U(n-1), \, c \,\mathcal{C} \}, \] 
with $\cC$ defined as in \eqref{eqn:cCdef}.
\end{claim}

Our  proof of the claim is an extension of the proof for sparse bounds  of the prototypical singular  Radon transforms in \cite{roberlin}, which itself builds on ideas in \cite{laceyJdA19}. It is contained in \S\ref{proofofclaim}.

\subsection{Proof of Theorem \ref{mainthm} given Claim \ref{claim}}\label{sec:proofofthmgivenclaim}

Fix $N_1\le N_2$, $f_1\in \mathrm{S}_{B_1}$,
$f_2\in \mathrm{S}_{B_2^*}$. We choose any  dyadic lattice with cubes of dyadic  side length as in the previous subsection. By \eqref{support-assu} we may choose a cube $Q_0\in \fQ$
of side length $2^{L(Q_0)}$ with $L(Q_0)\ge N_2$ such that $f_1$ is supported in $Q_0$. Then $\sum_{j=N_1}^{N_2} T_j f_1$ is supported in $\tr{Q_0}$. 
Define the operators $S_j=T_j$ when $N_1\le j\le N_2$ and $S_j=0$ otherwise. %
Then the assumptions of Theorem \ref{mainthm} apply to the family $\{S_j\}$. 
By \eqref{base-case} and Claim \ref{claim} applied to $S=\sum_{j=N_1}^{L(Q_0)} S_j=\sum_{j=N_1}^{N_2} T_j$ we obtain
\Be \notag \label{asy-sparse} |\inn{Sf_1}{f_2}|\leq c_{p,q,\varepsilon,d,\gamma}\, \cC\,\fG_{Q_0}(f_1,f_2).\Ee
In order to complete the proof of Theorem \ref{mainthm} it remains to replace $\fG_{Q_0}$ by the maximal sparse form $\Lamaxga_{p, q}$. This argument relies on  facts in dyadic analysis which we quote from  the book by Lerner and Nazarov \cite{lerner-nazarov}.

We first note that for $\ep>0$ there is a $\gamma$-sparse collection $\fS_\ep\subset\sD(Q_0)$ 
such that 
\begin{align}
\notag \big| \inn{Sf_1}{f_2}\big| 
&\le (c_{p,q,\varepsilon,d,\gamma}\,\mathcal{C}+\ep) \sum_{Q\in \fS_\ep} |Q| \jp{f_1}_{Q,p} \jp{f_2}_{\tr{Q},q'} 
\label{triple-cubes}
\\
&\le 3^{d/p-d}(c_{p,q,\varepsilon,d,\gamma}\,\mathcal{C}+\ep) \sum_{Q\in \fS_\ep} |\tr{Q}| \jp{f_1}_{\tr {Q},p} \jp{f_2}_{\tr{Q},q'}.
\end{align}

By the Three Lattice Theorem \cite[Theorem 3.1]{lerner-nazarov} 
there are  dyadic lattices $\sD^{(\nu)}$, $\nu=1,\dots, 3^d$, such that every cube in the collection 
$\tr{\fS}_\ep :=\{\tr{Q}:Q\in \fS_\ep\}$  belongs to one of the dyadic lattices $\sD^{(\nu)}$. Moreover, each  collection 
\[\fS_\ep^{(\nu)}= \tr{\fS_\ep}\cap \sD^{(\nu)}\]
is a  $3^{-d}\gamma$-sparse collection of cubes in $\sD^{(\nu)}$.
Each $\fS_\ep^{(\nu)}$  is 
a $3^d\gamma^{-1}$-Carleson family  in the sense of \cite[Definition 6.2]{lerner-nazarov}. By  \cite[Lemma 6.6]{lerner-nazarov}  we can write, for each integer $M\ge 2$,  the family $\fS_\ep^{(\nu)}$ as a  union of   $M$ sub-families $ \fS_{\ep,i}^{(\nu)}$, each of which is a 
$\widetilde M$-Carleson family, with $\widetilde M=
1+ M^{-1} (3^d\ga^{-1} -1)$. By \cite[Lemma 6.3]{lerner-nazarov} the collections $ \fS_{\ep,i}^{(\nu)} $ are $\tilde \gamma$-sparse families where 
$\tilde \ga= \widetilde{M}^{-1} = (1+ M^{-1} (3^d\ga^{-1} -1))^{-1}$.  By choosing $M$ large enough we can have $\tilde \ga>\ga$ and then, from \eqref{triple-cubes}, one has
\begin{align*}
| \inn{Sf_1}{f_2}|&\le 
3^{d/p-d }(c_{p,q,\varepsilon,d,\gamma}\,\mathcal{C}+\ep) \sum_{Q\in \fS_\ep} |\tr{Q}| \jp{f_1}_{\tr {Q},p} \jp{f_2}_{\tr{Q},q'} 
\\
&\le M
3^{d/p}(c_{p,q,\varepsilon,d,\gamma}\,\mathcal{C}+\ep) 
\sup_{\substack{i=1,\dots, M\\ \nu=1,\dots 3^d}} \sum_{R\in \fS_{\ep,i}^{(\nu)}} |R| \jp{f_1}_{ {R},p} \jp{f_2}_{R,q'} 
\\
&\le M3^{d/p}(c_{p,q,\varepsilon,d,\gamma}\mathcal{C}+\ep) \,  \Lamaxga_{p,q} (f_1,f_2)
\end{align*}
which gives the desired $\ga$-sparse bound
with $\|S\|_{\Sp_{\ga}(p,q')}\le M 3^{d/p} c_{p,q,\varepsilon,d,\gamma}\,\mathcal{C}$.

\subsection{Proof of Corollary \ref{mainthm-cor}}\label{sec:mainthm-part-two}
Corollary \ref{mainthm-cor} is a consequence of Theorem \ref{mainthm} and the following lemma, applied to $T=\sum_{j=N_1}^{N_2} T_j$.

\begin{lem} \label{lem:proofofcormain}
Let $T:\simpleone\to L^1_{B_2,\loc}$ and assume that
\[\|T\|_{\Sp_\ga(p,B_1;q',B_2^*)}  \le \sC.\]
Then we have for all $f\in \simpleone$ and all nonnegative simple $\om$ 
\Be\label{normofT}\int_{\bbR^d} |Tf(x)|_{B_2} \om(x)  dx \le  \sC \,  \Lamaxga_{p,B_1,q', \bbR}(f,\om) .\Ee
\end{lem} 
\begin{proof}
By the monotone convergence theorem we may assume that $\omega$ is a compactly supported simple function. Moreover, since $T:\simpleone\to L^1_{B_2,\loc}$ we can approximate, in the $L^1_{B_2}(K)$ norm for every compact set $K$,   the function $T f$  (for 
 $f\in \simpleone$) by simple $B_2$-valued functions.  Thus given $\ep>0$  there is  $h\in \mathrm{S}_{B_2}$ such that
 \[ \int_{\bbR^d} | Tf(x)-h(x)|\ci{B_2} \, \om(x) dx  \le \ep.\]
 Moreover, there is a compactly supported 
 $\la\in \mathrm S_{B_2^*}$ with $\max_{x\in \bbR^d}|\la(x)|_{B_2^*}\le 1$ (depending on $h$, $\omega$)   such that 
 \[ \int_{\bbR^d} |h(x)|\ci{B_2} \om(x) dx\le \ep+ \int_{\bbR^d} \inn{h(x)} {\la(x) } \, \om(x) dx,
 \]
 and we also have 
 \[\Big|\int_{\bbR^d} \inn{h(x) -Tf(x) } {\la(x) }\, \omega(x) dx \Big|\le 
 \int_{\bbR^d}|h(x)-Tf(x)|\ci{B_2} \, \om(x) dx  \le \ep.\]
 Consequently 
 \[
 \int_{\bbR^d} |T f(x)|_{B_2} \omega(x) dx  \le 3\epsilon +
 \int_{\bbR^d} \biginn{Tf(x)}{\la(x)}  \,
  \omega(x) dx .\]
 Thus  
 in order to show \eqref{normofT} it suffices to show 
 \begin{equation}\label{norms-sparse}
 \int_{\bbR^d} \biginn{T f(x)}{\la(x)} 
  \omega(x) dx 
  \le \sC \Lamaxga_{p_1,B_1,p_2,\bbR}(f,\om) 
\end{equation}
for any choice of compactly supported  $\la\in \mathrm S_{B_2^*}$ such that $\|\la\|_{L^\infty_{B_2^*}}\le 1$. Let $f_2(x)=\omega(x)\la(x)$.
Then $f_2\in \mathrm S_{B_2^*} $ with $|f_2(x)|_{B_2^*}\le \omega(x)$ for all $x\in \bbR^d$.
By  the hypothesis,  applied to $f$ and $f_2=\om\la$, 
\[\int_{\bbR^d} \biginn{T f(x)}{\om(x)\la(x)
}
 dx  \le \sC \, \Lamaxga_{p_1,B_1,p_2,B_2^*}(f,\om\la). 
\]
 Since $\langle \omega \la\rangle_{Q,{p_2},B_2^*}\le 
\langle \omega \rangle_{Q,{p_2},\bbR}$, we have established \eqref{norms-sparse}, and the proof is finished by letting $\ep\to 0$.
 \end{proof}

\subsection{The inductive step}\label{proofofclaim}

In this section we prove Claim \ref{claim}, the key ingredient in the proof of Theorem \ref{mainthm}.
Let $Q_0$ be a dyadic cube of side length 
$2^{N_2}$. Recall that $f_1$ is supported in $Q_0$ and  thus
\[\biginn { \sum_{j=N_1}^{N_2} T_j f_1}{f_2} =
\biginn { \sum_{j=N_1}^{N_2} T_j f_1}{f_2\bbone_{\tr{Q_0} }}. \] 
Hence without loss of generality we may assume that 
$f_2$ is supported in $\tr{Q_0}$.

Let $\cM$ denote the Hardy--Littlewood maximal operator  and let $\cM_pf= (\cM |f|^p)^{1/p}$. By  the  well known weak $(1,1)$ inequality for $\cM$,
\Be \notag \label{wt} \meas(\{x \in \bbR^d:\cM_pf>\la\})\le 5^d\la^{-p}\|f\|_p^p. \Ee
Define $\Omega=\Omega_1\cup \Omega_2$ where
\Be\label{Omdefs}
\begin{aligned}
\Omega_1&= \{x\in 3Q_0: \cM_p f_1(x)> \big(\tfrac{100^d}{1-\gamma}\big)^{1/p}\jp {f_1}_{Q_0,p}\},
\\
\Omega_2&= \{x\in 3Q_0: \cM_{q'} f_2(x)> \big(\tfrac{100^d}{1-\gamma}\big)^{1/q'}  \jp {f_2}_{\tr{Q_0},q'}\}.
\end{aligned}
\Ee
We then have 
$|\Omega|\le |\Omega_1|+|\Omega_2|< (1-\gamma)|Q_0|$ and if  we set %
\[E_{Q_0} = Q_0\setminus \Omega,\]
then $|E_{Q_0} |>\ga|Q_0|$. 

We perform a Whitney decomposition of $\Omega$. It is shown in  
 \cite[VI.1.2]{Ste70}
 that given any $\beta> \sqrt d$, one can  write  $\Omega $ as a union of disjoint dyadic Whitney cubes  $W\in \cW^\beta\subset \fQ$, with side length $2^{L(W)}$ and $L(W)\in \bbZ$,  so that
\[(\beta-\sqrt d) 2^{L(W)} \le \mathrm{dist}(W, \Om^\complement ) \le \beta 2^{L(W)+1}\] for $W\in \cW^\beta$.
In  \cite[VI.1.1]{Ste70} the choice of $\beta=2\sqrt d$ is made; here we need to choose $\beta$ sufficiently large and   $\beta=6\sqrt d$ will work for us. We  fix this choice and label as $\cW$ the  corresponding family of  Whitney cubes. We then have
\Be\label{wh} 5\,\diam(W) \le \mathrm{dist}(W, \Om^\complement ) \le 12\, \diam(W) \quad \text{for all } W\in \cW.\Ee
We set for $i=1,2$,
\begin{align*}
f_{i,W}&= f_i\bbone_W,\\
b_{i,W}&=(f_i-\av\ci W f_i)\bbone_W = (\Id- \bbE_{-L(W)})f_{i,W},
\end{align*}
and
\begin{align*} 
b_i&= \sum_{W \in \mathcal{W}} b_{i,W},
\\
g_i&= f_i\bbone_{\Omega^\complement}+ \sum_{W\in \cW}\av\ci W \!f_i\,\bbone_W.
\end{align*}
Then we have   the Calder\'on--Zygmund decompositions $f_i=g_i+b_i$ (using the same above family of Whitney cubes for $f_1$ and $f_2$).
For $i=1$ we add an observation, namely that
\[b_1= \sum_{\substack{W\in\cW\\W\subset Q_0}} b_{1,W}.\]
Since $f_1$ is supported in $Q_0$, this follows from the fact that
\begin{equation}\label{W in Q0}
W \cap Q_0 \neq \emptyset \quad  \Longrightarrow \quad  W \subsetneq Q_0.
\end{equation}
Indeed, if \eqref{W in Q0} fails, we must have $Q_0 \subseteq W$ as $W$ and $Q_0$ are dyadic. But then $|Q_0|\leq |W| \leq |\Omega| < (1-\gamma) |Q_0|$, which is a contradiction.

We note from \eqref{wh} and the definition of $\Omega$ that 
\begin{subequations}
\Be \label{upperboundsWh} 
\jp{f_1}_{W,p} \lc_{d,\ga}\jp{f_1}_{Q_0,p},
\qquad
\jp{f_2}_{W,q'} \lc_{d,\ga} \jp{f_2}_{\tr{Q_0},q'}
\Ee for every $W\in \cW$, as a fixed dilate of $W$ intersects $\Omega^\complement$. Indeed, 
\Be\label{upperboundsWhext} 
\jp{f_1}_{Q,p} \lc_{d,\ga}\jp{f_1}_{Q_0,p},\qquad \quad
\jp{f_2}_{Q,q'} \lc_{d,\ga} \jp{f_2}_{\tr{Q_0},q'}\Ee\end{subequations} 
for every cube $Q$ which contains a $W\in \cW$. Moreover, by the definition of $g_i$ and $\Omega$,
\Be \label{Linfty-g} \|g_1\|_{L^\infty_{B_1}} \lc_{d,\gamma}\jp{f_1}_{Q_0,p},
\qquad 
\|g_2\|_{L^\infty_{B_2^*}} \lc_{d,\gamma}\jp{f_2}_{\tr{Q_0},q'}.
\Ee 
Since $\supp(f_1)\subset Q_0$ and $\supp(f_2) \subset \tr{Q_0}$ we also get 
$\supp(g_1)\subset Q_0$ and $\supp(g_2) \subset 3{Q_0}$; here we use \eqref{W in Q0}. %
Since $\|\bbone_{Q}\|_{L^{r,1}}\lc_r |Q|^{1/r}$ for $r<\infty$,
we obtain from \eqref{Linfty-g} that for $r_1,r_2<\infty$,
\Be \label{ri-g} \|g_1\|_{L^{r_1,1}_{B_1}} \lc_{d,r_1,\gamma} |Q_0|^{1/r_1} \jp{f_1}_{Q_0,p},
\qquad 
\|g_2\|_{L^{r_2,1}_{B_2^*}} \lc_{d,r_2,\gamma} |Q_0|^{1/r_2}\jp{f_2}_{\tr{Q_0},q'}.
\Ee

For every dyadic cube $Q\in\fQ$ we have by disjointness of the $W$ 
\[\Big\|\sum_{W\subset Q}b_{1,W} \Big\|_{L^r_{B_1}}\lc 
\Big(\sum_{W\subset Q} \|f_{1,W}\|_{L^r_{B_1}}^r\Big)^{1/r}\] and thus 
\Be \label{Lrbounds-sumsb-1}
\Big\|\sum_{W\subset Q}b_{1,W} \Big\|_{L^r_{B_1}}\lc 
\Big(\int_Q|f_1(x)|_{B_1}^r dx\Big)^{1/r}.
\Ee Likewise we get for $f_2$, %
\Be \notag \label{Lrbounds-sumsb-2}
\Big\|\sum_{W\subset Q}b_{2,W} \Big\|_{L^r_{B_2^*}}\lc
\Big(\int_Q|f_2(x)|_{B_2^*}^r dx\Big)^{1/r}.
\Ee 
We now begin the proof of the induction step in Claim
\ref{claim}. Let $$S\equiv S_{N_1,N_2} = \sum_{j=N_1}^{N_2} T_j.$$ By the Calder\'on--Zygmund decomposition for $f_1$ we have
\Be \label{first splitting}
|\inn{Sf_1}{f_2}|\le |\inn{Sg_1}{f_2}|+ | \inn{Sb_{1}}{f_2} |.
\Ee
Using the $L^{q,1}_{B_1}\to L^q_{B_2}$ boundedness of $S$ (from the restricted strong type $(q,q)$ condition \eqref{bdness-rt}) and 
\eqref{ri-g} with $r_1=q<\infty$
we get 
\begin{align} \label{good-fct} 
|\inn {Sg_1}{f_2}| & \le \|Sg_1\|_{L^q_{B_2}}\|f_2\bbone_{\tr{Q_0}}\|_{L^{q'}_{B_2^*}}
\\ &\le A(q) \|g_1\|_{L^{q,1}_{B_1}} \|f_2\bbone_{\tr{Q_0}}\|_{L^{q'}_{B_2^*}}
\notag \\
&
\lc{_{d,q,\gamma}} A(q) |Q_0| \jp{f_1}_{{Q_0},p}
\jp{f_2}_{\tr{Q_0},q'}.
\notag
\end{align}
Define, for each $W\in \cW$ (recalling that the side length of $W$ is $2^{L(W)}$),
 \[S_Wf= 
 S_{N_1, L(W)}[f\bbone_W]\equiv \sum_{N_1\le j\le L(W)} T_j [f\bbone_W].\]
 We decompose the second term in \eqref{first splitting}
as in \cite{roberlin} 
and write
\[ \inn{Sb_{1}}{f_2} =I+II+III,\]
where
\begin{subequations}\label{three terms}
\begin{align} 
I&=  \biginn{\sum_{W\in \cW}S_W f_1}{f_2}\,,
\label{termI}
\\
II&= -
\biginn{\sum_{W\in \cW}S_W(\av_W[f_1]  \bbone_W)}{f_2}\,,
\label{termII}
\\
III&= \biginn{\sum_{W\in \cW}(S-S_W) b_{1,W}}{f_2}.
\label{termIII}
\end{align}
\end{subequations}

The first term \eqref{termI} is handled by the induction hypothesis. In view of \eqref{W in Q0},
each  $W$ that contributes a non-zero summand in \eqref{termI} is a proper subcube of $Q_0$. Therefore  we have 
$L(W)-N_1\le n-1$ and thus by the induction hypothesis,
\[|\inn {S_W f_1}{f_2} |\le U(n-1) \fG_{W}(f_1\bbone_W,f_2).\]
That is, given any $\ep>0$ there is a $\gamma$-sparse collection $\fS_{W,\ep}$ of subcubes of $W$ such that
\Be\label{SWest}|\inn {S_W f_1}{f_2} |\le (U(n-1)+\ep) 
\sum_{Q\in \fS_{W,\ep}} |Q| \jp {f_1}_{Q,p} \jp{f_2}_{\tr{Q},q'}.\Ee
Because of the $\gamma$-sparsity there are measurable subsets $E_Q$ of $Q$ with $|E_Q|\ge \gamma|Q|$ so that the 
$E_Q$ with $Q\in \fS_{W,\ep}$ are disjoint.
We combine the various collections $\fS_{W,\ep}$
and form the collection $\fS_\ep$ of cubes 
\Be\notag 
\fS_\ep:= \{Q_0\} \cup  \bigcup_{\substack{W\in \cW
\\W\subset Q_0}}\fS_{W,\ep}\,.
\Ee
Observe that the collection $\fS_\ep$ is indeed $\gamma$-sparse: as defined above, $E_{Q_0}=Q_0\setminus \Omega$, and therefore $|E_{Q_0}| > \gamma |Q_0|$. By disjointness of the $W \subset Q_0$ the sets $E_Q$, for $Q\in \fS_\ep$ are disjoint; moreover they satisfy $|E_Q|\ge \gamma|Q|$.

We consider the term $II$ in \eqref{termII}. Here we will use that the restricted strong type $(q,q)$ condition \eqref{bdness-rt} implies $\|S_W\|_{L^{q,1}_{B_1}\to L^q_{B_2}} \le A(q)$, and \[\|\av_W[f_1]  \bbone_W\|_{L^{q,1}_{B_1}}\lc 
|\av_W[f_1]|_{B_1}|W|^{1/q}  \lc_q \jp{f_1}_{W,p}|W|^{1/q}\] 
for $q < \infty$. Together with the disjointness of the cubes $W$ and \eqref{upperboundsWh}, we get
\begin{align} \label{IIest} |II| &\le \sum_{W\in \cW}
\|S_W (\av_W[f_1] \bbone_W)\|_{L^q_{B_2}} \|f_2\bbone_{\tr{W}}\|_{L^{q'}_{B_2^*}}
\\
&\lc_{q } \sum_{W\in\cW} \jp{f_1}_{W,p} A(q) |W|^{1/q} (3^d|W|)^{1/q'} \jp{f_2}_{\tr{W},q'}
\notag
\\
&\lc_{d,q,\gamma} A(q) \sum_{W\in \cW}|W| \jp{f_1}_{Q_0,p} 
\jp{f_2}_{\tr{Q_0},q'} 
\notag
\\
&\lc_{d,q,\gamma} A(q) |Q_0| \jp{f_1}_{Q_0,p} 
\jp{f_2}_{\tr{Q_0},q'} \,.
\notag
\end{align}

Regarding the third term in \eqref{termIII} we claim that %
\begin{multline}\label{error-claim} 
|III|\lc_{p,q,\eps,d,\ga} (A(p)+A_\circ(p,q)
\log(2+ \tfrac{B}{A_\circ(p,q)}))|Q_0| \jp{f_1}_{Q_0,p} 
\jp{f_2}_{\tr{Q_0},q'} \,.
\end{multline}
Taking  \eqref{error-claim} for granted we obtain
from \eqref{good-fct}, \eqref{IIest}, 
\eqref{error-claim} and \eqref{SWest} 
that there exist constants $C_1(d,q,\gamma)$ and $C_2(p,q,\eps,d,\gamma)$ such that
 \begin{align*}
|\inn{Sf_1}{&f_2}| \le 
C_1(d,q,\gamma) A(q) |Q_0|
 \jp{f_1}_{{Q_0},p}\jp{f_2}_{\tr{Q_0},q'}
\\
&+
C_2(p,q,\eps, d,\gamma) (A(p)+A_\circ(p,q)
\log(2+ \tfrac{B}{A_\circ(p,q)})
|Q_0| \jp{f_1}_{Q_0,p} 
\jp{f_2}_{\tr{Q_0},q'} 
\\
&+\sum_{\substack{W\in \cW \\ W \subset Q_0}} \sum_{Q\in \fS_{W,\ep}} 
(U(n-1)+\ep) 
 |Q| \jp {f_1}_{Q,p} \jp{f_2}_{\tr{Q},q'}.
 \end{align*}
 
This implies %
\begin{align*}
|\inn{Sf_1}{f_2}| &\le \max\{ U(n-1)+\ep, c_{p,q,\eps,d,\ga}\,  \cC\}
\sum_{Q\in \fS_\ep} |Q|\jp {f_1}_{Q,p} \jp{f_2}_{\tr{Q},q'}
\\ &\le \max\{ U(n-1)+\ep, c_{p,q,\eps,d,\ga}\cC\}\,\fG_{Q_0}(f_1,f_2)
\end{align*}
for all $\ep>0$.  Letting $\ep\to 0$ implies Claim \ref{claim}. We are now coming to the most technical part of the proof, the estimation of the error term $III$ in \eqref{termIII} for which we have to establish the claim \eqref{error-claim}.

\begin{proof}[Proof of \eqref{error-claim}]
We now use the Calder\'on--Zygmund decomposition for $f_2=g_2+\sum_{W\in \cW} b_{2,W}$ as   described above. We split $III=\sum_{i=1}^4 III_i$ where 
\begin{subequations}\label{eq:IIIsplitting}
\begin{align}
\label{eq:IIIone-new}
III_1&= \biginn{
\sum_{W\in \cW} S b_{1,W}}{g_2},
\\ \label{eq:IIItwo-new}
III_2&= -
\sum_{W\in \cW} \inn{S_W b_{1,W}}{g_2},
\\
III_3&= \sum_{N_1\le j\le N_2} \sum_{\substack{W\in \cW:\\ L(W)<j}}\sum_{\substack{W'\in \cW:\\ L(W')\ge j}}
\inn{T_j b_{1,W} }{b_{2,W'}},
\label{eq:IIIthree-new}
\\
III_4&= \sum_{N_1\le j\le N_2} \sum_{\substack{W\in \cW:\\ L(W)<j}}\sum_{\substack{W'\in \cW:\\ L(W')< j}} 
\inn{T_j b_{1,W} }{b_{2,W'}}.
\label{eq:IIIfour-new}
\end{align}
\end{subequations}
We use the weak type $(p,p)$ condition \eqref{bdness-wt} that $S$ maps  $L^p_{B_1}$ to
$L^{p,\infty}_{B_2}$, which is isometrically embedded in $ L^{p,\infty}_{B_2^{**}}$.
As $p>1$, we obtain  using \eqref{ri-g} for $r_2=p'<\infty$ 

\begin{align*} 
|III_1| &\le \Big\|S[ \sum_{W\in \cW} b_{1,W}]\Big\|_{L^{p,\infty}_{B_2^{**}}}
\|g_2\|_{L^{p',1}_{B_2^*}}
\\
&\lesssim_{d,p,\gamma} \|S\|_{L^p_{B_1}\to L^{p,\infty}_{B_2}} 
\Big\|\sum_{W\in \cW} b_{1,W}\Big\|_{L^p_{B_1}} 
|Q_0|^{1/p'} \jp{f_2}_{\tr{Q_0},q'}.
\end{align*} By 
\eqref{Lrbounds-sumsb-1} for $r=p$ we obtain 
\begin{align*}
|III_1|
\lc_{d,p,\gamma}  A(p) |Q_0| \jp {f_1}_{Q_0,p} 
\jp{f_2}_{\tr{Q_0},q'}.
\end{align*}

Likewise, the weak type $(p,p)$ condition \eqref{bdness-wt} implies 
$L^p_{B_1}\to L^{p,\infty}_{B_2}$ boundedness of $S_W$. Using this and $\supp (S_Wb_{1,W}) \subset \tr{W}$, \eqref{upperboundsWh},
\eqref{Linfty-g}, and $p>1$ we estimate, 
\begin{align*}
|III_2|&\le \sum_{W\in \cW}\|S_W  b_{1,W}\|_{L^{p,\infty}_{B_2^{**}}} \|g_2\bbone_{\tr{W}}\|_{L^{p',1}_{B_2^*}}
\\
&\le
A(p) 
\sum_{W\in \cW}\|b_{1,W}\|_{L^{p}_{B_1}} \|g_2\|_{L^\infty_{B_2^*}} \|\bbone_{\tr{W}}\|_{L^{p',1}_{B_2^*}}
\\
&\lc_{d,p,\gamma} A(p) \sum_{W\in \cW} |W|^{1/p} \jp{f_1}_{W,p} 
\|g_2\|_{L^\infty_{B_2^*}} |W|^{1/p'}
\\ &\lc_{d,\ga} A(p) \sum_{W\in \cW} |W| 
 \jp {f_1}_{Q_0,p} \jp{f_2}_{\tr{Q_0},q'}
\end{align*}
 and hence, by the disjointness of the cubes $W$,
\[ |III_2|\lesssim_{d,p,\gamma}  A(p) |Q_0|\jp {f_1}_{Q_0,p} 
 \jp{f_2}_{\tr{Q_0},q'}.
 \]
 
Next we estimate $III_3$ and first show that 
\begin{equation}\label{restriction on j}
\def\arraystretch{1.4}
\left.
\begin{array}{l}
      \inn{T_j b_{1,W} }{b_{2,W'}} \neq 0 \\ 
L(W)<j\le L(W')
\end{array}\right\} \quad  \implies  \quad
j\le L(W')\le L(W)+2\le j+2. 
\end{equation}
To see \eqref{restriction on j} first observe that
$T_j b_{1,W}$ is supported on a cube $R_W$ centered at $x_W$ with side length $2^{j+1}+2^{L(W)}$. Hence, if 
$\inn{T_j b_{1,W} }{b_{2,W'}}\neq 0$, 
then we get from
\eqref{wh} and the triangle inequality
\begin{align*}5\,\diam(W')&\le \mathrm{dist}(W',\Om^\complement)\le \diam(W') +\diam(R_W) +\mathrm{dist}(W, \Om^\complement)
\\&\le \diam(W') +(2^{j+1}+2^{L(W)})\sqrt d  +12 \sqrt d 2^{L(W)}.
\end{align*}
Hence since $L(W)<j\le L(W')$ we get $2^{L(W')+1}\le 13 \cdot 2^{L(W)}$ which gives \eqref{restriction on j}.
Also, with these specifications  $W\subset \tr{W'}$  if 
$\inn{T_j b_{1,W} }{b_{2,W'}}
\neq 0$.
By the single scale $(p,q)$ condition \eqref{p-q-rescaled},
\Be \label{eqn:p-qj}
\|T_j\|_{L^p\to L^q} \le 2^{-jd(1/p-1/q)} A_\circ(p,q).
\Ee
Hence using H\"older's inequality, \eqref{eqn:p-qj} and \eqref{upperboundsWh}  we get 
\begin{align*}
|III_3| &\le \sum_{N_1 \leq j \leq N_2}\sum_{\substack{W'\in \cW:\\ j\le L(W')\le j+2}}\,\, \sum_{\substack{W\in \cW:\\L(W')-2\le L(W)\le j\\ W\subset 3W'}}|\inn{T_j b_{1,W} }{b_{2,W'}}|
\\ &\le A_\circ(p,q) 
\sum_{N_1 \leq j \leq N_2} 2^{-jd(1/p-1/q)} 
\sum_{\substack{W,W'\in \cW:\,W\subset 3W'\\j\le L(W')\le j+2\\
L(W')-2\le L(W)\le j}} 
\|b_{1,W} \|_{L^p_{B_1}}\|b_{2,W'}\|_{L^{q'}_{B_2^*}}
\\
&\lc_{d,\ga}  A_\circ(p,q) \jp{f_1}_{Q_0,p} \jp {f_2}_{\tr{Q_0},q'} 
\\&\qquad\qquad \times 
\sum_{N_1 \leq j \leq N_2}
\sum_{\substack{W,W':\,W\subset 3W'\\j\le L(W')\le j+2\\
L(W')-2\le L(W)\le j}} 
2^{-jd(1/p-1/q)} |W|^{1/p} |W'|^{1-1/q}
\\
&
\lc_{d,\ga}  A_\circ(p,q)
\jp{f_1}_{Q_0,p} \jp {f_2}_{\tr{Q_0},q'} \sum_{\substack W'\in \cW} 
|W'| 
\end{align*} and thus, by disjointness of the $W'$,
\[|III_3| \lc_{d,\gamma}   A_\circ(p,q) 
\jp{f_1}_{Q_0,p} \jp {f_2}_{\tr{Q_0},q'} |Q_0|.\]

Finally, consider the term
\begin{equation}\label{eq:termIV}
III_4
= \sum_{N_1\le j\le N_2} \sum_{\substack{(W,W')\in \cW\times\cW \\ L(W)<j\\L(W')< j}} 
\inn{T_j b_{1,W} }{b_{2,W'}}.
\end{equation} Let $\eps'>0$ such that 
\Be\label{eqn:defeps'} \eps' <\min \{1/p,1/q',\eps\} \Ee and let $\ell$ be a positive integer so that
\Be\label{elldef}{\ell} < \frac{100}{\eps'} \log\big(2+\frac B{A_\circ(p,q)} \big) \le {\ell+1}.
\Ee We split 
\[\mathcal{V}_j=(-\infty,j)^2\cap \bbZ^2 = \mathcal{V}_{j,1}\cup \mathcal{V}_{j,2} \cup \mathcal{V}_{j,3}\] 
into three regions
putting %
\begin{align*} 
\cV_{j,1}= \{ (L_1,L_2)\in \cV_j:  j-\ell\le L_1<j,\,\, j-\ell\le L_2<j\},
\end{align*} %
\begin{align*} 
\cV_{j,2} = 
\{(L_1, L_2)\in \cV_j \setminus \cV_{j,1}: \, L_1\le L_2 \},\\
\cV_{j,3} = 
\{(L_1, L_2)\in \cV_j \setminus \cV_{j,1}: \, L_1> L_2 \}.
\end{align*} 
Then $III_4=\sum_{i=1}^3 IV_i$ where for $i=1,2,3$,
\begin{equation}\label{eq:IVi def}
IV_i= 
\sum_{N_1\le j\le N_2} \,\sum_{\substack {W,W' \in \cW
\\
(L(W), L(W'))\in \cV_{j,i}}}
\inn{T_j b_{1,W} }{b_{2,W'}}.
\end{equation}

Let $\fR_j$ be the collection of dyadic subcubes of $Q_0$ of side length $2^j$. 
To estimate $IV_1$ we tile $Q_0$ into such cubes  and write
 \begin{align}\label{eq:IV1} 
IV_1&=\sum_{N_1\le j\le N_2} \sum_{R\in \fR_j} \,
\biginn{ \sum_{\substack{W\subset R\\
j-\ell \le L(W)<j}}
   T_j b_{1,W} } 
{
\sum_{\substack{
j-\ell \le L(W')<j}}b_{2,W'} \bbone_{\tr{R}} }. 
\end{align}
By H\"older's inequality and the single scale $(p,q)$ condition \eqref{p-q-rescaled} (in the form of \eqref{eqn:p-qj})
we get
\begin{align*} 
|IV_1| &\le A_\circ(p,q) \sum_{N_1\le j\le N_2} 2^{-jd(1/p-1/q)}
\\
&\qquad\qquad\times\sum_{R\in \fR_j} \,
\Big\|\sum_{\substack{W\subset R\\
j-\ell \le L(W)<j}}
    b_{1,W} \Big\|_{L^p_{B_1}} \,\Big\|
\sum_{\substack{
j-\ell \le L(W')<j}}b_{2,W'} \bbone_{\tr{R}} \Big\|_{L^{q'} _{B_2^*}}
\\
&\le A_\circ(p,q) \sum_{N_1\le j\le N_2} \sum_{R\in \fR_j} |R|^{-(1/p-1/q)}\,
\\ &\qquad \qquad  \times
\Big(\sum_{\substack{W\subset R\\
j-\ell \le L(W)<j}}\|
    b_{1,W} \|_{L^p_{B_1}}^p  \Big)^{1/p}
    \Big( 
\sum_{\substack{W'\subset\tr{ R}\\
j-\ell \le L(W')<j}}\|b_{2,W'}\|_{L^{q'}_{B_2^*}}^{q'} \Big)^{1/q'}
\end{align*}
and using \eqref{upperboundsWh}  this expression is bounded by  $C_{d,\gamma} A_\circ(p,q) $  times 
\begin{align*}
& \sum_{N_1\le j\le N_2} 
\jp{f_1}_{Q_0,p} \jp{f_2}_{\tr{Q_0},q'} 
\sum_{R\in \fR_j} |R|^{-(1/p-1/q)}\,
\\ &\qquad \qquad  \times
\Big(\sum_{\substack{W\subset R \\
j-\ell \le L(W)<j}}|W| \Big)^{1/p}
    \Big( 
\sum_{\substack{W'\subset\tr{ R}\\
j-\ell \le L(W')<j}}|W'| \Big)^{1/q'}
\\
&
\lc_{d,\ga}  \sum_{N_1\le j\le N_2} 
\jp{f_1}_{Q_0,p} \jp{f_2}_{\tr{Q_0},q'} 
\\&\qquad \qquad  \times
 \sum_{R\in \fR_j} |R|^{-(1/p-1/q)}
 \Big(\sum_{\substack{W\subset \tr{R} \\
j-\ell \le L(W)<j}}|W| \Big)^{1/p+1-1/q}\,.
\end{align*}
Using $p \leq q$ and the disjointness of the $W$  we see that  the last expression is dominated by a constant $\tilde C_{d,\ga,p,q}$ times %
\begin{align*} 
&
\jp{f_1}_{Q_0,p} \jp{f_2}_{\tr{Q_0},q'} 
 \sum_{N_1\le j\le N_2} \sum_{R\in \fR_j} 
 \sum_{\substack{W\subset \tr{R} \\
j-\ell \le L(W)<j}}|W| 
\\
&\le 3^d 
\jp{f_1}_{Q_0,p} \jp{f_2}_{\tr{Q_0},q'} 
 \sum_{N_1\le j\le N_2} \sum_{R\in \fR_j} 
 \sum_{\substack{W\subset{R} \\
j-\ell \le L(W)<j}}|W| 
\\
&\lc_d
\jp{f_1}_{Q_0,p} \jp{f_2}_{\tr{Q_0},q'} 
\sum_{W\in \cW}|W|
\sum_{\substack{j: N_1\le j\le N_2\\ 
L(W)<j\le L(W)+\ell}}1
\\
 &\lc_d  \ell |Q_0|
\jp{f_1}_{Q_0,p} \jp{f_2}_{\tr{Q_0},q'} .
\end{align*}
Thus, using the definition of $\ell$ in \eqref{elldef} we get
\Be\label{IV1bound}
|IV_1| \lc_{d,\ga,\eps,p,q} A_\circ(p,q) \log \big(2+ \frac{B}{A_\circ(p,q)}\big)|Q_0|
\jp{f_1}_{Q_0,p} \jp{f_2}_{\tr{Q_0},q'} .
\Ee

We now turn to the terms $IV_2$, $IV_3$ and claim that
\Be\label{IV23bound}
|IV_2|+|IV_3| \lc_{d,\ga,p,q,\eps} A_\circ(p,q) |Q_0|
\jp{f_1}_{Q_0,p} \jp{f_2}_{\tr{Q_0},q'} .
\Ee

We first note that by the single scale $\varepsilon$-regularity conditions \eqref{p-q-rescaled-reg-a}, \eqref{p-q-rescaled-reg-b}, and Corollary \ref{cor:reg},
\begin{subequations}
\begin{align} \label{Tjrega}
\|T_j(\mathrm I -\bbE_{s_1-j})\|_{L^p_{B_1} \to L^q_{B_2}} &\lc_{\varepsilon} B 2^{-jd(\frac 1p-\frac 1q)} 2^{-\eps's_1} , \\
\label{Tjregb}
\|T_j^*(\mathrm I -\bbE_{s_2-j})\|_{L^{q'}_{B_2^*} \to L^{p'}_{B_1^*}} &\lc_{\varepsilon} B
2^{-jd(\frac 1p-\frac 1q)} 2^{-\eps' s_2}.
\end{align}
\end{subequations}
where $\eps'$ is as in \eqref{eqn:defeps'}.

Write, with $\fR_j$ as in \eqref{eq:IV1},
\begin{multline} \label{eq:IV2}
IV_2=\sum_{N_1\le j\le N_2} \sum_{R\in \fR_j} \,
\sum_{s_2=1}^\infty \sum_{s_1=\max\{s_2,\ell+1\}}^\infty
\\ \biginn{ \sum_{\substack{W\subset R\\
 L(W)=j-s_1}}
   T_j b_{1,W} } 
{
\sum_{\substack{W'\subset 3R\\
L(W')=j-s_2}}b_{2,W'} \bbone_{\tr{R}} }. 
\end{multline}
Note that for $L(W)=j-s_1$, $b_{1,W} =(I-\bbE_{s_1-j}) f_{1,W}$.
By H\"older's inequality and \eqref{Tjrega}  we have for $R\in \fR_j$,
\begin{align*}
&\Big|\biginn{ \sum_{\substack{W\subset R\\
 L(W)=j-s_1}}
   T_j b_{1,W} } 
{
\sum_{\substack{W'\subset\tr{R}\\
L(W')=j-s_2}}b_{2,W'} \bbone_{\tr{R}} } \Big|
\\
&\le\Big\| T_j (I-\bbE_{s_1-j}) \sum_{\substack{W\subset R\\
 L(W)=j-s_1}} f_{1,W} \Big\|_{L^q_{B_2}}
 \Big\|
\sum_{\substack{ W'\subset\tr{R}\\
L(W')=j-s_2}}b_{2,W'} \Big \|_{L^{q'}_{B_2^*}}
\\
&\lc_{\varepsilon}
B 2^{-\eps' s_1}
|R|^{-(\frac 1p-\frac 1q)}
\Big(
\sum_{\substack{W\subset R\\
 L(W)=j-s_1}}\| f_{1,W} \|_{L^p_{B_1}}^p\Big)^{1/p} 
 \Big(
\sum_{\substack{ W'\subset\tr{R}\\
L(W')=j-s_2}}\|b_{2,W'}\|_{L^{q'}_{B_2^*}}^{q'}\Big)^{1/q'}.
  \end{align*}
  In the above formula for $IV_2$ we interchange the $j$-sum and the $(s_1,s_2)$-sums, write $j=s_1n+i$ with $i=1,\dots, s_1$ and estimate 
  (invoking \eqref{upperboundsWh} again) 
  \begin{align*} 
  |IV_2|&\lc_{\varepsilon} \sum_{s_2=1}^\infty
  \sum_{s_1=
  \max\{s_2,\ell+1\}}^\infty B 2^{-\eps' s_1}\sum_{i=1}^{s_1} 
  \sum_{\substack{n\in \bbZ\\ s_1n+i\in [N_1,N_2]}}
  \sum_{R\in \fR_{s_1n+i}}|R|^{-(\frac 1p-\frac 1q)}
  \\ 
  &\qquad \times
    \Big(
\sum_{\substack{W\subset R\\
 L(W)=s_1n+i-s_1}}\|f_{1,W}\|_{L^p_{B_1}}^p\Big)^{\frac 1p} 
 \Big(
\sum_{\substack{ W'\subset\tr{R}\\
L(W')=s_1n+i -s_2}}\|b_{2,W'}\|_{L^{q'}_{B_2^*}}^{q'}\Big)^{\frac{1}{q'}}    \\
  &\lc_{d,\ga}
  \sum_{s_2=1}^\infty\sum_{s_1=
  \max\{s_2,\ell+1\}}^\infty  B 2^{-\eps' s_1}\,
  \\ & \qquad  \times
  \sum_{i=1}^{s_1}\sum_{\substack{n\in \bbZ
  \\ s_1n+i\in [N_1,N_2]}}
    \sum_{R\in \fR_{s_1n+i}} \jp{f_1}_{Q_0,p} 
    \jp{f_2}_{\tr{Q_0},q'} \Gamma(R,n,i),
  \end{align*}
where
  \[\Gamma(R,n,i) =
  |R|^{-(\frac 1p-\frac 1q)}
  \Big(
\sum_{\substack{W\subset R\\
 L(W)=s_1n+i-s_1}}|W|\Big)^{\frac 1p} 
 \Big(
\sum_{\substack{ W'\subset\tr{R}\\
L(W')=s_1n+i -s_2}}|W'|  \Big)^{1-\frac 1q}.
\]
We crudely estimate, using $p \leq q$,
\begin{align*} 
\Gamma(R,n,i) &\le
  |R|^{-(1/p-1/q)}
  \Big(
\sum_{\nu=s_1n+i-s_1}^{s_1n+i-s_2}\sum_{
\substack{W\subset \tr{R}\\
 L(W)=\nu}}|W|\Big)^{1/p+1-1/q} 
\\
&\le 3^{d(1/p-1/q)}
\sum_{\nu=s_1n+i-s_1}^{s_1n+i-s_2}\sum_{
\substack{W\subset \tr{R}\\
 L(W)=\nu}}|W|.
\end{align*}
For fixed $W\in \cW$ consider the set  of all triples $(R,n,i)$ such that $s_1n+i-s_1\le L(W)\le s_1n+i-s_2$, $R\in \fR_{s_1n+i}$ and $W\subset \tr{R}$, and observe that the cardinality of this set is bounded above by
$3^d(s_1-s_2+1)$. Combining this with the above estimates and summing over $W\in \cW$ we obtain the bound
\begin{align*} 
|IV_2|&\lc_{d,\ga,\varepsilon} 
\jp{f_1}_{Q_0,p} 
    \jp{f_2}_{\tr{Q_0},q'} |Q_0| 
    \sum_{s_2=1}^\infty\sum_{s_1=
  \max\{s_2,\ell+1\}}^\infty  B 2^{-\eps' s_1} (s_1-s_2+1)
  \end{align*}
and the double sum is bounded by
\begin{multline*}C_{\eps,p,q} \Big(\sum_{s_2=1}^{\ell+1} 
B 2^{-\eps' \ell} 
(\ell+1)+ \sum_{s_2=\ell+1}^\infty  B 2^{-\eps'  s_2}  \Big)\\ \lc_{\eps,p,q} B 2^{-\eps'\ell} (\ell+1)^2 
\lc_{\eps,p,q} B 2^{-\eps'\ell/2} 
\lc_{\eps ,p,q} A_\circ(p,q)
\end{multline*}
by the definition of $\ell$ in \eqref{elldef}.
This establishes \eqref{IV23bound} for the term $|IV_2|$.

The estimation of $IV_3$ is very similar.  We may write
\begin{multline}\label{eq:IV3} 
IV_3=\sum_{N_1\le j\le N_2} \sum_{R\in \fR_j} \,
\sum_{s_1=1}^\infty \sum_{s_2=\max\{s_1+1,\ell+1\}}^\infty
\\ \biginn{ \sum_{\substack{W\in \cW: \\W\subset R\\
 L(W)=j-s_1}}
    b_{1,W}\bbone_R } 
{
\sum_{\substack{W'\in \cW: \\
L(W')=j-s_2}}T_j^*[b_{2,W'} \bbone_{\tr{R}}] } 
\end{multline} 
By H\"older's inequality and  \eqref{Tjregb} 
we get  for $R\in \fR_j$ 
\begin{multline*}
\Big|\biginn{ \sum_{\substack{W\subset R\\
 L(W)=j-s_1}}
    b_{1,W} } 
{
\sum_{\substack{
L(W')=j-s_2}}T_j^*[b_{2,W'} \bbone_{\tr{R}}] } \Big|
\, \\  \lc_{\varepsilon} 
B 2^{-\eps' s_2}
|R|^{-(\frac{1}{p}-\frac{1}{q})} 
\Big(
\sum_{\substack{W\subset R\\
 L(W)=j-s_1}}\| b_{1,W} \|_{L^p_{B_1}}^p\Big)^{1/p} 
 \Big(
\sum_{\substack{ W'\subset\tr{R}\\
L(W')=j-s_2}}\|f_{2,W'}\|_{L^{q'}_{B_2^*}}^{q'}\Big)^{1/q'}
  \end{multline*}
  and from here on the argument is analogous to the treatment of the term $IV_2$.
\end{proof}

\begin{remarka} 
It is instructive to observe that the term $III_4$ can be treated more crudely if one does not aim to obtain the constant $A_\circ(p,q)\log(2+\frac{B}{A_\circ(p,q)})$ in \eqref{eqn:cCdef}. More precisely, one simply splits $(-\infty, j)^2 \cap \bbZ^2$ into two regions $\tilde{\mathcal{V}}_{j,2}$ and $\tilde{\mathcal{V}}_{j,3}$, where
  \begin{align*}
      \tilde{\mathcal{V}}_{j,2}=\{(L_1,L_2) : L_1 \leq L_2< j\} \quad \text{and} \quad 
      \tilde{\mathcal{V}}_{j,3}=\{(L_1,L_2) : j>L_1 > L_2\}.
  \end{align*}
  Then split $III_4=\sum_{i=2}^3 IV_{i}$, where $IV_i$ are as in \eqref{eq:IVi def} but with $\mathcal{V}_{j,i}$ replaced by $\tilde{\mathcal{V}}_{j,i}$. One then considers the sum in $s_1$ in \eqref{eq:IV2} to start directly from $s_2$, and the sum in $s_2$ in \eqref{eq:IV3} to start directly from $s_1+1$. Using the same arguments, one obtains
  $$
  |IV_2| + |IV_3| \lesssim_{d,\gamma,p,q,\varepsilon} B |Q_0| \jp{f_1}_{Q_0,p} \jp{f_2}_{\tr{Q_0},q'}
  $$
  instead of \eqref{IV23bound}. Note that, as the term $IV_1$ does not appear in this case (see the bound \eqref{IV1bound}), this yields sparse domination with the constant $\mathcal{C}$ in \eqref{eqn:cCdef} replaced by
  $
  A(p)+A(q)+A_\circ(p,q) + B.
  $
\end{remarka}

\section{Maximal operators, square functions and long variations}\label{sec:squarefct-etc}

In this section we show that Corollary \ref{mainthm-cor} yields sparse domination results for maximal functions,  $\ell^r$-valued variants, $r$-variation norm operators and maximal and variational truncations of sums of operators. An application of Theorem \ref{thm:necessarycor} also yields necessary conditions for our sparse domination inequalities. We will formally state necessary conditions only for maximal functions and $\ell
^r$-valued functions (Theorem \ref{thm:ellr}) 
and leave to the reader the analogous formulations of those  conditions for $r$-variation norm operators (Theorem \ref{thm:Vr}), maximal truncations (Theorem \ref{mainthmtrunc}) and variational truncations (Theorem \ref{mainthmtruncvar}).

\subsection{Maximal functions and \texorpdfstring{$\ell^r$-variants}{variants}}\label{sec:maxfctandellr}

Given a family of operators $\{T_j\}_{j \in \bbZ}$ in $\mathrm{Op}_{B_1,B_2}$,
{consider the operators}
\Be \label{eqn:SrTdef}
S_rT f(x)= \Big(\sum_{j\in \bbZ}|T_j f(x)|_{B_2}^r\Big)^{1/r}
\Ee 
when $1\le r<\infty$ and also  the maximal operator
\Be\label{eqn:Sinfty}
S_\infty Tf(x)=\sup_{j \in \bbZ}|T_jf(x)|_{B_2}.
\Ee

 \begin{thm}\label{thm:ellr} 
 Let $1<p\le q<\infty$ and $1 \leq r \leq \infty$.
Let $\{T_j\}_{j\in \bbZ}$ be a family of operators in $\mathrm{Op}_{B_1,B_2}$
satisfying  \eqref{support-assu}. 

(i) Suppose that the inequalities
\begin{equation}\label{bdness-ellr}
    \big\|S_rTf \big\|_{L^{p,\infty}}\le A(p) \|f\|_{L^p_{B_1}} \qquad \text{and} \qquad
    \big\|S_rTf \big\|_{L^q}\le A(q) \|f\|_{L^{q,1}_{B_1}}
\end{equation}
hold for all $f \in \simpleone$. Moreover, assume that the rescaled operators $\dil_{2^j} T_j$ satisfy the %
single scale $(p,q)$ condition
\eqref{p-q-rescaled}
and %
single scale $\varepsilon$-regularity conditions
\eqref{p-q-rescaled-reg-a} and 
\eqref{p-q-rescaled-reg-b}.  Let $\cC$ be as in \eqref{eqn:cCdef}.
Then for all $f \in \simpleone$ and all $\bbR$-valued nonnegative measurable functions $\omega$,
\Be  \label{eqn:sparse-ellr}
\inn{S_rT f}{\omega}  \lc \cC 
\Lamaxga_{p,q'}(f,\om).
\Ee

(ii) In addition, assume $1< p < q < \infty$. If the family of operators $\{T_j\}_{j \in \bbZ}$ satisfies $T_j: \simpleone \to L^1_{B_2, \loc}$ and the strengthened support condition \eqref{eqn:strengthenedsupp}, then %
the condition
\[\|S_r T\|_{L^p_{B_1} \to L^{p,\infty} } +\|S_r T\|_{L^{q,1}_{B_1} \to L^q} +\sup_{j \in \bbZ}\|\dil_{2^j} T_j\|_{L^p_{B_1}\to L^q_{B_2} }  <\infty
\]
is necessary for the conclusion \eqref{eqn:sparse-ellr} to hold.
 \end{thm}

\begin{proof} [Proof of Theorem \ref{thm:ellr}]
We begin with the proof for $1\le r<\infty$.

Let $\delta_{kk}=1$ and $\delta_{jk}=0$ if $j\ne k$. Let $N_1\le N_2$ be integers, and for each integer $j\in [N_1,N_2]$, we define the operator $H_j$ sending $L^p_{B_1}$ functions to $\ell^r_{B_2}$-valued functions  by
\Be\label{eqn:Hjdef}  H_j f(x,k)= \begin{cases} \delta_{jk} T_jf(x)&\text{ if } N_1\le k\le N_2,
\\ 0 &\text{ if } k\notin [N_1,N_2].
\end{cases}
\Ee
We note that 
\Be\label{eqnLHT-identification} 
\Big(\sum_{j=N_1}^{N_2} |T_jf(x)|_{B_2}^r\Big)^{1/r} = \Big(\sum_{k=-\infty}^{\infty} \Big| \sum_{j=N_1}^{N_2} H_j f(x,k) \Big|^r_{B_2} \Big)^{1/r}.
\Ee 
By \eqref{bdness-ellr} we have 
\[ \Big\|\sum_{j=N_1}^{N_2} H_j \Big\|_{L^p_{B_1}\to L^{p,\infty}(\ell^r_{B_2})} \le A(p), \qquad 
\Big\|\sum_{j=N_1}^{N_2} H_j\Big\|_{L^{q,1}_{B_1}\to L^q(\ell^r_{B_2})} \le A(q),
\]
where we write $L^{p,\infty}(\ell^r_{B_2})$ to denote $L^{p,\infty}_{\ell^r_{B_2}}$.
The adjoint of $H_j$, acting on   $\ell^{r'}_{B_2^*}$-valued functions $g$,  is given by 
\[H_j^*g(x)= \sum_{k=N_1}^{N_2} \delta_{jk} T_j^*g_k(x).\]
The assumptions on $\dil_{2^j}T_j$ can be rewritten as 
\[
\sup_{j \in \bbZ} \|\dil_{2^j} H_j \|_{L^p_{B_1}\to L^q(\ell^r_{B_2})} \le A_\circ(p,q)
\] and 
\begin{align*} 
&\sup_{|h|\le 1}|h|^{-\eps}\sup_{j \in \bbZ} \|(\dil_{2^j} H_j )\circ\Delta_h\|_{L^p_{B_1}\to L^q(\ell^r_{B_2})} \le B
\\
&\sup_{|h|\le 1}|h|^{-\eps}\sup_{j\in \bbZ} \|(\dil_{2^j} H_j ^*)\circ\Delta_h\|_{L^{q'}(\ell^{r'}_{B_2^*})\to L^{p'}_{B_1^*}}\le B.
\end{align*} 
By Corollary \ref{mainthm-cor} applied to the sequence $\{H_j\}_{j \in \bbZ}$ in $\mathrm{Op}_{B_1,\ell^r_{B_2}}$, we get the conclusion  
\[ \int_{\bbR^d} \Big(\sum_{k=-\infty}^{\infty}\Big|\sum_{j=N_1}^{N_2} H_j f(x,k) \Big|^r_{B_2}\Big)^{1/r} \omega(x) dx
\lc \cC  
\Lamaxga_{p,q'}(f,\om),
\] which by 
\eqref{eqnLHT-identification} implies 
\[ \int_{\bbR^d} \Big(\sum_{j=N_1}^{N_2} |T_j f(x)|_{B_2}^r\Big)^{1/r}  \omega(x) dx
\lc \cC  
\Lamaxga_{p,q'}(f,\om).\] 
We apply the monotone convergence theorem to let $N_1\to-\infty$ and $N_2\to \infty$ and obtain the desired conclusion. This is possible since the implicit constant in the conclusion of Corollary \ref{mainthm-cor} does not depend on $B_1, B_2$.

The proof for $r=\infty$ is essentially the same, with notional changes. Since $H_jf(\cdot,k)=0$ when $k\notin [N_1,N_2]$, we can work with $\ell^\infty_{B_2}$ over the finite set $\bbZ \cap [N_1,N_2]$. Then there are no complications with the dual space, which is $\ell^1_{B_2^*}$ over $\bbZ \cap [N_1,N_2]$.

For  part (ii) one uses Theorem \ref{thm:necessarycor} in conjunction with  \eqref{eqnLHT-identification} and immediately arrives at the desired conclusion, via the monotone convergence theorem.\end{proof}

\subsection{Variation norms} \label{sec:variationnorms}
We now turn to the variation norms $V^r_{B_2} \equiv V^r_{B_2}(\bbZ)$ defined on $B_2$-valued functions of the integers $n\mapsto a(n)$. Let
$|a|_{V^\infty_{B_2}}=|a|_{L^\infty_{B_2}}$ and, for $1\le r<\infty$,
\Be\label{eqn:Vrdef} |a|_{V^r_{B_2}}= \sup_{n_1<\dots< n_M} |a(n_1)|_{B_2}+\Big(\sum_{\nu=1}^{M-1} |a(n_{\nu+1})-a(n_{\nu})|_{B_2}^r\Big)^{1/r}
\Ee 
where the supremum is taken over all positive integers $M$ and all finite increasing sequences of integers $n_1<...<n_M$.
Similarly, if $I_{N_1,N_2}= [N_1,N_2]\cap \bbZ$ we define the $V^r_{B_2}(I_{N_1,N_2})$ norm on functions on $I_{N_1,N_2}$ in the same way, restricting $n_1,\dots, n_M$ to $I_{N_1,N_2}$.

Given a sequence $T=\{T_j\}_{j\in \bbZ}$  in $\mathrm{Op}_{B_1,B_2}$ we define $\cV^r Tf(x)$ to be the $V^r_{B_2}$ norm of the sequence $j\mapsto T_j f(x)$. The $L^p$ norm of $\cV^rT f$ is just the $L^p(V^r_{B_2})$ norm of 
the sequence $\{T_j f\}_{j\in \bbZ}$. 
We define $\cV_{N_1,N_2}^r Tf(x)$ to be the $V^r_{B_2}$ norm of the sequence
$j\mapsto \bbone_{I_{N_1,N_2}} (j) T_j f(x)$.

The proof of the following theorem is almost identical to that of Theorem \ref{thm:ellr}.
 
 \begin{thm} \label{thm:Vr} Let $1 < p\le q< \infty$ and  $1\le r\le \infty$. Let $\{T_j\}_{j \in \bbZ}$ be a family of operators in $\mathrm{Op}_{B_1,B_2}$ satisfying \eqref{support-assu}.
 Suppose that the inequalities
 \Be \label{bdness-Vr}\begin{aligned}
    \big \|\cV^rT f\big\|_{L^{p,\infty}}\le A(p) \|f\|_{L^p_{B_1}} \qquad \text{and} \qquad
    \big \|\cV^rTf\big\|_{L^q}\le A(q) \|f\|_{L^{q,1}_{B_1}}
 \end{aligned}\Ee hold for all $f \in \simpleone$.
Moreover, assume that the rescaled operators $\dil_{2^j} T_j$ satisfy the %
single scale $(p,q)$ condition
\eqref{p-q-rescaled}
and %
single scale $\varepsilon$-regularity conditions
\eqref{p-q-rescaled-reg-a} and 
\eqref{p-q-rescaled-reg-b}.  Let $\cC$ be as in \eqref{eqn:cCdef}.
Then for all $f \in \simpleone$ and all $\bbR$-valued nonnegative measurable functions $\omega$,
\Be \label{eqn:sparse-Vr}
\inn{ 
 \cV^rT f}{ \omega}  \lc \cC 
\Lamaxga_{p,q'}(f,\om).
\Ee
\end{thm}

\begin{proof} In view of Theorem \ref{thm:ellr}, it suffices to consider the case $r<\infty$.
Given $N_1\le N_2$ we define 
$H_jf (x,k)$ as in 
\eqref{eqn:Hjdef}, for $N_1\le k\le N_2$.
Note that for fixed $x$, $N_1\le n_1<\dots< n_M\le N_2$,
\begin{multline} \label{eqn:secondHjidentificaton}
|T_{n_1} f(x)|_{B_2}+ \Big(\sum_{\nu=1}^{M-1} |T_{n_{\nu+1}}f(x)- T_{n_{\nu}}f(x)|_{B_2}^r\Big)^{1/r}
\\=\Big|\sum_{j=N_1}^{N_2} H_jf(x,n_1)\Big|_{B_2} +
\Big(\sum_{\nu=1}^{M-1} \Big|\sum_{j=1}^N H_j f(x,n_{\nu+1}) - \sum_{j=1}^N H_j f(x,n_{\nu})  \Big|_{B_2}^r\Big)^{1/r}.
\end{multline}
By \eqref{bdness-Vr} we have 
\[ \Big\|\sum_{j=N_1}^{N_2} H_j \Big\|_{L^p_{B_1}\to L^{p,\infty}(V^r_{B_2})} \le A(p), \qquad 
\Big\|\sum_{j=N_1}^{N_2} H_j\Big\|_{L^{q,1}_{B_1}\to L^q(V^r_{B_2})} \le A(q),
\]
where $V^r_{B_2}$ is interpreted to be the space $V^r_{B_2}(I_{N_1,N_2})$ and all the constants in what follows will be independent of $N_1$ and $N_2$. The pairing between $V^r_{B_2}(I_{N_1,N_2})$ and its dual is the standard one, \[\inn{a}{b}=\sum_{n=N_1}^{N_2} \inn{a(n)}{ b(n)}_{(B_2,B_2^*)}\]
and we have, 
\[|b|_{(V^r_{B_2}(I_{N_1,N_2}))^*}=\sup_{|a|_{V^r_{B_2}(I_{N_1,N_2})}\le 1 } \Big|\sum_{n=N_1}^{N_2} \inn{a(n)}{b(n)}_{(B_2,B_2^*)}\Big|.\] For  $\delta^{(j)}=(\delta_{j,N_1},\dots, \delta_{j,N_2}) $ we have,
for $j=N_1,\dots, N_2$, 
\[|\delta^{(j)}|_{V^r_{B_2}(I_{N_1,N_2})}=2^{1/r}\quad
\text{and}\quad 
|\delta^{(j)}|_{(V^r_{B_2}(I_{N_1,N_2}))^*}=2^{1/r'}.\]
The adjoint of $H_j$, acting on   $(V^r_{B_2}(I_{N_1,N_2}))^*$-valued 
functions $g=\{g_k\}_{k=N_1}^{N_2}$,   is given by 
\[H_j^*g(x)= \sum_{k=N_1}^{N_2} \delta_{jk} T_j^*g_k(x) .\]
These observations imply
\begin{align*}
  &\|\dil_{2^j} H_j \|_{L^p_{B_1}\to L^q(V^r_{B_2})} = 2^{1/r} \|\dil_{2^j} T_j\|_{L^p_{B_1}\to L^q_{B_2}},
    \\
    &\|(\dil_{2^j} H_j )\circ\Delta_h\|_{L^p_{B_1}\to L^q(V^r_{B_2})} 
    =2^{1/r}\|(\dil_{2^j} T_j )\circ\Delta_h\|_{L^p_{B_1}\to L^q_{B_2}} ,
    \\
    &\|(\dil_{2^j} H_j ^*)\circ\Delta_h\|_{L^{q'}((V^r_{B_2})^{*}) \to L^{p'}_{B_1^*}}
    =2^{1/r'}\|(\dil_{2^j} T_j ^*)\circ\Delta_h\|_{L^{q'}_{B_2^*}\to L^{p'}_{B_1^*}}.
\end{align*}
The hypothesis of Theorem \ref{mainthm} are then satisfied for the sequence $\{H_j\}_{j \in \bbZ}$ in $\mathrm{Op}_{B_1,V^r_{B_2}}$. Thus, by Corollary \ref{mainthm-cor} we obtain 
\[ \int_{\bbR^d} \Big| \sum_{j=N_1}^{N_2} H_j f(x,\cdot)\Big|_{V^r_{B_2}}  \omega(x) dx
\lc \cC  
\Lamaxga_{p,q'}(f,\om), 
\] which by \eqref{eqn:secondHjidentificaton}
implies 
\[ \int_{\bbR^d} \cV_{N_1,N_2}^rTf(x)  \omega(x) dx
\lc \cC  
\Lamaxga_{p,q'}(f,\om).
\] 
As the implicit constant in Corollary \ref{mainthm-cor} does not depend on the Banach spaces $B_1$, $B_2$ we may apply the  monotone convergence theorem and let $N_1\to-\infty$ and $N_2\to \infty$ to  obtain the desired conclusion \eqref{eqn:sparse-Vr}.
\end{proof}

\subsection{Truncations of sums} 
We will give a variant of Corollary \ref{mainthm-cor} in the spirit of Cotlar's inequality on maximal operators for truncations of singular integrals. %
 \begin{thm} \label{mainthmtrunc}
Let $1<p \leq q<\infty$.
Let $\{T_j\}_{j\in \bbZ}$ be a family of operators in $\mathrm{Op}_{B_1,B_2}$
satisfying  \eqref{support-assu}, 
\eqref{p-q-rescaled}, \eqref{p-q-rescaled-reg-a}, and \eqref{p-q-rescaled-reg-b}.
Moreover, assume that the estimates
\begin{subequations}
 \begin{align}\label{bdness-wt-trunc}  & \Big\|\sup_{\substack{(n_1,n_2):\\N_1\le n_1\le n_2\le N_2}}\big| \sum_{j=n_1}^{n_2}T_j f\big|_{B_2} \Big\|_{L^{p,\infty} }
 \le A(p) \|f\|_{L^p_{B_1}}, \\ 
 &\Big\|\sup_{\substack{(n_1,n_2):\\N_1\le n_1\le n_2\le N_2}}\big| \sum_{j=n_1}^{n_2}T_j f\big|_{B_2} \Big\|_{L^{q} }
 \le A(q) \|f\|_{L^{q,1}_{B_1}} 
 \label{bdness-rt-trunc}
 \end{align} 
 \end{subequations} hold uniformly for all $(N_1,N_2)$ with $N_1\le N_2$. Let $\cC$ be as in \eqref{eqn:cCdef}.
 Then
for all  $f\in \simpleone$, all $\bbR$-valued nonnegative
measurable %
functions $\omega$, and all integers $N_1, N_2$ with $N_1\le N_2$, 
 \begin{equation*}
\int_{\bbR^d}
\sup_{\substack{(n_1,n_2):\\N_1\le n_1\le n_2\le N_2}}\big| \sum_{j=n_1}^{n_2}T_j f(x)\big|_{B_2} \, \omega(x) dx  \lesssim \cC 
\Lamaxga_{p,q'}(f,\om).
\end{equation*}
\end{thm}
\begin{proof}
Define $U(N_1,N_2)=\{(n_1,n_2): N_1\le n_1\le n_2\le N_2\}$ and $\ell^\infty_{B_2}$ as the space of all bounded $B_2$-valued functions on $U(N_1,N_2)$.  Define operators  $H_j$ in  $\mathrm{Op}_{B_1,\ell^\infty_{B_2}}$ by 
\Be \notag H_jf(x,n_1,n_2)= \begin{cases} T_j f(x) & \text{ if } N_1 \le n_1\le j\le n_2 \le N_2,
\\
0 & \text{ otherwise.}\end{cases} 
\Ee
Then apply 
Corollary \ref{mainthm-cor}
to the operators $\sum_{j=N_1}^{N_2} H_j$ as in the proof of Theorems \ref{thm:ellr}.
\end{proof}

We also have a variational analogue.
\begin{thm} \label{mainthmtruncvar}
Let $1<p\leq q<\infty$.
Let $\{T_j\}_{j\in \bbZ}$ be a family of operators in $\mathrm{Op}_{B_1,B_2}$
satisfying  \eqref{support-assu}, \eqref{bdness-wt}, \eqref{bdness-rt},
\eqref{p-q-rescaled}, \eqref{p-q-rescaled-reg-a}, and \eqref{p-q-rescaled-reg-b}.
Moreover, assume that  the estimates
\begin{subequations}
 \begin{align}\label{bdness-wt-trunc-var}  & \Big\|\sup_{M \in \bbN}\sup_{N_1\le n_1<\dots< n_{M} \le N_2} 
 \Big(\sum_{\nu=1}^{M-1} \big| \sum_{j=n_\nu+1}^{n_{\nu+1}} T_j f\big|_{B_2} ^r\Big)^{1/r} \Big\|_{L^{p,\infty}} 
  \le A(p) \|f\|_{L^p_{B_1}} \\ \label{bdness-rt-trunc-var} 
   & \Big\|\sup_{M \in \bbN}\sup_{N_1\le n_1<\dots< n_{M} \le N_2} 
 \Big(\sum_{\nu=1}^{M-1} \big| \sum_{j=n_\nu+1}^{n_{\nu+1}} T_j f\big|_{B_2} ^r\Big)^{1/r} \Big\|_{L^{q}} 
  \le A(p) \|f\|_{L^{q,1}_{B_1}} 
  \end{align} 
  \end{subequations}
  hold uniformly for all $(N_1,N_2)$ with $N_1 \leq N_2$. Let $\cC$ be as in \eqref{eqn:cCdef}.
 Then
for all  $f\in \simpleone$, all $\bbR$-valued nonnegative
measurable %
functions $\omega$, and  all integers $N_1, N_2$ with $N_1\le N_2$,
 \begin{multline}\label{norms-trunc}
\int_{\bbR^d} 
\sup_{M \in \bbN}\sup_{N_1\le n_1<\dots< n_{M} \le N_2} 
 \Big(\sum_{\nu=1}^{M-1} \big| \sum_{j=n_\nu+1}^{n_{\nu+1}} T_j f(x)\big|_{B_2} ^r\Big)^{1/r} 
 \omega(x) dx  \\ \lesssim  \cC 
\Lamaxga_{p,q'}(f,\om).
\end{multline}
\end{thm}
\begin{proof}
Let $V^r_{B_2} \equiv V_{B_2}^r (I_{N_1,N_2})$ denote the $r$-variation space of $B_2$-valued functions over the integers in $[N_1,N_2]$ and for $N_1\le j\le N_2$, 
$N_1\le n\le  N_2$, define the  operators $H_j \in \mathrm{Op}_{B_1,V^r_{B_2}}$ by 
\begin{equation*}
    H_jf(x,n)= \begin{cases}
     T_j f(x) & \text{ if } N_1\le j\le n \le N_2, \\
     0 & \text{ if } j>n.
    \end{cases}
\end{equation*}
Note that, by definition of $H_j$, $| \sum_{j=N_1}^{N_2} H_j f(x, \cdot) |_{V^r_{B_2}}$ equals to
\begin{equation}\label{eq:Vr trunc expanded}
\sup_{\substack {M\in \bbN\\ N_1 \leq n_1 < \cdots < n_M \leq N_2}} \big| \sum_{j=N_1}^{n_1} T_j f(x) \big|_{B_2} + \Big( \sum_{\nu=1}^{M-1} \big| \sum_{j=N_1}^{n_{\nu+1}} T_j f(x) - \sum_{j=N_1}^{n_\nu} T_j f(x)  \big|_{B_2}^r  \Big)^{1/r}
\end{equation}
and $|  H_j f(x, \cdot) |_{V^r_{B_2}}=| T_j f(x)|_{B_2}$. Arguing as in Theorem \ref{thm:Vr}, one may apply  Corollary \ref{mainthm-cor} to the operators $\sum_{j=N_1}^{N_2} H_j$ in $\mathrm{Op}_{B_1,V_{B_2}^r}$. Note, in particular, that in view of \eqref{eq:Vr trunc expanded} the conditions \eqref{bdness-wt} and \eqref{bdness-rt} for $\sum_{j=N_1}^{N_2} H_j$ follow from \eqref{bdness-wt-trunc-var} and \eqref{bdness-rt-trunc-var} together with the fact that $\{T_j\}_{j \in \bbZ}$ in $\mathrm{Op}_{B_1,B_2}$ satisfy \eqref{bdness-wt} and \eqref{bdness-rt}. This automatically yields \eqref{norms-trunc}.
\end{proof}

\subsection{Some simplifications for maximal operators}\label{sec:more max}
The goal of this section is to remark that the proof of Theorem \ref{thm:ellr} can be simplified in the case $q \leq r \leq \infty$. Rather than deducing it  from Corollary \ref{mainthm-cor}, we shall apply the proof method of Theorem \ref{mainthm} to the operators $S_r$ and observe that a Calder\'on--Zygmund decomposition on $f_2$ is not required for the proof to work. In particular, this allows us to remove the regularity hypothesis \eqref{p-q-rescaled-reg-b} on the adjoints $T_j^*$. The precise statement reads as follows.

\begin{thm}\label{thm:max weak ass}
 Let $1<p\leq q<\infty$ and $q \leq r \leq \infty$.
Let $\{T_j\}_{j\in \bbZ}$ be a family of operators in $\mathrm{Op}_{B_1,B_2}$
satisfying  \eqref{support-assu}. 
Suppose that the inequalities
\begin{equation}\label{bdness-ellr-alt}
    \big\|S_rTf \big\|_{L^{p,\infty}}\le A(p) \|f\|_{L^p_{B_1}} \qquad \text{and} \qquad
    \big\|S_rTf \big\|_{L^q}\le A(q) \|f\|_{L^{q,1}_{B_1}}
\end{equation}
hold for all $f \in \simpleone$. Moreover, assume that the rescaled operators $\dil_{2^j} T_j$ satisfy the %
single scale $(p,q)$ condition
\eqref{p-q-rescaled}
and %
single scale $\varepsilon$-regularity condition
\eqref{p-q-rescaled-reg-a}.  Let $\cC$ be as in \eqref{eqn:cCdef}.
Then for all $f \in \simpleone$ and all $\bbR$-valued nonnegative measurable functions $\omega$,
\Be \notag  \label{eqn:sparse-ellr-alt}
\inn{ S_rT f}{\omega}  \lc \cC \Lamaxga_{p,q'} (f,\om) .
\Ee
\end{thm}

\begin{proof}
We sketch the main changes with respect to the proof of Theorem \ref{mainthm}. As in Theorem \ref{mainthm}, it suffices to show
\[ \int_{\bbR^d}  S_{r,N_1,N_2}f_1(x)  f_2(x) dx
\lc \cC  \Lamaxga_{p,q'} (f_1,f_2) %
\]
uniformly in $N_1 \leq N_2$ for all $f_1 \in \simpleone$ and $f_2 \in S_{\bbR}$, where
\begin{align*}
S_{r,N_1,N_2} f(x)&:=\Big( \sum_{j=N_1}^{N_2} |T_jf(x)|_{B_2}^r\Big)^{1/r}
\intertext{for $q \leq r < \infty$ and}
S_{\infty,N_1,N_2} f(x)&:= \sup_{N_1 \leq j \leq N_2} |T_jf(x)|_{B_2}.
\end{align*}
This will in turn follow from verifying the inductive step in Claim \ref{claim} for the operators $S_{r,N_1,N_2}$.

If $r=\infty$, let $\lambda_j(x) \in B_2^*$ with $|\lambda_j(x)  |_{B_2^*} \leq 1$ such that $| T_{j} f_1 (x)|_{B_2}= \inn{T_j f_1}{\lambda_j}_{(B_2,B_2^*)}$ and let $x \mapsto j(x)$ be a measurable function such that
$$
S_{\infty, N_1,N_2} f_1(x) \leq 2 \, |T_{j(x)} f_1(x)|_{B_2}.
$$
Setting $X_j:=\{x : j(x)=j\}$, note that
$$ Sf_1(x) \equiv |T_{j(x)} f_1(x)|_{B_2} =  \sum_{j=N_1}^{N_2} \inn{ T_j f_1(x)}{ \lambda_j(x)\bbone_{X_j} (x) }_{(B_2,B_2^*)} 
$$
and that the $X_j$ are disjoint measurable sets such that $\sum_j \bbone_{X_j} \leq \bbone_{\tr{Q_0}}$.  If $q \leq r < \infty$, we linearise the $\ell^r(B_2)$-norm for each $x$. That is, there exists $\{a_j(x)\}_{j \in \bbZ} \in \ell^{r'}(B_2^*)$, with $\| a_j(x) \|_{\ell^{r'}(B_2^*)}\leq 1$, such that
$$
Sf_1(x) \equiv S_{r,N_1,N_2} f_1 (x) = \sum_{j=N_1}^{N_2}  \inn{T_j f_1 (x)}{a_j(x)}_{(B_2,B_2^*)}.
$$
Note that we can treat the cases $r=\infty$ and $q \leq r < \infty$ together by setting $a_j(x)=\lambda_j(x) 1_{X_j}(x)$ for all $x \in \tr{Q_0}$  and all $N_1 \leq j \leq N_2$, and $a_j(x)=0$ otherwise; then $\{a_j(x)\}_{j \in \bbZ} \in \ell^1(B_2^*)$. Clearly, the operator $S$ satisfies the bounds \eqref{bdness-wt}, \eqref{bdness-rt} in view of \eqref{bdness-ellr-alt}.

We then perform a Calder\'on--Zygmund decomposition of $f_1$ as in \eqref{first splitting}. The first term in \eqref{first splitting}, corresponding to $g_1$, can be treated analogously. The second term in \eqref{first splitting}, corresponding to $\sum_{W \in \mathcal{W}} b_{1,W}$ can be further split as in \eqref{three terms}, and $I$ and $II$ can be treated analogously. One is then left with proving \eqref{error-claim} for $III$. Rather than performing a Calder\'on--Zygmund decomposition on $f_2$, we estimate the term directly.

Indeed, the analysis for $III$ amounts to a simplified version of the analysis of the term $III_4$ in \eqref{eq:termIV}. One can define $\ell$ as in \eqref{elldef} and split
\[(-\infty, j) \cap \bbZ=\mathscr{V}_{j,1} \cup \mathscr{V}_{j,2},\] %
where $\mathscr{V}_{j,1}:=\{L :  j - \ell \leq L < j \}$ and $\mathscr{V}_{j,2} = \{ L : L<j-\ell \}$.
Note that here there is no further need to split $\mathscr{V}_{j,2}$, since we do not make use of a Calder\'on--Zygmund decomposition of $f_2$. Write $III=IV_1^\flat+IV_2^\flat$,
where for $i=1,2$,
\[ IV_{i}^\flat = \biginn{\sum_{N_1\le j\le N_2} \sum_{\substack{W\in \cW,\\ L(W)\in \mathscr{V}_{j,i} }}T_j b_{1,W}}{f_2}. \]
We first focus on $IV^\flat_{1} $. By H\"older's inequality with respect to $x$ and $j$
\begin{equation}\label{eq:Holder x and j}
IV^\flat_1 \le IV_{1,1}^\flat IV^\flat_{1, 2},  
\end{equation}
where
\begin{align*} 
IV^\flat_{1,1}&= \Big(\sum_{j=N_1}^{N_2}\int \Big|
 \sum_{j-\ell \le L(W)<j} T_j b_{1,W}(x) \Big|_{B_2}^qdx\Big)^{1/q},
\\
IV^\flat_{1,2} &= \Big(\sum_{j=N_1}^{N_2}\int |a_j(x)|_{B_2}^{q'} |f_2(x)|^{q'} dx\Big)^{1/q'}.
\end{align*}
Using that $\| a_j(x) \|_{\ell^{q'}(B_2^*)} \leq \| a_j(x) \|_{\ell^{r'}(B_2^*)} =1$ if $1 \leq r'\leq q'$, we get
\Be \label{III2n}
IV^\flat_{1,2}\lc \Big(\int_{\tr{Q_0}} |f_2(x)|^{q'}  dx\Big)^{1/q'} \lc |Q_0|^{1-1/q} \jp{f_2}_{\tr{Q_0},q'}.
\Ee
For the term $IV^\flat_{1,1}$, introduce as in \eqref{eq:IV1} the family $\fR_j$ of subcubes of $Q_0$ of side length $2^j$ and use the bounded overlap of $\tr{R}$ to write
\[
IV^\flat_{1,1} \lesssim \Big( \sum_{j=N_1}^{N_2} \sum_{R \in \fR_j} \Big\|  \sum_{\substack{ W \subset R \\ j-\ell\le L(W)<j}} T_j b_{1,W}  \Big\|_{L^q_{B_2}}^q \Big)^{1/q}.
\]
The right-hand side above can then be handled essentially as $IV_1$ in \eqref{eq:IV1} after using the %
single scale $(p,q)$ condition \eqref{p-q-rescaled} for each $T_j$ (in the form \eqref{eqn:p-qj}); the only difference is the presence of an $\ell^q$-sum. More precisely,
\begin{align*}
    IV^\flat_{1,1} &\lesssim A_\circ(p,q) \Big(\sum_{j=N_1}^{N_2}\sum_{R\in \fR_j} |R|^{-(\frac 1p-\frac1q)q}\Big( \sum_{\substack{W \subset R \\ j-\ell \leq L(W) < j}}  \| b_{1,W} \|_{L^p_{B_1}}^p \Big)^{q/p}\Big)^{1/q} \\
    & \lesssim_{d} A_\circ(p,q) \jp{f_1}_{Q_0,p, B_1} \Big(\sum_{j=N_1}^{N_2}\sum_{R\in \fR_j} |R|^{-(\frac 1p-\frac1q)q}\Big( \sum_{\substack{W \subset R \\ j-\ell \leq L(W) < j}}  |W| \Big)^{q/p}\Big)^{1/q} \\
&\lesssim_{d} A_\circ(p,q) \jp{f_1}_{Q_0,p,B_1}
\Big(\sum_{j=N_1}^{N_2}\sum_{R\in \fR_j} \sum_{\substack{W\subset R\\j-\ell\le L(W)<j}} |W| 
\Big)^{1/q} \\
&\lesssim \ell^{1/q} \jp{f_1}_{Q_0,p, B_1} |Q_0|^{1/q},
\end{align*}
and combining this with \eqref{III2n}, the bound for $IV^\flat_{1}$ immediately follows. %

Regarding $IV^\flat_{2}$, write $IV^\flat_{2}= \sum_{s=\ell+1}^\infty IV^\flat_{2}(s)$, where $IV^\flat_{2}(s)$ has the sum in $L(W)< j-\ell$ further restricted to $L(W)=j-s$. For each fixed $s$, one can apply H\"older's inequality with respect to $x$ and $j$ as in \eqref{eq:Holder x and j},
\[ IV^\flat_2(s) \le IV^\flat_{2,1}(s) IV^\flat_{2,2}, \]
where the term $IV^\flat_{2,2}$ (which is independent of $s$) %
can be treated as $IV^\flat_{1,2}$ in \eqref{III2n}. For each $IV^\flat_{2,1}(s)$ we write again
$$
IV^\flat_{2, 1}(s) \lesssim \Big( \sum_{j=N_1}^{N_2} \sum_{R \in \mathfrak{R}_j} \Big\|  \sum_{\substack{ W \subset R \\  L(W)=j-s}} T_j b_{1,W}  \Big\|_{L^q_{B_2}}^q \Big)^{1/q}.
$$
This term can now be treated as the term $IV_2$ in \eqref{eq:IV2} {using} the $\varepsilon$-regularity condition \eqref{p-q-rescaled-reg-a} to get a decay of $2^{-s\varepsilon'}$ (as in \eqref{Tjrega}). The only difference with respect to \eqref{eq:IV2} is the presence of the $\ell^q$-sum, which introduces no difficulty, as shown above for $IV^\flat_{1,1}$. This completes the proof.
\end{proof}

\section{Fourier multipliers}\label{sec:Fourier-multipliers} 

In this section we deduce Theorems \ref{thm:localtoglobal-sparse} and \ref{thm:localtoglobal-sparse-s} from a  more general result which will lead to more precise sparse domination results and also cover Hilbert space valued versions. We are given two separable Hilbert spaces $\sH_1$, $\sH_2$ and denote by $\sL(\sH_1,\sH_2)$ the space of bounded linear operators from $\sH_1$ to $\sH_2$ (in our applications one of the Hilbert spaces will be usually $\bbC$).
Consider the translation invariant operator  $\cT=\cT_m$ mapping $\sH_1$-valued functions to $\sH_2$-valued functions given via a multiplier 
\[\widehat {\cT f}(\xi)=m(\xi)\widehat f(\xi),\]
where $m(\xi)\in \sL(\sH_1,\sH_2)$ for almost every $\xi$. For $1\le p\le q\le \infty$ we write $m\in M^{p,q}_{\sH_1,\sH_2}$ 
if the inequality
\[\|\cT f\|_{L^q(\sH_2)} \le C \|f\|_{L^p(\sH_1)} \] holds for all $\sH_1$-valued Schwartz functions, and the best constant 
defines the norm  in $M^{p,q}_{\sH_1,\sH_2}$.
We may occasionally drop the Hilbert spaces if it is understood from the context and also write $M^p$ for $M^{p,p}$.  Note that $m\in M^{p,q}_{\sH_1,\sH_2}$ implies by a duality argument that 
$m\in M^{q',p'}_{\sH_2^*, \sH_1^*}$. The $M^{2,2}_{\sH_1,\sH_2}$ norm is bounded by $\|m\|_{L^\infty_{\sH_1,\sH_2}}  $ where  we write $L^\infty_{\sH_1,\sH_2} $ for $L^\infty_{\sL(\sH_1,\sH_2)}$. Also note that by the  Marcinkiewicz--Zygmund theorem  \cite[\S2.1b]{hytonen-etal} any scalar multiplier in $M^{p,q}$ extends naturally, for any {separable}  Hilbert space $\sH$,  to a multiplier \[ m\otimes  I_{\sH}\in M^{p,q}_{\sH, \sH}, \text{ with } m \otimes I_\sH: 
\begin{cases} \bbR^d\to \sL(\sH, \sH), \\ \xi \mapsto (v \mapsto m(\xi) v)\end{cases} \]
 and we have $\|m\otimes I_{\sH}\|_{M^{p,q}_{\sH,\sH}}\le C \|m\|_{M^{p,q}}$ where $C$ does not depend on the Hilbert space.

\subsection{The main  multiplier theorem}\label{sec:more-precise}

In what follows let $\phi$ be a radial $C^\infty$ function supported in $\{\xi\in \widehat {\bbR}^d: 1/2<|\xi|<2\}$ (not identically zero). 
Let $\Psi_0\in C^\infty(\bbR^d)$ be supported in $\{x \in \bbR^d:|x|<1/2\}$ such that $\Psi_0(x)=1$ for $|x|\le 1/4$. For $\ell>0$ define 
\Be \label{eqn:defofPsiell}
\Psi_\ell(x)= \Psi_0(2^{-\ell} x)-\Psi_0(2^{-\ell+1}x)\Ee
which is supported in $\{x:2^{\ell-3}\le |x|\le 2^{\ell-1}\}.$
Define \begin{subequations}
\begin{align}
\label{eqn:Bm}
\cB[m]
&:=\sum_{\ell\ge 0} \sup_{t>0} \|[\phi m(t\cdot)]*\widehat {\Psi}_\ell\|_{M^{p,q}_{\sH_1,\sH_2}} 2^{\ell d(1/p-1/q)} (1+\ell),
\\
\label{eqn:Bmcirc}
\cB_\circ[m]
&:=\sum_{\ell\ge 0} \sup_{t>0} \|[\phi m(t\cdot)]*\widehat {\Psi}_\ell\|_{L^\infty_{\sH_1,\sH_2} }.
\end{align}
\end{subequations}

\begin{thm}\label{thm:sparsemult}
Let $1<p\le q<\infty$, $1/q'=1-1/q$,  and assume that $m\in L^\infty_{\sH_1, \sH_2}$ is such that
$\cB_\circ[m]$ and $\cB[m]$ are finite.  Then $\cT_m\in \Sp(p,\sH_1,q',\sH_2^*)$ with 
\[\|\cT_m\|_{\Sp_\gamma(p,\sH_1,q',\sH_2^*)} \lc_{d,\gamma,p,q} \cB[m]+\cB_\circ[m].\]
The implicit constant does not depend on $m,\sH_1, \sH_2$.
\end{thm}

We note that the finiteness of  $\cB_\circ[m]$ 
is implied by the finiteness of $\cB[m]$ in the case $\sH_1=\sH_2=\bbC$.
\begin{rem}\label{remark multiplier}
The function space of all $m$ with $\cB_\circ[m]+\cB[m]<\infty$ exhibits familiar properties of similarly defined function spaces in multiplier theory. For example:
\begin{enumerate}
\item  The space is invariant under multiplication by a standard smooth symbol of order 0. %
This fact will be used in the proof of Theorem \ref{thm:sparsemult} and for the convenience of the reader, the precise statement and proof are contained in \S \ref{sec:bmsmooth-pf} below.

\item The finiteness of $\cB[m]$ and  $\cB_\circ[m]$ is independent of the choice of the specific functions $\phi$ and $\Psi$. This observation will be convenient in the proof of Theorem \ref{thm:sparsemult}. It can be verified by standard arguments but, for completeness, the proof is provided in \ref{sec:finiteness Bm} below.
\end{enumerate}
\end{rem}

We begin by showing how Theorem \ref{thm:sparsemult} implies Theorems \ref{thm:localtoglobal-sparse} and \ref{thm:localtoglobal-sparse-s}. Then we review some known facts and estimates for Fourier multipliers and deduce the proof of Theorem \ref{thm:sparsemult} from our main Theorem \ref{mainthm}.

\begin{proof}[Proof of Theorem \ref{thm:localtoglobal-sparse} using Theorem \ref{thm:sparsemult}]\label{sec:proofoflocaltoglobal-sparse}
We have to check the assumptions of Theorem \ref{thm:sparsemult}. %
Assumption \eqref{eqn:Holder-single-scale} is equivalent with 
\[ \|[\phi m(t\cdot)]*\widehat{\Psi}_\ell \|_{M^2} \le 2^{-\ell\eps}.\]
Thus interpolating \eqref{eqn:Holder-single-scale} and \eqref{eqn:Mp-single-scale} we get for $p\in (p_0,2)$,
\[ \|[\phi m(t\cdot)]*\widehat{\Psi}_\ell\|_{M^p} \le 2^{-\ell\eps(p)} \text{ where } \eps(p)=\eps \big(\tfrac 1{p_0}-\tfrac 1p\big)\big / \big(\tfrac 1{p_0}-\tfrac 12 \big). \]
 Let $\chi\in C^\infty_c(\widehat {\bbR}^d\setminus \{0\})$ so that $\chi(\xi)=1$ in a neighborhood of $\supp (\phi)$. Then by Young's convolution inequality for all $p\in (p_0,2)$, $q\in [p,\infty]$, $t>0$,
 \[
 \|\chi ([\phi m(t\cdot)]*\widehat{\Psi}_\ell)\|_{M^{p,q}} \lc\| [\phi m(t\cdot)]*\widehat{\Psi}_\ell\|_{M^{p}}  \lc 2^{-\ell\eps(p) }.
 \]
On the other hand we claim that
\begin{equation}\label{eqn:localtoglobalpf-tail}
 \|(1-\chi) ([\phi m(t\cdot)]*\widehat{\Psi}_\ell)\|_{M^{p,q}} \lc_N \|\varphi m(t\cdot)\|_1  2^{-\ell N}.
 \end{equation}
 Indeed, integration by parts in $\xi$ in the integral
 \[ \int \int e^{ix\cdot \xi} (1-\chi)(\xi) [\phi m(t\cdot)](\zeta) 2^{\ell d} \widehat{\Psi}(2^\ell (\xi-\zeta)) d\xi d\zeta \]
 implies the pointwise estimate
 \[ |\mathcal{F}^{-1}\big((1-\chi) ([\phi m(t\cdot)]*\widehat{\Psi}_\ell)\big)(x)| \lesssim_N 2^{-\ell N} (1+|x|)^{-N} \|\phi m(t\cdot)\|_1, \]
 whence \eqref{eqn:localtoglobalpf-tail} follows from Young's convolution inequality.
 
Fix $p\in (p_0,2)$. Combining the two estimates we see that condition \eqref{eqn:Bm} holds for a pair of exponents $(p_1,q_1)$ if  $p_1\in (p_0,p)$, $q_1>p_1$ and \[d(1/p_1-1/q_1) <\eps(p_1).\] Then Theorem \ref{thm:sparsemult} gives $\cT\in \Sp(p_1,q_1')$.
One can then choose $\delta=\delta(p)>0$ small enough so that $p_1=p-\delta$ and $q_1'=p'-\delta$ satisfy the above conditions. This concludes the proof.
\end{proof}

\begin{proof}[Proof of Theorem \ref{thm:localtoglobal-sparse-s} using Theorem \ref{thm:sparsemult}]\label{sec:proofoflocaltoglobal-sparse-s}
We need to check that $\cB[m]~<~\infty$, which will follow from showing that
\Be \label{eqn:multcond}
\sup_{t>0} \|[\phi m(t\cdot)]* \widehat {\Psi}_\ell\|_{M^{p,q}} \lc 2^{-\ell s}
\Ee for some $s>d(1/p-1/q)$. Here we are in the case $\sH_1=\sH_2=\bbC$, so this also implies $\cB_\circ[m]< \infty$.

We decompose $\Psi_\ell$ into slighly smaller pieces.
Recall that $\Psi_1$ is supported in $\{x:1/4\le |x|\le 1\}$ and $\Psi_\ell(x)=\Psi_1(2^{1-\ell}x)$. We form a  partition of unity $\{\varsigma_\nu: \nu\in \cI\}$ such that $\sum_{\nu\in \cI} \varsigma_\nu(x)=1$ for $|x|\in [1/8,2]$,
and $\varsigma$ is a $C^\infty$ function supported in a ball $B(x_\nu,r_\nu)$ centered at $x_\nu$, with $|x_\nu|\in [1/4,1]$ and radius $r_\nu\le 10^{-2}$. Let 
\[u_\nu = \frac{\pi}{2} \frac{x_\nu}{|x_\nu|^2}\] so that 
$|u_\nu|\in [1, 8]$ and $\inn{x_\nu}{u_\nu}= \pi/2$. This implies that $|\Im (e^{i\inn{x}{u_\nu}}-1)|>1/2$ for $x\in \supp(\varsigma_\nu)$. Define, for $M$ as in \eqref{eqn:MpqHolder},
\[ \Psi_{1,\nu}(x)= \frac{\Psi_1(x) \varsigma_\nu(x)}
{(e^{i\inn{x}{u_\nu}} -1)^M }, \quad \Psi_{\ell,\nu}(x) = \Psi_{1,\nu}(2^{1-\ell }x) 
\]
and note that $\Psi_{1,\nu}$ is smooth and $\Psi_\ell(x)= \sum_\nu \Psi_{\ell,\nu} (x) (e^{i\inn{x}{2^{1-\ell}u_\nu}}-1)^M$.
Hence 
\[ \phi m(t\cdot)* \widehat {\Psi}_\ell= \sum_\nu 
\Delta^M_{-2^{1-\ell}u_\nu} [\phi m(t\cdot)] *\widehat{\Psi}_{\ell,\nu}\]
and by assumption we have for some $s>d(1/p-1/q)$, 
\[\|\phi m(t\cdot)* \widehat {\Psi}_\ell\|_{M^{p,q}}
\lc \sum_\nu \|
\Delta^M_{-2^{1-\ell}u_\nu} [\phi m(t\cdot)]\|_{M^{p,q}} \lc 2^{-\ell s}.
\] 
This implies \eqref{eqn:multcond}
and  now Theorem 
\ref{thm:localtoglobal-sparse-s} follows from  Theorem \ref{thm:sparsemult}.
\end{proof}

\subsection{A result involving  localizations of Fourier multipliers}\label{sec:combinescales-mult}
We recall   a theorem  from \cite{See88} (see also \cite{carbery-revista} for a similar result) which we will formulate in the vector-valued version (see also \cite{GuoRoosYung}).

Let $\phi$ be as before,  
and fix $1<p<\infty$. Assume 
\begin{subequations}
\begin{align}
\label{Lpcond}
&\sup_{t>0} \|\phi m(t\cdot)\|_{M^p_{\sH_1,\sH_2}} \le \fa
\\&\sup_{t>0} \|\phi m(t\cdot) \|_{L^\infty_{\sH_1,\sH_2}} \le \fa_\circ,\label{L2cond}
\end{align}
and
\Be\label{Hoerm-cond}
\sum_{|\alpha|\le d+1} \sup_{t>0} \sup_{\xi \in \widehat{\bbR}^d} |\partial_\xi^\alpha
(\phi m(t\cdot))(\xi)|_{\sL(\sH_1,\sH_2)}  \le \fb,
\Ee
\end{subequations}
where $\alpha \in \bbN_0^d.$ Then
\Be\label{multconcl}
\|m\|_{M^p_{\sH_1,\sH_2}} \lc \fa_\circ+\fa \log(2+\fb/\fa)^{|\frac 1p-\frac 12|}.
\Ee
Of course, in the special case $\sH_1=\sH_2=\bbC$ the $L^2$-boundedness condition \eqref{L2cond} with $\fa_\circ\le \fa$ is implied by \eqref{Lpcond} (cf. an analogous remark following  Theorem \ref{thm:sparsemult}).

\subsection{Proof of Theorem \ref{thm:sparsemult}}\label{sec:proofofFM}
First assume that $m$ is compactly supported in $\widehat{\bbR}^d\setminus\{0\}$ without making any quantitative assumption on the support.

Note that by Remark \ref{remark multiplier} we have some freedom to make a convenient choice of the localizing function $\phi$, and we will denote this choice by $\varphi$.  In what follows, let $\theta\in C^\infty_c(\bbR^d)$  be radial such that $\theta $ is supported in $\{x \in \bbR^d: |x|<1/2\}$, such that 
$\int\theta(x)\pi(x) dx=0$ for all polynomials $\pi$ of degree at most $10 d$,
 and such that $\widehat \theta(\xi)>0$ for $1/4\le |\xi|\le 4$.
We then choose $\varphi$  to be a radial $C^\infty$ function supported in $\{\xi\in \widehat {\bbR}^d: 1/2<|\xi|<2\}$ such that
\Be \notag \label{eqn:varphitheta}
\sum_{k\in \bbZ}  \varphi(2^{-k}\xi) \widehat \theta(2^{-k}\xi)=1
\Ee
for all $\xi\neq 0$.

We then decompose $\cT$ by writing 
\begin{equation} \notag \label{eqn:comp-m}
m(\xi)= \sum_{k=n_1}^{n_2} \widehat \theta(2^{-k}\xi) \varphi(2^{-k} \xi)m(\xi)
\end{equation}
where $n_1, n_2\in \bbZ$.
We then decompose 
\[\F^{-1}[ \varphi m(2^k\cdot)](x)= \sum_{\ell\ge 0} 
\F^{-1}[ \varphi m(2^k\cdot)](x)\Psi_\ell(x)\]
which yields 
\[\cF^{-1}[m\widehat f] (x)=
\sum_{\ell\ge 0} \cT^\ell f(x)= \sum_{\ell\ge 0}
\sum_{k=n_1}^{n_2}\cT^{\ell,k} f(x) 
\]
where
\Be \notag \widehat{\cT^{\ell,k} f}(\xi)=  \widehat \theta(2^{-k}\xi)  \,\,[\varphi m(2^k\cdot)]\!*\!\widehat \Psi_\ell (2^{-k} \xi)
.
\Ee
We can write $\cT^{\ell,k}f= K^\ell_k*f$ with 
\[ K^\ell_k (x) = \int \cF^{-1} [\varphi(2^{-k}\cdot)m](x-y)
\Psi_\ell (2^k (x-y)) 2^{kd} \theta(2^ky)\, dy.
\]
Observe that $K^\ell_k(x)$ is supported in $\{x \in \bbR^d:|x|\le 2^{\ell+1-k}\}$.
We wish to  apply Theorem \ref{mainthm} to the operators $\cT^\ell$ defined by
\Be\notag \label{Telldef}\cT^\ell f=\sum_{k=n_1}^{n_2} K^\ell_k*f= \sum_{j=\ell+1-n_2}^{\ell+1-n_1} T^\ell_jf \quad \text{ with } T^\ell_jf =K^\ell_{\ell+1-j}*f.
\Ee
The operators $T_j^\ell$ satisfy the support condition \eqref{support-assu}.
To check the conditions \eqref{bdness-wt}, \eqref{bdness-rt} we apply the above mentioned theorem from \cite{See88} (see \eqref{multconcl}).
We first claim that
\Be \label{claimone}
\|\varphi \sum_{k=n_1}^{n_2} \widehat {K^\ell_k}(s\cdot) \|_{M^p_{\sH_1,\sH_2}} 
\lc \fa_\ell:=\sup_{t>0} \|[\varphi m(t\cdot)]*\widehat {\Psi}_\ell
\|_{M^{p,q}_{\sH_1,\sH_2}} 2^{\ell d(1/p-1/q)}
\Ee
and 
\Be \label{claimonemod}
\|\varphi \sum_{k=n_1}^{n_2} \widehat {K^\ell_k}(s\cdot) \|_{L^\infty_{\sH_1,\sH_2}} 
\lc \fa_{\circ,\ell}:=\sup_{t>0} \|[\varphi m(t\cdot)]*\widehat {\Psi}_\ell
\|_{L^\infty_{\sH_1,\sH_2}},  
\Ee
uniformly in $n_1,n_2$.
We only give the proof of \eqref{claimone} as the the proof of \eqref{claimonemod} is similar but more straightforward. 
To see this we estimate, using dilation invariance,
\begin{align*}
\|\varphi \sum_{k=n_1}^{n_2} \widehat {K^\ell_k}(s\cdot) \|_{M^p_{\sH_1,\sH_2}} 
\le \sum_{k=n_1}^{n_2} 
\|\varphi \widehat\theta(2^{-k} s\cdot) \|_{M^p}
\| [\varphi m(2^k\cdot)]\!*\!\widehat{ \Psi}_\ell\|_{M^p_{\sH_1,\sH_2}}.
\end{align*}
Since $\theta\in \cS(\bbR^d)$ and since all moments of $\eta$ up to order $10d$  vanish we get
\begin{equation}\label{theta moments}
\|\varphi \widehat\theta(2^{-k} s\cdot) \|_{M^p}\lc \min\{ (2^{-k} s)^{10d}, (2^{-k} s)^{-10d}\}.
\end{equation}
Moreover,
\Be\label{eqn:Mp-from-Mpq}\| [\varphi m(2^k\cdot)]\!*\!\widehat \Psi_\ell\|_{M^p_{\sH_1,\sH_2}} 
\lc 2^{\ell d(1/p-1/q)} 
\| [\varphi m(2^k\cdot)]\!*\!\widehat \Psi_\ell\|_{M^{p,q}_{\sH_1,\sH_2}} 
\Ee
and \eqref{claimone} follows  combining the above. To verify \eqref{eqn:Mp-from-Mpq}
we decompose \[f=\sum_\nu f_\nu,\] where $f_\nu=f\bbone_{R_{\ell,\nu}}$ and  the $R_{\ell,\nu}$ form a grid of cubes of side length $2^\ell$. Note that the convolution kernel
$\cK_\ell:=\cF^{-1}[\varphi m(2^k\cdot)\!*\!\widehat \Psi_\ell]$ is supported in the ball of radius $2^\ell$ centered at the origin. Hence, by H\"older's inequality
\begin{align*} 
\|\cK_\ell  * f\|_{L^{p}_{\sH_2}} & = \Big\|\sum_\nu \cK_\ell *f_\nu\Big\|_{L^{p}_{\sH_2}} 
\lc \Big(\sum_\nu \| \cK_\ell* f_\nu\|_{L^{p}_{\sH_2}}^p\Big)^{1/p}
\\
&\lc 2^{\ell d(1/p-1/q)} \Big (\sum_\nu \| \cK_\ell* f_\nu\|_{L^{q}_{\sH_2}} ^p\Big)^{1/p}
\\&\lc 2^{\ell d(1/p-1/q)} \|\widehat{\cK_\ell} \|_{M^{p,q}_{\sH_1,\sH_2}} 
     \Big (\sum_\nu \|  f_\nu\|_{L^{p}_{\sH_1}}^p\Big)^{1/p}
\end{align*}
and since $(\sum_\nu \| f_\nu\|_{L^{p}_{\sH_1}}^p)^{1/p} =\|f\|_{L^{p}_{\sH_1}}$ we get \eqref{eqn:Mp-from-Mpq}.

Straightforward calculation  using \eqref{theta moments} yields
\Be\label{claimtwo}
\sum_{|\alpha|\le d+1} \sum_{k=n_1}^{n_2} \sup_{t>0} \sup_{\xi \in \hat{\bbR}^d} |\partial_\xi^\alpha
(\varphi \widehat {K^\ell_k} (t\xi))|_{\sL(\sH_1,\sH_2)}  \le\fb_\ell:= \|m\|_{L^\infty_{\sH_1,\sH_2}}  2^{\ell(d+1)}
\Ee
uniformly in $n_1,n_2$.
We combine the two estimates \eqref{claimone}, \eqref{claimonemod} and \eqref{claimtwo}  and using  \eqref{multconcl} we get
\begin{equation*}
\|\cT^\ell \|_{L^{p}_{\sH_1}\to L^{p}_{\sH_2}}  \lc (1+\ell)^{|\frac 1p-\frac 12|} \fa_\ell + \fa_{\circ, \ell}. 
\end{equation*}

The $L^q$ estimates are similar. 
For $m(\xi)\in \sL(\sH_1,\sH_2)$ denote by $m^*(\xi)\in \sL(\sH_2^*, \sH_1^*)$ the adjoint. Note that 
\[\|[\varphi m^*(t\cdot)]*\widehat \Psi_\ell\|_{L^\infty_{\sH_2^*,\sH_1^*}}
=
\|[\varphi m(t\cdot)]*\widehat \Psi_\ell\|_{L^\infty_{\sH_1,\sH_2}} \le \fa_{\circ,\ell}.
\]
Since $q'\le p'$ the previous calculation gives   
\begin{align*}&\|\cT^\ell \|_{L^{q}_{\sH_1}\to L^{q}_{\sH_2}} =\|(\cT^\ell)^* \|_{L^{q'}_{\sH_2^*}\to L^{q'}_{\sH_1^*}} 
\\&\lc (1+\ell)^{|\frac 1{q'}-\frac 12|} 2^{\ell d(1/{q'}-1/p')} \sup_{t>0}
\| [\varphi m^*(t\cdot)]\!*\!\widehat \Psi_\ell\|_{M^{q',p'}_{\sH_2^*,\sH_1^*}} 
+  \fa_{\circ,\ell}\\&
\lc
(1+\ell)^{|\frac 1q-\frac 12|} 2^{\ell d(1/p-1/q)} \sup_{t>0}
\| [\varphi m(t\cdot)]\!*\!\widehat \Psi_\ell\|_{M^{p,q}_{\sH_1,\sH_2}} + \fa_{\circ,\ell}\\
& = (1+\ell)^{|\frac 1q-\frac 12|}  \fa_\ell + \fa_{\circ, \ell}.
\notag
\end{align*}
To summarize, 
\begin{multline} \label{eqn:TellLp}
\Big\|\sum_j T^\ell_j\Big\|_{L^p\to L^p}+\Big\|\sum_j T^\ell_j\Big\|_{L^q\to L^q}\\
\le \fa_{\circ, \ell} +\fa_{\ell} ((1+\ell)^{|\frac 1p-\frac 12|} +(1+\ell)^{|\frac 1q-\frac 12|} ).
\end{multline}

To verify the single scale $(p,q)$ condition \eqref{p-q-rescaled} we next  examine the $L^{p}_{\sH_1}\to L^{q}_{\sH_2}$-norms of the convolution operators  
 $\dil_{2^j} T_j^\ell$ with convolution kernels
$2^{jd} K^\ell_{\ell+1-j}(2^j\cdot)$. We have
\[\|\dil_{2^j} T_j^\ell\|_{L^{p}_{\sH_1}\to L^{q}_{\sH_2}} =
\|\widehat {K^\ell_{\ell+1-j}} (2^{-j}\cdot)\|_{M^{p,q}_{\sH_1,\sH_2}}
\]
and 
\[\widehat {K^\ell_{\ell+1-j}} (2^{-j}\xi)=
\widehat \theta(2^{-\ell-1}\xi) [\varphi m(2^{\ell+1-j}\cdot)]*\widehat\Psi_\ell  (2^{-\ell-1}\xi),
\] and we get
\begin{equation*}
\|\widehat {K^\ell_{\ell+1-j}} (2^{-j}\cdot)\|_{M^{p,q}_{\sH_1,\sH_2}} 
\le \|\theta\|_1  2^{(\ell+1)d(1/p-1/q)}\|[\varphi m(2^{\ell+1-j}\cdot)]*\widehat\Psi_\ell 
\|_{M^{p,q}_{\sH_1,\sH_2}}.
\end{equation*}
Hence
\begin{equation} \label{eqn:TellLpLq} 
\sup_j\|\dil_{2^j} T_j^\ell\|_{L^p\to L^q} 
\lesssim \fa_\ell
\end{equation}

Next we turn to the $\varepsilon$-regularity  conditions
\eqref{p-q-rescaled-reg-a}  and \eqref{p-q-rescaled-reg-b}. By translation invariance of the operators $T_j^\ell$ it suffices to verify \eqref{p-q-rescaled-reg-a}.
Using the above formulas for the Fourier transform of
$2^{jd} K^{\ell}_{\ell+1-j}(2^j\cdot)$ we get 
\begin{align*}
&\|(\dil_{2^j} T_j^\ell)\circ\Delta_h\|_{L^{p}_{\sH_1}\to L^{q}_{\sH_2}}
\\
&\quad=%
\|\widehat\theta(2^{-\ell-1}\cdot) 
[(\varphi m(2^{\ell+1-j}\cdot))*\widehat{\Psi}_\ell ]
 (2^{-\ell-1}\cdot)  (e^{i\inn{\cdot}{h}}-1)\|_{M^{p,q} _{\sH_1,\sH_2}}
\\
&\quad\le 2^{(\ell+1)d(1/p-1/q)}
\big\| \widehat\theta [e^{i\inn{2^\ell\cdot}{h}} -1]\big\|_{M^p}
\big\|
[\varphi m(2^{\ell+1-j}\cdot)]*\widehat{\Psi}_\ell\big\|_{M^{p,q}_{\sH_1,\sH_2}}.
\end{align*}
Observe that  for $0<\eps<1$, 
\[
|h|^{-\eps} \big\| \widehat\theta [e^{i\inn{2^\ell\cdot}{h}} -1]\big\|_{M^p}\lc 2^{\ell\eps}
\]
and hence we get 
\begin{equation} 
\label{eqn:TellLpLqreg} 
\sup_{|h|\le 1} |h|^{-\eps}\sup_j\|(\dil_{2^j} T_j^\ell)\circ\Delta_h\|_{L^{p}_{\sH_1}\to L^{q}_{\sH_2}} \lc 2^{\ell \varepsilon} \fa_\ell
\end{equation}
In view of \eqref{eqn:TellLp}, \eqref{eqn:TellLpLq} and \eqref{eqn:TellLpLqreg} we can now apply Theorem \ref{mainthm} and obtain 
\begin{multline*} \|\cT^\ell\|_{\Sp(p,q)}\lc 
\sup_{t>0}
\|[\varphi m(t\cdot)]*\widehat{\Psi}_\ell 
\|_{L^\infty_{\sH_1,\sH_2}} \,+\,\\
\quad \big((1+\ell)^{|\frac 1p-\frac 12|}+(1+\ell)^{|\frac 1q-\frac 12|}+ (1+\ell) \big) 2^{\ell d(\frac 1p-\frac 1q)}
\sup_{t>0}
\|[\varphi m(t\cdot)]*\widehat{\Psi}_\ell 
\|_{M^{p,q}_{\sH_1,\sH_2}}.
\end{multline*}
The desired conclusionthen follows from summing in $\ell \geq 0$.

Finally, to remove the assumption of $m$ being compactly supported we observe that by Lemma \ref{lem:density} it suffices to prove the sparse bound 
\Be\label{eqn:mult-sparse}\int_{\bbR^d} \cF^{-1}[m \widehat f_1](x) f_2(x) dx \le C \big(\cB_\circ[m]+\cB[m]\big)\Lamaxga_{p_1,p_2} (f_1,f_2)\Ee
for $f_i$ in the dense class $\cS_0(\bbR^d, \sH_i)$ of functions whose Fourier transform is compactly supported in $\widehat{\bbR}^d\setminus \{0\}$. But for those functions we have 
$\cF^{-1}[m \widehat f_1]= \cF^{-1} [m_{n_1,n_2} \widehat f_1]$, where 
\[m_{n_1,n_2}=\sum_{k=n_1}^{n_2} \widehat \theta(2^{-k}\xi) \varphi(2^{-k} \xi)m(\xi)\]
with  suitable $n_1, n_2\in \bbZ$ (depending on $f_1$). 
By invariance under multiplication by smooth symbols (see Lemma \ref{lem:mult-by-symbols0}) we have  \Be \notag \label{eqn:comp-vs-global} \sup_{n_1,n_2} \cB[m_{n_1,n_2}] \lc \cB[m] \Ee and an analogous  inequality involving $\cB_\circ[m]$. We  then get  \eqref{eqn:mult-sparse} for  $f_i\in \cS_0(\bbR^d, \sH_i)$, $i=1,2$.
A second application of Lemma \ref{lem:density} yields \eqref{eqn:mult-sparse} for all $f\in L^{p}_{\sH_1}$ and all $f_2\in L^{p'}_{\sH_2^*}$. 
\qed

\section{Sample applications}\label{sec:applications} 

In this section we give a number of specific examples of operators to which Theorem \ref{mainthm} and its consequences can be applied. Some of the resulting sparse bounds are well-known and others appear to be new.

\subsection{Operators generated by compactly supported distributions}\label{subsec:max op Fourier} 
In what follows let 
$\sigma$ be a distribution which is compactly supported 
and let $\sigma_t\equiv \dil_{1/t}\sigma$ denote the $t$-dilate $t^{-d}\sigma(t^{-1}\cdot)$ given by
\[
\inn{\sigma_t}{f}= \inn {\sigma }{f(t\cdot)}.
\] 
Without loss of generality we may assume that the support of $\sigma$ is contained in  $\{x:|x|\le 1\}$, otherwise argue with a  rescaling.

Let \Be A_tf(x)=f*\sigma_t(x)
\label{eqn:Atdef} \Ee 
which is well defined on Schwartz functions as a continous function of $(x,t)$.
  Many interesting operators in harmonic analysis are generated by dilations of such a single compactly supported distribution (often a measure) and we shall be interested in the corresponding maximal and variational operators.  The  domain of the dilation parameter $t$ will be  either $(0,\infty)$ or $[1,2]$ or  a more general subset $E$ of $(0,\infty)$. 

\subsubsection{Maximal functions}\label{sec:MEmax}
We are interested in sparse domination results for the maximal functions,
as defined in \eqref{eqn:MEsigma}, 
$$
M_E^\sigma f(x)= \sup_{t\in E} |A_tf(x)|
$$
where $E\subset (0,\infty)$.

If we assume that $f$ is a Schwartz function then $M_E^\sigma$ is well defined as a measurable function, but for general $L^p$ functions the measurability of $M_E^\sigma$ is a priori not clear unless we assume that $E$ is countable.  In our statements we will restrict ourselves to a priori estimates, but note that in many applications the proof of $L^p$ bounds  shows also a priori estimates for the function $t\mapsto \sigma_t*f(x)$ in suitable subspaces of $C(\bbR)$,  for almost all  $x\in \bbR^d$. This observation then ensures the measurability of the maximal functions for $f$ in the relevant $L^p$ classes. In the general case, 
let $I_{k,n} = [k2^{-n}, (k+1)2^{-n} )$ and pick, for each $(k,n)$ such that 
$E\cap I_{k,n}\neq\emptyset$, a representative $t_{k,n} \in E\cap I_{k,n}$ and let $\widetilde E$ consist of these picked $t_{k,n}$.   Then $\widetilde E$ is countable 
and we have
$M_{E}^\sigma f(x)= M_{\widetilde E} ^\sigma f(x)$
for all $x\in \bbR^d$ and all Schwartz functions  $f$.
Thus one can assume that $E$ is countable without loss of generality.

We shall now discuss sparse domination inequalities for the operator  $M_E^\sigma$.
Recall the local variants $M^\sigma_{E_j}$, with the rescaled sets $E_j\subset [1,2]$ as in \eqref{def:Ej}.
In what follows  recall that \[\La^\fS_{p,q'}(f,\om)=\sum_{Q\in \fS}|Q| \langle f\rangle_{Q,p} \langle \om \rangle_{Q, q'},\] with $\Lamaxga_{p,q'}(f,\om)$ the supremum of all $\Lambda^{\fS}_{p,q'}(f,\om)$ over all $\gamma$-sparse families $\fS$. %

\begin{prop}\label{prop:MEsigma}
(i) Let $1<p\le q<\infty$. Let $\sigma$ be a compactly supported distribution such that
\Be\label{eqn:pp-qq-max}  \|M^\sigma_E\|_{L^p\to L^{p,\infty} } 
+\|M^\sigma_E\|_{L^{q,1}\to L^{q}} <\infty,
\Ee
\Be\label{eqn:pq-max} \sup_{j \in \bbZ} \|M^\sigma_{E_j} \|_{L^p\to L^q}  
<\infty,  \Ee
and assume that there is an $\epsilon>0$ so that for all $\la\ge 2$,
\Be  \label{eqn:pqmodregularity-cond} \|M^\sigma_{E_j} f\|_q\le C \la^{-\eps} \|f\|_p, \quad   f\in \Eann(\la).\Ee
Then for all $f\in L^p$ and all simple non-negative functions $\omega$, we have  the sparse domination inequality 
\Be\inn{ M_E^\sigma f}{\om}\lc \Lamaxga_{p,q'}  (f,\om).
\label{eqn:sparseMEsigma}
\Ee
(ii) Conversely, if $\sigma$  has compact support in $\bbR^d\setminus\{0\}$ then the sparse bound  \eqref{eqn:sparseMEsigma} for $p < q$ implies that
conditions \eqref{eqn:pp-qq-max} and \eqref{eqn:pq-max} hold. 
\end{prop} 

\begin{proof} 
We will apply   {Theorem \ref{thm:ellr}}   with $r=\infty$, $B_2=\ell^\infty(E')$,  where $E'$ is a finite subset of $E$,  and 
\Be\label{Tjformax} T_jf(x,t)=
\begin{cases}\sigma_t*f(x) &\text{ if $t\in E'\cap[2^j, 2^{j+1}) $}
\\0 &\text{ otherwise.} \end{cases}
\Ee
Note that

\begin{subequations} \Be
\label{eqn:sumN1N2max}
S_\infty T f(x) \equiv \sup_{j \in \bbZ} |T_j f(x)|_{B_2}=
M^\sigma_{E'} f(x), 
\Ee and, with $E'_j= 2^{-j}E'\cap[1,2]$, 
\Be\label{eqn:dilTjmax} |\dil_{2^j} T_j f (x)|_{B_2}= M^\sigma_{E'_j} f(x), \qquad j \in \bbZ.\Ee
\end{subequations}

As $\sigma$ is supported in $\{x: |x| \leq 1\}$, the operators $T_j$ satisfy the support condition \eqref{support-assu}. Moreover, \eqref{eqn:pp-qq-max} and \eqref{eqn:sumN1N2max}  guarantee \eqref{bdness-ellr} with $r=\infty$, and similarly  \eqref{eqn:pq-max} and 
\eqref{eqn:dilTjmax} guarantee the single scale $(p,q)$ condition \eqref{p-q-rescaled}. It remains to verify  the single scale  $\varepsilon-$regularity conditions \eqref{p-q-rescaled-reg} for the operators $T_j$. But this {follows from \eqref{eqn:pqmodregularity-cond} and \eqref{eqn:pq-max} via} Lemma 
\ref{lem:regbyFourier} 
and the fact that for translation-invariant operators $T_j$, the conditions \eqref{p-q-rescaled-reg-a} and \eqref{p-q-rescaled-reg-b} are equivalent (alternatively, one can apply Theorem \ref{thm:max weak ass} for maximal functions). All hypotheses in the first part of Theorem \ref{thm:ellr}  are then satisfied and we thus obtain a sparse bound
for the maximal operator $M_{E'}^\sigma$. An application of the monotone convergence theorem then yields the desired sparse bound for $M^\sigma_E$ and concludes the proof of part (i). 

For part (ii)
note that the assumption that $\sigma$ is supported away from the origin corresponds to the strengthened support condition  \eqref{eqn:strengthenedsupp}.
Thus we can deduce part (ii) directly from part (ii) of  Theorem \ref{thm:ellr}.
\end{proof} 

\begin{proof}[Proof of Theorem \ref{thm:MEintro}]
Because of the $L^p\to L^q$ condition on the operators $M^\sigma_{E_j}$ in \eqref{eqn:apriori} and the  $\eps$-regularity assumption \eqref{eq:eps reg cond max fn}, it follows by interpolation that the condition 
\eqref{eqn:pqmodregularity-cond} is satisfied for all $(1/p,1/q)$ in the interior of $\mathscr{L}(\sigma,E)$. Thus Proposition \ref{prop:MEsigma} establishes the sufficiency of the conditions, that is, \eqref{eqn:type-sparse-implication}.
The converse follows immediately from part (ii) of Proposition \ref{prop:MEsigma}.
\end{proof}

Prototypical examples for Proposition \ref{prop:MEsigma}
are the spherical maximal functions where $\sigma$ is the surface measure on the sphere (for $L^p$ bounds see the classical results by Stein \cite{SteinPNAS1976} and Bourgain \cite{BourgainJdA1986}, and for $L^p\to L^q$ bounds see \cite{Schlag1997, SchlagSogge1997}). 
The proposition covers the results by Lacey \cite{laceyJdA19} for the lacunary and  full spherical maximal functions and also the extension to  spherical maximal operators with  suitable  assumptions about various  fractal dimensions of $E$, see \cite{SeegerWaingerWright1995, AHRS, RoosSeeger}.
In this context we note that in \cite{BHRT, GT}, Lacey's approach was used to establish   sparse domination results for  two versions  of  lacunary spherical maximal functions on the Heisenberg group, defined via the  automorphic dilations, and essentially optimal results for the problem considered in  \cite{BHRT} can be obtained by combining the sparse technique developed in that paper with recent $L^p\to L^q$ bounds in \cite{roos-seeger-srivastava}.

One  can also cover more singular  measures with Fourier conditions (as in \cite{DR86}, \cite{DuoandikoetxeaVargas}) and this leads to questions about the precise range of  $L^p$ improving  estimates for the local variants of the  maximal functions.
As an example consider a curve $s\mapsto \gamma(s)$  in $\bbR^3$ with nonvanishing curvature and torsion, and the  measures $\mu_t$ given by
\Be \notag \label{eqn:momentcurve} \inn{f}{\mu_t} =\int f(t\gamma(s) ) \chi(s) ds 
\Ee
with compactly supported $\chi$. A result in \cite{PramanikSeeger2007}, applied in combination with decoupling results in \cite{Wolff2000, BourgainDemeter2015} yields that the maximal operators 
$M_E$ 
are bounded on $L^p(\bbR^3)$ for $p>4$. The optimal result for $p>3$ was recently obtained in \cite{BeltranGuoHickmanSeeger} and in \cite{KoLeeOh}.  Moreover, 
the  analysis in these papers  yield, for the local analogues of these maximal functions (i.e. $E=[1,2]$),   certain $L^p\to L^{q} $ bounds for some $q>p$. It would be very interesting to find precise ranges of $L^p\to L^q$ boundedness of $M_E$ depending on $E$, and corresponding sparse bounds for related global maximal functions. Similar questions can be considered in higher dimensions but the optimal bounds  are currently unknown (for partial results see
\cite{BeltranGuoHickmanSeeger2}, \cite{KoLeeOh2}).

\subsubsection{Variational operators}
Given $1\le r\le \infty$ and a set $E\subset (0,\infty)$ we define the $r$-variation seminorm $|\cdot|_{\mathrm v^r(E)} $ and the $r$-variation norm $|\cdot|_{V^r(E)} $  of a function  $a:E \to \bbC$ 
by 
\begin{align*} 
|a|_{\mathrm v^r(E)} &=
\sup_{M \in \bbN} \sup_{\substack{t_1 < \cdots < t_M \\ t_i \in E}}
\Big( \sum_{i=1}^{M-1} |a(t_{i+1}) -a(t_i)|^r \Big)^{1/r} 
\\
|a|_{V^r(E)} &= 
\sup_{M \in \bbN} \sup_{\substack{t_1 < \cdots < t_M \\ t_i \in E}}
\Big\{ |a(t_1)| +\Big( \sum_{i=1}^{M-1} |a(t_{i+1}) -a(t_i)|^r \Big)^{1/r}\Big\}.
\end{align*} 
Define the $r$-variation operators $\mathrm v^r_E A$, 
$\sV_E^rA $ 
for the family of operators of convolution with $\sigma_t$ by taking the $r$-variation norm in $t$,
\begin{equation}\label{eq:def var op}
\mathrm v^r_E A f(x):=   | \{\sigma_t*f(x) \}|_{\mathrm v^r(E)}, \qquad
\sV^r_E A f(x):=  | \{\sigma_t*f(x) \}|_{V^r(E)}.
\end{equation}
This that the above definition of variation is analogous to the definition in \eqref{eqn:Vrdef} where we considered the $r$-variation for functions integers. The results in \S\ref{sec:squarefct-etc} mostly apply to the situation where the current $E$ is a subset of $\{2^j: j\in \bbZ\}$. For general sets $E \subset (0,\infty)$, we will deduce results directly from Corollary \ref{mainthm-cor} and Theorem \ref{thm:necessarycor}.

As before we may assume that $E$ is countable (as  this does not affect  priori estimates).
Let $E_j\subset [1,2]$  be the rescaled sets as in \eqref{def:Ej}.

\begin{prop}\label{prop:VEsigma}
(i) Let $1<p\le q<\infty$.  
Let $\sigma$ be a compactly supported distribution such that
\Be\label{eqn:pp-qq-var}  \|\sV^r_E A\|_{L^p\to L^{p,\infty} } 
+\|\sV^r_E A\|_{L^{q,1}\to L^{q}} <\infty, \Ee
\Be \label{eqn:pq-var} \sup_{j \in \bbZ} \|\sV^r_{E_j}  A \|_{L^p\to L^q}  
<\infty \Ee
and assume that there is an $\epsilon>0$ so that for all $\la\ge 2$,
\Be  \label{eqn:pqmodregularity-condV} \|\sV^r_{E_j}  f\|_q\le C \la^{-\eps} \|f\|_p, \quad   f\in \Eann(\la).\Ee
Then for all $f\in L^p$ and all simple nonnegative functions $\omega$, we have  the sparse domination inequality 
\Be \label{eqn:sparseVEsigma} \inn{ \sV^r_E f}{\om}\lc \Lamaxga_{p,q'}  (f,\om).
\Ee 
(ii) Conversely,  if $\sigma$  has compact support in $\bbR^d\setminus\{0\}$ then the sparse bound  \eqref{eqn:sparseVEsigma} for $p < q$ implies that
conditions \eqref{eqn:pp-qq-var} and \eqref{eqn:pq-var} hold. 
\end{prop} 

\begin{proof}
We are aiming to apply 
Corollary \ref{mainthm-cor} with $B_2= V^r(E')$ for any finite $E'\subset E$. 
With $T_j f(x,t) $ as in \eqref{Tjformax} 
and  $ E'(N_1,N_2)= E'\cap[2^{N_1}, 2^{N_2+1} ]$,  
we get 
\begin{subequations} \Be\label{eqn:sumN1N2var}\Big |\sum_{j=N_1}^{N_2}T_j f(x)\Big|_{V^r_{E'} }   =  \sV^r_{E'(N_1,N_2) } A f(x) \Ee and  
\Be\label{eqn:dilTjvar} |\dil_{2^j} T_j f(x) |_{V^r_{E'}}=  \sV^r_{2^{-j} E' \cap[1,2] } A f(x).\Ee
\end{subequations}
We need to check the assumptions of Corollary \ref{mainthm-cor}  (i.e. the assumptions of Theorem \ref{mainthm}).
Conditions \eqref{bdness-wt},  
\eqref{bdness-rt} hold by  \eqref{eqn:pp-qq-var} and \eqref{eqn:sumN1N2var},   condition \eqref{p-q-rescaled} holds by \eqref{eqn:pq-var} and \eqref{eqn:dilTjvar} and condition 
\eqref{p-q-rescaled-reg-a} follows from {\eqref{eqn:pq-var}},
\eqref{eqn:pqmodregularity-condV},  and Lemma \ref{lem:regbyFourier}. Condition \eqref{p-q-rescaled-reg-a}  is equivalent with \eqref{p-q-rescaled-reg-b} in the current translation invariant setting.

For the necessity, observe  that the 
assumption that $\sigma$ is supported away from the origin which corresponds to the strengthened support condition in Theorem \ref{thm:necessarycor}.
A sparse bound for $V^r_EA$ implies via \eqref{eqn:sumN1N2var} a sparse bound for $\sum_{j=N_1}^{N_2} T_j$ for any pair of integers $N_1 \leq N_2$.
 We apply Theorem \ref{thm:necessarycor} and obtain via \eqref{eqn:sumN1N2var}
and \eqref{eqn:dilTjvar} 
that
\begin{align*} &\| \sV^r_{E'(N_1,N_2)} A\|_{L^p\to L^{p,\infty} } + \|\sV^r_{E'(N_1,N_2) } A\|_{L^{q,1}\to L^q} \le C,
\\
&\sup_{N_1\le j\le N_2} \|\sV^r_{2^{-j}E'\cap [1,2]}  A \|_{L^p\to L^q} \le C, 
\end{align*}
with the constant $C$ independent of $N_1, N_2$ and the particular finite subset $E'$ of $E$. Applications of the monotone convergence theorem then yield the asserted necessary conditions for $\sV^r_EA$, that is, \eqref{eqn:pp-qq-var} and \eqref{eqn:pq-var}.
\end{proof} 

Proposition \ref{prop:VEsigma}  can be applied to obtain a sparse domination inequality for the $r$-variation operator associated with the spherical means in $\bbR^d$. For the necessary global $L^p\to L^p$ bounds see \cite{JonesSeegerWright} and for  $L^p\to L^q$ bounds for the local variation operators we refer to the recent paper \cite{BGORSS}.   This addresses a question posed in \cite{laceyJdA19} and \cite{aimPL}.

\begin{rem}\label{rem:longshort}
In verifying $L^p\to L^{p,\infty}$ and $L^{q,1}\to L^q$ assumptions for the variation operators it is 
(as shown in \cite{JonesKaufmanRosenblattWierdl, JonesSeegerWright}) 
often advantageous to   write
  $  V^r_EA f(x) \le V^r_{\mathrm{dyad}}A f(x) + V^r_{E,\mathrm{sh}}A f(x) $
where \[V^r_{\mathrm{dyad}} A f(x):= \sV^r_{2^{(\bbZ)}} Af(x) \] is the standard variation norm over $2^{(\bbZ)}:=\{2^j:j\in \bbZ\}$, labeled  the {\it dyadic}  or  {\it long variation operator}
and where 
\[V^r_{E,\mathrm{sh}} A f(x):=
    \Big(\sum_{j\in \bbZ}|\mathrm v^r_{E\cap[2^j, 2^{j+1}]} A f(x)|^r\Big)^{1/r},
\] 
is the so-called {\it short variation operator} which uses  only variation seminorms  over $E$ within dyadic intervals. The $L^p$-boundedness of the long variation operators is usually reduced to Lepingle's theorem \cite{Lepingle} (which requires $r>2$) while the short variation operator is often estimated  using a Sobolev embedding inequality (see \cite{JonesKaufmanRosenblattWierdl}, \cite{JonesSeegerWright}). 
We note that it is possible to prove results analogous to Proposition \ref{prop:VEsigma} for the long variation operator and the short variation operators individually as direct consequences of Theorems \ref{thm:Vr} and \ref{thm:ellr} respectively; the details are left to the reader.
\end{rem}
    
\subsubsection{Lacunary maximal functions for convolutions associated with the wave equation}\label{sec:wavelac} %
In this 
section we consider  a   maximal function   generated by convolutions with dilates of a tempered distribution,
which is  not compactly supported (but still concentrated  on a compact set). This class is  associated with
$L^p$ regularity results for solutions of  the wave equation.
For  both simplicity and definiteness of results  we shall only consider  a  lacunary version, but the argument to deduce the sparse bound extends to other sets of dilations and also to variational variants (for which Lemma \ref{lem:nonlocalerror} would be useful to treat nonlocal error terms).

For $\beta>0$ define 
\[ m_\beta(\xi)= \frac{\cos|\xi|}{(1+|\xi|^2)^{\beta/2} }
\]
and let
\[\cM^\beta_\lac f(x) =\sup_{k\in \bbZ} |m_\beta(2^k D) f(x) |.\]
It was shown by Peral in \cite{Peral1980} and Miyachi in \cite{Miyachi-wave-1980} that $m_\beta(D)$ is bounded on $L^p$ for $\beta\ge(d-1)|1/p-1/2|$, $1<p<\infty$.
$L^p\to L^q$ results for $m_\beta$ go back to  \cite{Strichartz1970,littman, Brenner1975}; it is known that $m_\beta(D):L^p\to L^q$ is bounded if either  %
\begin{enumerate}
    \item[(W)] $1<p\le 2, \, p\le q\le p', \,\beta \geq  (d-1)\big(\frac 1p-\frac 12) +\frac 1p-\frac 1q$, or
\item[(W$'$)] $ 1<p<\infty,  \, \max\{p,p'\} \le q<\infty, \,
    \beta \geq  (d-1)\big(\frac 12-\frac 1q) +\frac 1p-\frac 1q$.
\end{enumerate}

Note that (W$'$) follows from (W) by duality. Moreover it can be shown that $\cM^\beta_\lac$ is bounded on $L^p$ for $\beta>(d-1)|1/p-1/2|$ via a single scale analysis, and either Littlewood--Paley theory for $p \geq 2$ or the result stated in \S \ref{sec:combinescales-mult} for $1 < p < 2$.

We have the following sparse bound for $\cM^\beta_\lac$ in the non-endpoint case.
\begin{prop}
Suppose $1<p\le q<\infty$ and that one of the following two  conditions holds.
\begin{enumerate}
    \item[\upshape ($\mathrm{W}_*$)] $1<p\le 2, \, p\le q\le p', \,\beta> (d-1)\big(\frac 1p-\frac 12) +\frac 1p-\frac 1q$.
\item[\upshape ($\mathrm{W}_*^\prime$)] $ 1<p<\infty,  \, \max\{p,p'\} \le q<\infty, \,
    \beta >  (d-1)\big(\frac 12-\frac 1q) +\frac 1p-\frac 1q$.
\end{enumerate}
Then 
 $\cM_\lac^\beta \in \mathrm{Sp}(p,q')$.
\end{prop}

\begin{proof}
Let
$K= \cF^{-1}[m_\beta(2\cdot)]$  so that the singular support of $K$ is $\{x:|x|=1/2\}$.  Let  $K_{\ell,0}=K*\eta_\ell= \mathcal{F}^{-1}[ m_\beta \widehat\eta_\ell] $, with $\eta_\ell$ defined as in \eqref{eqn:resofidmod}, 
and split $K_{\ell,0}*f=\sA_\ell f+\sR_\ell f$ where the convolution kernel $R_\ell$ of $\sR_\ell $ is supported in  $\{x: |x|\geq 1\}$.
The maximal function associated to $R_\ell$ is dominated by $2^{-\ell N} $ times the Hardy--Littlewood maximal function of $f$, similarly the maximal functions associated to $A_\ell $ are controlled by the Hardy--Littlewood maximal function for small $\ell$  and therefore satisfy a $(p,q')$ sparse bound by \S \ref{subsec:sparse HL}. 
We use the notation $A_{\ell,0}$, $R_{\ell,0}$ for the convolution kernels of $\sA_\ell$ and $\sR_\ell$. Set 
$K_{\ell,k}= 2^{-kd} K_{\ell,0}(2^{-k}\cdot)$, and similarly define the kernels $A_{\ell,k}$ and $R_{\ell,k}$.

By the $L^p\to L^p$ result for $m_\beta(D)$
together with the multiplier result mentioned in \S\ref{sec:combinescales-mult} one can easily derive for $\ell>0$ and any $\varepsilon>0$ 
\[
\Big \| \Big(\sum_{k\in \bbZ} | K_{\ell,k} *f |^2\Big)^{1/2} \Big\|_p \lc 2^{\ell ( (d-1)|\frac 1p-\frac 12|+\eps-\beta )} \|f\|_p
\]
for all $1 < p <\infty$ which of course implies
\Be \notag \label{eqn:p-pforwave}
\big \| \sup_{k\in \bbZ} | K_{\ell,k} *f |\big\|_p \lc 2^{\ell ( (d-1)|\frac 1p-\frac 12|+\eps-\beta )} \|f\|_p.
\Ee

We also have the single scale results  %
\Be \label{q less than p'}
2^{\ell \beta}
\| K_{\ell,0} * f\|_q 
\lc 2^{\ell ((d-1)(\frac 1p-\frac 12)+\frac 1p-\frac 1q)} \| f \|_p
\Ee
if 
$1<p\le 2$, $ p\le q\le p'$, and
\Be
\label{q bigger than p'}
2^{\ell\beta}\|K_{\ell,0} * f\|_q 
\lc 2^{\ell ((d-1)(\frac 12-\frac 1q)+\frac 1p-\frac 1q)} \| f \|_p
\Ee
if $1<p\le 2$, $p'<q<\infty$ or  $2\le p<\infty$, $p\le q<\infty$.

By the above mentioned bounds for the operator $\sR_\ell$ and the lacunary maximal operator generated by it we can replace $K_{\ell,k}$ and $K_{\ell,0}$ by
$A_{\ell,k}$ and $A_{\ell,0}$, respectively,

Note that the exponents for $L^p$-boundedness 
and for $L^q$ boundedness, i.e.
$(d-1)|1/p-1/2|$, $(d-1)|1/q-1/2|$ 
are not larger than  the exponents in the displayed inequalities \eqref{q less than p'} and \eqref{q bigger than p'} in their respective ranges.
 An application of Proposition \ref{prop:MEsigma} gives  the desired sparse results for
 the maximal function generated by the $A_{\ell,k}$ and then also for the maximal function generated by convolution with $K_{\ell,k}$. Summing in $\ell$ we can complete the proof of the proposition.
\end{proof}

\begin{remarka} The multiplier $m_\beta$ can be replaced by other variants such as
\[m_{\beta,1}(\xi)  = \frac{\sin|\xi|}{|\xi|} \frac{1}{ (1+|\xi|^2)^{(\beta-1)/2}},
\qquad 
m_{\beta,2}(\xi) = \frac{ J_{\beta- 1/2}(|\xi|)} {|\xi|^{\beta-1/2}}\, .
\]
\end{remarka}

\subsection{General classes of multipliers} \label{sec:genclassesFM} 
It is well known that the classical Mikhlin--H\"ormander multiplier theorem \cite{hormander1960, Ste70}  can be interpolated with  the $L^2$-estimate for multiplier transformations $m(D)$ with bounded multipliers \cite{LionsPeetre1964, Littman1965}. In particular one gets for $1<p\le 2$,
  \Be\label{eqn:Hinterpolated}  \|m\|_{M^p}  \lc \sup_{t>0} \|\phi m(t\cdot)\|_{L^r_\alpha}, 
  \quad 1/r= 1/p- 1/2, \quad \alpha>d/r\,, 
  \Ee where $\|g\|_{L^r_\alpha} =\|(1+|D|^2)^{\alpha/2}g\|_r$  and $\phi$ is a nontrivial radial function supported with compact support away from the origin.
  
  We give a sparse bound for this class of multipliers.

  \begin{prop}\label{prop:Hinterpol}
  Let $1<p\le 2$, $1/r=1/p-1/2$ 
  and let  $m$ satisfy    \Be \label{eqn:Hcond}  \sup_{t>0} \|\phi m(t\cdot) \|_{L^r_\alpha} \le A.\Ee 
  Suppose one of the following holds:
\begin{enumerate}[\upshape(i)]
    \item  $1 < p\le q\le 2$,  and  $\alpha> d(1/p-1/2)$.
    
    \item  $2\le q < \infty$ and $\alpha>d(1/p-1/q)$.
\end{enumerate}
\noindent   Then
  \[ \|m(D)\|_{\Sp_\ga(p,q') } \lc_{p,r,q,\alpha,\gamma} A.\]
  
  \end{prop}
  
  \begin{proof} 
  We deduce this result from Theorem \ref{thm:localtoglobal-sparse-s}. Observe the inequality 
  \[\|g\|_{M^{p,2} } \le \|g\|_{L^r},\qquad 1/r=1/p-1/2,\]
  valid for $1\le p\le 2$ %
  which follows by interpolation from the standard cases  $p=1$ and $p=2$.
  In view of the embedding  $B^0_{r,1} \hookrightarrow L^r$ (see \cite{triebel1983} for the definition and properties of Besov spaces) we get, for $1\le p\le 2,$
  \Be\label{eqn:Mp2} \|g\|_{M^{p,2} } \lc \| g\|_{B_{r,1}^0}, \qquad 1/r=1/p-1/2.
  \Ee 
  Interpolating Bernstein's theorem %
  $B^{d/2}_{2,1}\hookrightarrow \widehat {L^1}$ (which follows from the Cauchy--Schwarz inequality and Plancherel's theorem) %
  with the embedding $B^0_{\infty,1} ~\hookrightarrow~ L^\infty$, {we also have for $1 \leq p \leq 2$},
  \Be\label{eqn:Mpp}
  \|g\|_{M^{p,p} } \lc \|g\|_{B^{d/r}_{r,1} } , \qquad  1/r=1/p-1/2.
  \end{equation}
  A further interpolation of \eqref{eqn:Mp2} and \eqref{eqn:Mpp} yields  
  for ${1 \leq }p\le q\le 2$
  \[ \|g\|_{M^{p,q} } \lc \|g\|_{B^{d(\frac 1q-\frac 12)}_{r,1} } \,,
 \qquad 1/r=1/p-1/2.
 \]
  Finally, we have for $M\ge s$,  the well-known inequality
  \[ \|\Delta_h^M g
  \|_{ B^{\beta}_{r,1} }   
  \lc |h|^s  \|g\|_{B^{s+\beta}_{r,1}}
  \] which we shall use for $\beta=d(1/q-1/2)$ and which can be deduced from standard $L^1$-convolution inequalities.
  
  Now  let $r=2p/(2-p)$, \textit{i.e.}  $1/r=1/p-1/2$. Applying the above inequalities to $g=\phi m(t \cdot)$ we get for $M\ge s$, 
  \[
  \sup_{|h|\le 1} |h|^{-s} \|\Delta_h^M[\phi m(t\cdot) ]\|_{M^{p,q} } \lc 
  \|\phi m(t\cdot) \|_{B^{d(\frac 1q-\frac 12) +s}_{r,1}} 
  \]
  Now since $\alpha> d(1/p-1/2)$ we can find $s>d(1/p-1/q)$ such that 
  $$\alpha>d(1/q-1/2) +s>d(1/p-1/2).$$  Thus if  $\alpha$  is as in the display, then 
  $L^r_\alpha\hookrightarrow B^{d(1/q-1/2)+s}_{r,1} $ and 
  an application of Theorem 
  \ref{thm:localtoglobal-sparse-s} yields the  sparse bound stated in part (i).

  For part (ii) let $2\le q{< \infty}$ and observe that 
  \Be\label{eqn:compmult} \|g\|_{M^{p,q} }\lc_K  \|g\|_{M^{p,2}} \quad \text{ if } \supp(g)\subset K, \quad K \text{ compact.}
  \Ee To see this take a Schwartz function $\upsilon$ whose Fourier transform  equals $1$ on $K$ and observe  that by Young's inequality convolution with $\upsilon$ maps $L^2$ into $L^q$. We  see from \eqref{eqn:compmult} and \eqref{eqn:Mp2} that  for such compactly supported $g$ and $M>s$,
  \[|h|^{-s} \|\Delta_h^M g\|_{M^{p,q} }\lc |h|^{-s} \|\Delta^M_h  g\|_{B^0_{r,1}} \lc  \|g\|_{B^s_{r,1}}. \]  This we use  for $g=\phi m(t\cdot)$ and ${\alpha >}s>d(1/p-1/q)$. Then part (ii) follows by {the embedding $L^r_\alpha \hookrightarrow B_{r,1}^s$ and} an application of Theorem \ref{thm:localtoglobal-sparse-s}. 
   \end{proof}
  
  \begin{rem} \label{rem:dualityrem}
  The assumption $p\le 2$ is not a significant restriction. Indeed observe that by definition of the sparse operator classes  we have $T\in \Sp(p_1,p_2) $ if and only if $T^*\in \Sp(p_2,p_1)$. For multiplier transformations we have 
  $m(D)^*=m(-D)$ and $m(-D) f(-x)= m(D) [f(-\cdot)] (x)$ which implies that
  $m(D)\in \Sp (p_1,p_2) $ if and only if $m(D)\in \Sp(p_2,p_1)$.
  
  We can draw two conclusions from this duality argument. First, the range $1\le p\le 2$, $ q\ge p'$ in Proposition \ref{prop:Hinterpol} could be deduced
  from the result in the range $1\le p\le 2$, $2\le q\le p'$. Second, the result in Proposition \ref{prop:Hinterpol} also implies a result in the range  $2\le p\le q<\infty$. Namely, in this case, if $1/r=1/2-1/q$ and $\alpha>d(1/2-1/q)$ then  one gets  
  $m(D)\in\Sp(p,q')$ under the assumption  \eqref{eqn:Hcond}.
    \end{rem}
\subsubsection{Miyachi classes and subdyadic H\"ormander conditions}  We now  discuss some consequences for multiplier classes considered by Miyachi \cite{Miyachi1980} and their corresponding versions under a subdyadic H\"ormander-type formulation \cite{BeltranBennett}. 
Given $a>0, b \in \bbR$, let $\mathrm{Miy}(a,b)$  denote the class of smooth functions $m:\bbR^d \to \mathbb{C}$ supported on $\{\xi : |\xi| \geq 1\}$ and satisfying the differential inequalities
\Be\label{eqn:miyachibounds}
|\partial^\iota m(\xi) | \lesssim_{\iota} |\xi|^{-b + |\iota|(a-1)}
\Ee
for all $|\xi| \geq 1$ and all multiindices $\iota \in \bbN_0^d$ satisfying  $|\iota| \leq \lfloor d/2 \rfloor +1$.
  The oscillatory multipliers $m_{a,b}$ defined below in 
  \eqref{eqn:oscmultdef} are considered model cases, at least in regards to the  $L^p\to L^p$ boundedness properties.  It is known that multipliers in $\mathrm{Miy}(a,b)$ belong to $M^p$ whenever $b\ge ad|1/p-1/2|$ and $1<p<\infty$, see \cite{FeffermanStein-Acta,Miyachi1980}.
  It has also been observed that these endpoint results are special cases of H\"ormander-type multiplier theorems involving certain endpoint Besov spaces, see \cite{BaernsteinSawyer, SeegerStudia90}.
  Sparse bounds for multipliers in $\mathrm{Miy}(a,b)$ in the non-endpoint range $b>ad|1/p-1/2|$ were obtained by Cladek and the first author in \cite{beltran-cladek} via a single scale analysis, under the additional assumption that \eqref{eqn:miyachibounds} hold for all multiindices $\iota \in \bbN_0^d$. We note that in the range $0<a<1$ they also extended these results  to larger closely related  classes of  pseudo-differential operators, \cf.  \cite{fefferman-pseudo, beltran-cladek}.
  
    The subdyadic H\"ormander-type classes, also   extending the class $\mathrm{Miy}(a,b)$  are  obtained by replacing  the pointwise condition \eqref{eqn:miyachibounds} by \begin{equation}\label{HormSD}
\sup_{B}\; \dist(B,0)^{b+(1-a)|\iota|} \Big( \frac{1}{|B|}\int_B |D^\iota m(\xi)|^2 d\xi \Big)^{1/2}< \infty
\end{equation}
for all $\iota \in \bbN^d_0$ with $|\iota|\leq \lfloor d/2 \rfloor +1$. Here the supremum is taken over all euclidean balls $B$ in $\mathbb{R}^d$ with $\dist(B,0)\geq 1$
such that $r(B) \sim \dist(B,0)^{1-a}$, where $r(B)$ denotes the radius of $B$. This class was considered in \cite{BeltranBennett} which contains sharp weighted inequalities of Fefferman--Stein type that can be used to recover the sharp $L^p$ estimates.
In \cite[\S3]{beltran-cladek} the question was  raised  whether the results on sparse bounds for multiplier transformations in the Miyachi class can be extended to multipliers satisfying a subdyadic condition above, in the sense that it is sufficient to assume that \eqref{eqn:miyachibounds} or \eqref{HormSD} hold for all $|\iota| \leq \lfloor d/2 \rfloor + 1$ rather than for all $\iota \in \bbN_0
^d$. We shall see that this is the case, and that such and more general multi-scale results can be obtained from Proposition \ref{prop:Hinterpol}. The following  simple observation will be helpful; note that condition \eqref{HormSD} (and therefore \eqref{eqn:miyachibounds}) implies \eqref{eqn:Horm type cond}.

  \begin{lem} Let $a>0$,  $2\le r<\infty$, $b>ad/r$.
  Suppose $m_k$ are supported  in $ \{\xi: |\xi|\ge 1\} $ and suppose that there is a constant $C$ such that
  \begin{equation}\label{eqn:Horm type cond}
  \sup_{t>1}  t^{b-a|\imath|}\Big( t^{-d} \int_{t/2 \le |\xi|\le 2t}   \big[|\xi|^{|\imath|} \partial^\imath m_k(\xi)|\big]^r d\xi\Big)^{1/r} \le C
  \end{equation}
  for all multiindices $\imath$ with $|\imath|\le \lfloor d/r\rfloor+1$ and for all $k\in \bbZ$.  Then  the family $\{m_k\}$ satisfies condition  
  \Be\label{eqn:miyachi-limited}  \sup_{k\in \bbZ} \sup_{t>0} t^{b-a\alpha} \|\phi m_k(t\cdot) \|_{L^r_\alpha} <\infty.\Ee
  \end{lem}
  
  \begin{proof}
  A change of variable shows that the condition \eqref{eqn:Horm type cond} is equivalent to
   \[ \big\| \partial^\imath [ \phi m_k(s\cdot)]\big\|_{r} \lc s^{a|\imath|-b} \]
  for all multiindices $\imath$ with $|\imath|\le \floor{d/r}+1$.
  Pick $\alpha\in (d/r, \floor{d/r}+1)$ such that $\alpha a<b$. Then the condition implies
  \[\sup_s\|\phi m_k(s\cdot) \|_{L^r_\alpha} s^{b-a\alpha} <\infty,\]
  which implies \eqref{eqn:miyachi-limited} in view of the assumption on the supports  since $\alpha>d/r$.
  \end{proof}

We shall now formulate a result for families of multipliers satisfying condition \eqref{eqn:miyachi-limited}.
For simplicity of our statements, we consider only the case $p\le 2$ and argue 
by duality for $p>2$ (see Remark \ref{rem:dualityrem}).

  \begin{prop} \label{Miyachiprop}
  Let $1<p\le 2$, 
  $r=\tfrac{2p}{2-p}$
  (i.e. $\tfrac 1r=\tfrac 1p-\tfrac 12$),  and  let, for $k\in \bbZ$, $m_k$  be supported in $\{\xi:|\xi|\ge 1\} $.  Let $a, b\ge 0$ such that $b>a d(\tfrac 1p-\tfrac 12)$ and suppose that either
  \begin{enumerate}[\upshape (i)]
  \item  ${1<}p\le q\le 2$ and $b>ad(\tfrac 1p-\tfrac 12)$,   or 
  \item  $2\le q{<\infty}$ and $b>ad(\tfrac 1p-\tfrac 1q)$. 
 \end{enumerate} 
\noindent  Let $\alpha>d(\tfrac 1p-\tfrac 12)$ 
  and   assume 
  $\sup_{k\in \bbZ} \sup_{t>0} t^{b-a\alpha} \|\phi m_k(t\cdot) \|_{L^r_\alpha} <\infty.$
  
      Then  
     $m:= \sum_{k\in \bbZ} m_k(2^k\cdot) \in M_p$ and $m(D) \in \Sp(p,q')$
  \end{prop}

  \begin{proof}
  We split, by a dyadic decomposition $m_k(\xi)=\sum_{n=1}^\infty m_{k,n}(\xi) $ where $m_{k,n}$ is supported in an annulus  $\{\xi: |\xi|\approx 2^n\} $, for all $k\in \bbZ$; in fact we can set $m_{k,n}(\xi)=m_k(\xi)\widehat\eta_n(\xi)$ with $\widehat \eta_n$ as in \eqref{eqn:suppetaell}. Observe that $m_{k}\widehat\eta_0=0$ by the support properties of $m_k$ and $\widehat\eta_0$. Now form $m^n(\xi)=\sum_{k\in \bbZ} m_{k,n}(2^k \xi)$.
  We wish to apply Proposition  \ref{prop:Hinterpol} to $m^n$, for every $n\ge 0$.
  
  Fix any $t>0$ and use the assumption to compute
  \begin{align*}
  \|\phi m^n(t\cdot )\|_{L^r_\alpha} \lc \sum_{\substack {k\in \bbZ: \\c_1 2^n\le  2^k t\le C_1 2^n}} \|\phi m_{k,n}(2^k t\cdot) \|_{L^r_\alpha} \lc 2^{-n(b-a\alpha)}. 
  \end{align*}
  By  \eqref{eqn:Hinterpolated} we obtain $\| m^n \|_{M^p} \lesssim 2^{-n(b-a \alpha)}$ and similarly,  by Proposition \ref{prop:Hinterpol} we obtain
  $\|m^n(D) \|_{\Sp_\ga(p,q')} \lc 2^{-n(b-a\alpha)}  $. The desired bounds follow by summing in $n$ as $\alpha<b/a$. 
  \end{proof}

As a consequence we can obtain a sparse bound  for the lacunary   maximal function $\sup_{k}| m(2^k D) f|$ and indeed a  square function that dominates it.

  \begin{cor} \label{cor:laccor} Let $p,r,q,a,b$ as in Proposition \ref{Miyachiprop}.  Let $m$ be supported in $\{\xi: |\xi|\ge 1\}$ satisfying  
  $\sup_{t>0} t^{b-a \alpha}\|\phi m(t\cdot)\|_{L^r_\alpha} < \infty$. Then we have the $(p,q')$-sparse bound
  \[\int_{\bbR^d} \Big(\sum_{k\in\bbZ} |m(2^k D) f(x)|^2\Big)^{1/2} \om(x) dx \lc \Lamaxga_{p,q'} (f,\om) 
 .\]
 \end{cor}
 
 \begin{proof} 
  Consider  the multiplier 
  $m_v(\xi)=\sum_{k\in \bbZ} r_k(v) m(2^k\xi)$ where 
  $(r_k)_{k\in \bbN}$ denotes the sequence of Rademacher functions defined on the unit interval.  Then by Proposition \ref{Miyachiprop} applied to $m_k(\xi)=r_k(v) m(\xi)$ we obtain
   \Be \label{eqn:m_v}\Big| \int_{\bbR^d} m_v(D)  f_1 (x)f_2(x) dx\Big|\lc \Lamaxga_{p,q'} (f_1,f_2) ,
   \Ee
   with the implicit constant independent of $v$. Let $u_v(x) = \frac{m_v(D) f(x)} {|m_v(D) f(x) |} $ so that $u_v$ is unimodular, and we also get by \eqref{eqn:m_v} with $f_2=\om u_v$
  \begin{align*} \int_{\bbR^d} |m_v(D)  f |\om(x) dx & = \int_{\bbR^d} m_v(D) f(x)\,\omega(x) u_v(x) dx \\
    &\lc \Lamaxga_{p,q'} (f,u_v \om) =\Lamaxga_{p,q'} (f,\om).
   \end{align*}
  Integrating in $v$ and using Fubini's theorem and 
  Khinchine's inequality, one obtains
  \begin{align*}
      &\int_{\bbR^d}\Big(\sum_{k\in \bbZ}|m_k(2^k D) f|^2\Big)^{1/2} \om(x) dx\lc 
      \int_{\bbR^d} \int_0^1 |m_v(D) f(x) | \, dv\, \om(x) \,dx
      \\
        &=\int_0^1 \int_{\bbR^d} |m_v(D) f(x) | \,  \om(x) \,dx \, dv\,\lc  \Lamaxga_{p,q'} (f,\om)
   \end{align*}
   and the proof is complete.
  \end{proof}

  \begin{remarka}
   Similar results can be obtained for versions of the previous multiplier classes if $a<0$ and $m$ is supported in $\{ \xi : |\xi| \leq 1\}$. We omit the statements.
  \end{remarka}
  
  \subsubsection{Multiscale variants of oscillatory multipliers}\label{sec:oscmultiscale}
 Given $a>0$, $a \neq 1$, $b \in \bbR$,  consider  the oscillatory Fourier multipliers \Be\label{eqn:oscmultdef}
m_{a,b} (\xi)=\chi_{\infty} (\xi) |\xi|^{-b} e^{i|\xi|^a}, 
\Ee where $\chi_\infty \in C^\infty (\bbR^d)$ is such that $\chi(\xi)=0$ for $|\xi|\le 1$ and $\chi_\infty(\xi)=1 $ for $|\xi|\ge 2$. As already mentioned the operators $m_{a,b}(D)$ are sometimes  considered model cases of the class $\mathrm{Miy}(a,b)$, known to be bounded on $L^p$ if and only if $b\ge a d|1/p-1/2|$ and $1 < p < \infty$; see %
\cite{Stein1971-ann}, \cite{FeffermanStein-Acta}, \cite{Miyachi1980}. This result is sharp when $a\neq 1$; the case $a=1$ forms an exceptional case  corresponding to the wave multipliers considered in \S\ref{sec:wavelac}; we exclude it in this section.

Given a sequence $(c_k)_{k\in \bbZ}$ with $|c_k|\le 1$ we form the multiscale variant
\Be\label{eqn:mab-mscale}
m(\xi) =\sum_{k\in \bbZ} c_k m_{a,b}(2^k\xi)
\Ee which is bounded on $L^p$ for $b>ad|1/p-1/2|$.  
Proposition \ref{Miyachiprop} shows that for $1<p\le 2$ we have  $m(D)\in \Sp(p,2)$  for $q\le 2$, but in order to get a $\Sp(p,q')$ bound for  $q>2$ we  had to impose the  more restrictive condition  $b>ad(1/p-1/q)$. \
We show that this estimate can be improved, in particular an  additional restriction is not necessary for $q\le p'$ and  in this range we can upgrade the $\Sp(p,2)$ bound  to an $\Sp(p,p) $ bound for the  multipliers in \eqref{eqn:mab-mscale} (see Figure \ref{fig:sparse region osc mult}).

This improvement relies on special features of the multipliers $m_{a,b}$ which are not shared by a general multiplier in the class $\mathrm{Miy}(a,b)$.
Unlike in the proof of Proposition \ref{Miyachiprop} we can no longer rely on analyzing the problem on the multiplier side. Instead we have to analyze Schwartz kernels and employ stationary phase estimates, taking advantage of the fact that  the Hessian of the phase function $\xi\mapsto |\xi|^a$ is nondegenerate when  $a\neq 1$, $a>0$. Incidentally, this also reveals that the $m_{a,b}$ satisfy better $L^p \to L^q$ mapping properties than a general multiplier in $\mathrm{Miy}(a,b)$ when $1 < p \leq 2$, $2 < q \leq p'$. It is therefore more natural to base the proof directly on Theorem \ref{thm:sparsemult} rather than on  the formulation in Theorem \ref{thm:localtoglobal-sparse-s}.

\begin{figure}
\begin{tikzpicture}[scale=2]

\begin{scope}[scale=1.8]

\draw[thick,->] (0,0) -- (1.1,0) node[below] {\small{$ \frac 1 p$}};
\draw[thick,->] (0,0) -- (0,1.1) node[left] {\small{$ \frac{1}{q'}$}};

\draw[loosely dashed] (0,1) -- (1.,1.)  -- (1.,0); 
\draw[loosely dashed] (0,1) -- (1/6,5/6);
\draw[loosely dashed] (5/6,1/6) -- (1,0);
\draw (1/6,5/6) -- (0.5,0.5)  -- (5/6,1/6);

\draw (.5,.02) -- (.5,-.02) node[below] {\small{$ \tfrac 12$}};
\draw (.02,.5) -- (-.02,.5) node[left] {\small{$ \tfrac 12$}};

\draw (5/6,1/6) circle (.04em) ;  

\draw (5/6,.02) -- (5/6, -.02) node[below] {\small{$\tfrac{1}{2} + \tfrac{b}{ad}$}} ;  

\draw (1/6,5/6) circle (.04em) ;

\draw[loosely dashed]  (0.5,0.5)  -- (1,1);

\draw[densely dotted] (5/6,1/6) -- (5/6,1/2);
\draw[densely dotted] (1/6,5/6) -- (1/2,5/6);
\draw[densely dotted] (5/6,1/2) -- (1/2,5/6); 

\fill[pattern=north west lines, pattern color=gray] (5/6,1/6) -- (5/6,1/2) -- (1/2, 5/6) --  (1/6,5/6) -- (5/6,1/6);

\draw (5/6,1/2) circle (.04em)  ;
\draw (1/2,5/6) circle (.04em) ;

\draw (2/3,2/3) circle (.04em)  ;

\draw (.02, 2/3) -- (-.02, 2/3) node[left] {\small{$\tfrac{1}{2} + \tfrac{b}{2ad}$}} ;  

\begin{scope}[xshift=50] 
\draw[thick,->] (0,0) -- (1.1,0) node[below] {\small{$ \frac 1 p$}};
\draw[thick,->] (0,0) -- (0,1.1) node[left] {\small{$ \frac{1}{q'}$}};

\draw[loosely dashed] (0,1) -- (1.,1.)  -- (1.,0); 
\draw[loosely dashed] (0,1) -- (1/6,5/6);
\draw[loosely dashed] (5/6,1/6) -- (1,0);
\draw (1/6,5/6) -- (0.5,0.5)  -- (5/6,1/6);

\draw (.5,.02) -- (.5,-.02) node[below] {\small{$ \tfrac 12$}};
\draw (.02,.5) -- (-.02,.5) node[left] {\small{$ \tfrac 12$}};

\draw (5/6,1/6) circle (.04em) ;  

\draw (5/6,.02) -- (5/6, -.02) node[below] {\small{$\tfrac{1}{2} + \tfrac{b}{ad}$}} ;  

\draw (1/6,5/6) circle (.04em) ;

\draw[loosely dashed]  (0.5,0.5)  -- (1,1);

\draw[densely dotted] (5/6,1/6) -- (5/6,5/6)--(1/6,5/6);

\fill[pattern=north west lines, pattern color=gray] (5/6,1/6) -- (5/6,5/6) -- (1/6,5/6) -- (5/6,1/6);

\draw (5/6,5/6) circle (.04em)  ;

\end{scope}
\end{scope}

\end{tikzpicture}

\caption{Sparse bounds for a general multiplier in $\mathrm{Miy}(a,b)$ (left) and for the oscillatory multipliers $m_{a,b}$ (right) for given $a,b >0$. The condition (ii) in Proposition \ref{Miyachiprop} can be relaxed for the specific $m_{a,b}$ (Proposition \ref{prop:oscmult-msc}).}
\label{fig:sparse region osc mult}

\end{figure}
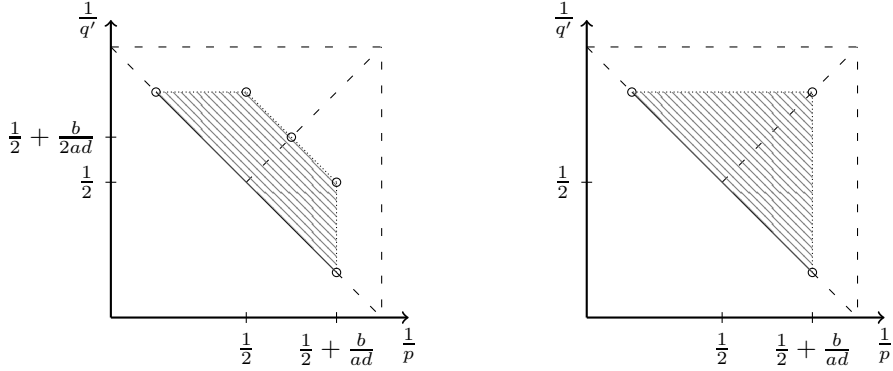

\begin{prop}\label{prop:oscmult-msc} 
Let $1<p\le 2$, 
$a\in (0,\infty)\setminus \{1\}$,   and
 $m$  as in \eqref{eqn:mab-mscale}, with $\sup_k|c_k|\le 1$.
 Let $b>ad(1/p-1/2)$. Then $m(D)\in \Sp(p,p)$.
\end{prop}
\begin{proof}

We decompose as in the proof of Proposition \ref{Miyachiprop}. Recall that 
$\widehat\eta_0$ is supported in $\{|\xi|<1\}$ and $m_{a,b}$ in $\{|\xi|>1\}$, hence $\widehat{\eta_0}m_{a,b}=0$, and  we can write
$m=\sum_{n=1}^\infty m^n$ where
\Be \notag \label{eqn:mn} m^n(\xi)=\sum_{k\in\bbZ} c_k m_{a,b}(2^k\xi) \widehat\eta_n(2^k\xi). \Ee
We shall show that 
\Be\label{eqn:mnsparse}\|m^n(D)\|_{\Sp_\ga(p,p)}\lc 2^{-n\eps(p)}\Ee
with $\eps(p)>0$ and then sum  in $n$.

To verify the claim \eqref{eqn:mnsparse}
we use Theorem \ref{thm:sparsemult}. For this we have to analyze, for radial $\phi\in C^\infty_c$ supported in $\{\xi: 1/2<|\xi|<2\}$, the expression 
\[ \|[\phi m^n(t\cdot)]*\widehat \Psi_\ell\|_{M^{p,p'}} 
\le \sum_{\substack{k\in \bbZ\\2^{n-3}\le 2^k t\le 2^{n+1}}}
\|[\phi  m_{a,b}(2^kt\cdot) \widehat\eta_n(2^kt\cdot)]*\widehat \Psi_\ell \|_{M^{p,p'}}
\]
and show that for some $\eps>0$
\Be\label{eqn:verificationThm} 
\sum_{\ell\ge 0} \sup_{t>0} \|[\phi m^n(t\cdot)]*\widehat \Psi_\ell\|_{M^{p,p'}} 2^{\ell (d(1/p-1/p')+\eps)} \lc 2^{-n\eps(p)} .
\Ee

To this end, fix $k,t$ with $2^{n-3} \le 2^kt\le 2^{n+1}$ and analyze the Fourier inverse of 
$\phi m_{a,b}(2^kt\cdot)\widehat {\eta}_n(2^kt\cdot)$, i.e. 
\[K_n(x) =(2\pi)^{-d}\int \phi(\xi) \widehat\eta_n(2^kt \xi) \chi_\infty(2^kt \xi) (2^kt|\xi|)^{-b} e^{i\inn{x}{\xi} + i(2^k t)^a|\xi|^a } d\xi.
\]
The phase function $\inn{x}{\xi} + (2^k t)^a|\xi|^a $ becomes stationary on the support of $\phi$ only when $|x|\approx (2^kt)^a\approx 2^{na}$ and  the  Hessian  of $|\xi|^a$ is nondegenerate there.  Thus by integration by parts we see that there are constants $c_1<1$, $C_1>1$ such that
\[ |K_n(x) | \le \begin{cases}
 C_N 2^{-n(b+N) }&\text{ for } |x|\le c_1 2^{na} 
 \\
 C_N 2^{-nb} |x|^{-N} &\text{ for }|x| \ge C_1 2^{na}
\end{cases}
\]
and by the method of stationary phase 
\[ |K_n(x) |\lc 2^{-n(b+ad/2)},  \quad \text{ for } \,\,c_1 2^{na}\le|x|\le C_1 2^{na}.\]
This implies for  $2^{n-3}\le 2^k t\le 2^{n+1} $ and suitable $C$, independently of $k$,  $t$, 
\begin{multline}  \label{eqn:oneinfty}  \|[\phi m^n(2^k t\cdot)]*\widehat \Psi_{\ell}\|_{M^{1,\infty}} \\ \lc 
\begin{cases} 2^{-n(b+ad/2)} &\text{ for } |\ell-na|\le C
\\ \min\{ C_N2^{-n(b+N) },\, C_N2^{-n b} 2^{-\ell N}  \} &\text{ for } |\ell-na|>C.
\end{cases}
\end{multline}  We also have the $M^{2,2}$ bound 
\begin{multline}\label{eqn:twotwo}
\|[\phi m^n(2^k t\cdot)]*\widehat \Psi_{\ell}\|_\infty \\\lc 
\begin{cases} 2^{-nb}  &\text{ for } |\ell-na|\le C
\\ \min\{ C_N2^{-n(b+N-d) } , C_N\,2^{-n b} 2^{-\ell (N-d)} \}&\text{ for } |\ell-na|>C.
\end{cases}
\end{multline}
Interpolating \eqref{eqn:oneinfty} and \eqref{eqn:twotwo} we get 
\[\|[\phi m^n(2^k t\cdot)]*\widehat \Psi_{\ell}\|_{M^{p,p'}} \lc 
\begin{cases} 2^{-n(b+ad(1/p-1/2))}  &\text{ for } |\ell-na|\le C
\\ 2^{-nb} C_N\min\{ 2^{-nN}, 2^{-\ell N} \}  &\text{ for } |\ell-na|>C.
\end{cases}  
\]
Only the five terms with $2^{n-3}\le 2^k t\le 2^{n+1} $ make a contribution. We sum those terms, then take a supremum in $t$ (observing that the displayed bound above is independent of $t$) and then sum in $\ell\ge 0$. We obtain 
\begin{multline*}
\sum_{\ell\ge 0} \sup_{t>0} \|[\phi m^n(t\cdot)]*\widehat\Psi_\ell\|_{M^{p,p'} } 2^{\ell(d(\frac 1p-\frac{1}{p'}) +\ep)} 
\\
\lc_\ep 2^{-n(b+ad(\frac 1p-\frac 12))} 2^{na(d(\frac 1p-\frac 1{p'})+\eps)}\lc
2^{-n ( b-ad(\frac 1p-\frac 12-\ep))} .
\end{multline*}
Since we assume $b>ad(1/p-1/2)$ this leads to  
\eqref{eqn:verificationThm} and then to the claim 
\eqref{eqn:mnsparse} via Theorem \ref{thm:sparsemult}. 
\end{proof}

Given Proposition \ref{prop:oscmult-msc}, we can now derive an improved sparse bound for a lacunary maximal function and a corresponding square function associated with the multipliers $m_{a,b}$; thus for these examples we improve on Corollary \ref{cor:laccor}.

\begin{cor} Let $1<p\le 2$, $a>0$, $b>ad(1/p-1/2)$. Then 
  \[\int_{\bbR^d} \Big(\sum_{k\in\bbZ} |m_{a,b}(2^k D) f(x)|^2\Big)^{1/2}  \om(x) dx \lc \Lamaxga_{p,p} (f,\om) 
 .\] \end{cor}
 \begin{proof}
Choose $c_k=\pm 1$ in \eqref{eqn:oscmultdef}. Then  Proposition \ref{prop:oscmult-msc} together with a randomization argument exactly as in the proof of Corollary \ref{cor:laccor}  yields the assertion.
\end{proof}

\subsection{Prototypical versions of singular Radon transforms} \label{sec:SingRadon} 
Let $\sigma $ be a bounded Borel measure 
supported in $\{x:|x|\le 1\}$ and  satisfying 
\begin{align}\label{eqn:SRTassu}
    \int d\sigma=0 \quad \text{ and }\quad \sup_{\xi \in \widehat{\bbR}^d} (1+|\xi|)^{b}  |\widehat \sigma(\xi)|<\infty 
    \quad \text{for some  $b>0$}.
    \end{align}
    Let $\{a_j\}_{j\in \bbZ}$ satisfy \Be\label{eqn:ajbd} |a_j|\le 1
    \Ee and define 
     \Be \notag \label{eqn:SRTdefN1N2} \cS^{N_1,N_2}  f(x) =\sum_{j =N_1}^{N_2} a_j 2^{-jd} \sigma(2^{-j} \cdot)* f(x)\Ee
     and 
     \Be\label{eqn:SRTdef} \cS f(x)= \lim_{\substack {N_2\to \infty\\N_1\to-\infty} }\cS^{N_1,N_2} f(x).
     \Ee
     This is the ``prototypical" singular Radon transform considered by R. Oberlin 
\cite{roberlin}, see also  Duoandikoetxea and Rubio de Francia \cite{DR86}. It is easy to see using the cancellation of the kernel  that the limit exists pointwise for $C^\infty_c$ functions.

In addition, we assume that  $\sigma$ is $L^{p_0}$ improving, i.e. 
    \Be \label{Lpimproving}
    \|\sigma*f\|_{q} \le A \|f\|_{p_0}\Ee
    for some $q$ with $p_0<q<\infty$. 
    The  following result  is due to R.  Oberlin. 
    \begin{prop}[\cite{roberlin}]\label{prop:oberlin} 
    Let $\sigma$ be as in \eqref{eqn:SRTassu}, and $\{a_j\}_{j \in \bbZ}$, $\cS$  be as in \eqref{eqn:ajbd}, \eqref{eqn:SRTdef}.  Let $1<p_0<p<q<\infty$ and assume that \eqref{Lpimproving} holds.  Then $\cS$ satisfies the $(p,q')$-sparse bound
    \[|\inn{\cS f}{\om} |\lc \Lamaxga_{p,q'}(f,\om). \] The same sparse bound holds for the operators $\cS_{N_1,N_2}$, uniformly in $N_1,N_2$.
    \end{prop} We emphasize that Oberlin also proved certain endpoint estimates for $p=p_0$, working with local Orlicz norms in the definition of sparse forms. 

One  can extend Proposition \ref{prop:oberlin} 
to cover associated maximal truncation and variational truncation operators defined by 
    \begin{align*}%
    \cS_*f(x)&= \sup_{N_1<N_2} \Big| 
        \sum_{j=N_1}^{N_2} a_j 2^{-jd} \sigma(2^{-j} \cdot)* f(x)\Big|,
      \\ %
      \cV^{r}_* \cS f(x) &=\sup_{M \in \bbN}\sup_{n_1<\dots< n_{M}} 
 \Big(\sum_{i=1}^{M-1} \big| \sum_{j=n_i+1}^{n_{i+1}} a_j2^{-jd}\sigma(2^{-j}\cdot)*f (x) \big| ^r\Big)^{1/r}   .
  \end{align*} 
  \begin{prop}
  Let $1<p_0<p<q<\infty$, $r>2$ and $\sigma$, $\{a_j\} $ be as in   \eqref{eqn:SRTassu}, \eqref{eqn:ajbd}, \eqref{Lpimproving}.  Then $\cS_*$ and $\cV^{r}_*\cS$ satisfy the sparse bounds
    \begin{align*} 
    |\inn{\cS_* f}{\om}| + |\inn{\cV^r_* \cS f}{\om} | &\lc \Lamaxga_{p,q'}(f,\om). \
    \end{align*}
  \end{prop}
    
\begin{proof}
We apply Theorems \ref{mainthmtrunc} and 
\ref{mainthmtruncvar}.
To verify  the assumptions \eqref{bdness-wt-trunc}, \eqref{bdness-rt-trunc} see \cite[Theorem E] {DR86}. To verify assumptions
    \eqref{bdness-wt-trunc-var}, \eqref{bdness-rt-trunc-var} see 
    \cite[Theorem 1.2]{JonesSeegerWright}.
    Interpolation arguments using the Fourier decay assumption in \eqref{eqn:SRTassu}, and Lemma \ref{lem:regbyFourier} can be used to establish the  additional H\"older condition in \eqref{p-q-rescaled-reg}. 
\end{proof}
The setup above is  also similar in spirit to the theorems on truncations of rough singular integrals with bounded kernels \cite{DiPlinioHytonenLi}. We have been deliberately short in our presentation as the results in this section are essentially known. 
For a more detailed exposition the reader may consult \S\ref{sec:anotherSRT} below, in which  a singular Radon transform built on spherical integrals,  and other versions of maximal functions associated to singular Radon transforms, are considered.

\subsubsection{An approach via Fourier multipliers} \label{sec:SingRadonFM}
In order to understand the scope of our multiplier theorems, it is instructive to deduce the sparse bounds for the prototypical singular Radon transform $\mathcal{S}$ in Proposition \ref{prop:oberlin} from  Theorem \ref{thm:sparsemult} (or Theorem  \ref{thm:localtoglobal-sparse-s}). 
Since $\sigma$ is a finite Borel measure we have 
$\|\phi \widehat \sigma(2^j t\cdot) \|_{M^{q,q} }=O(1)$ for $1\le q\le \infty$. By \eqref{eqn:SRTassu} and interpolation with $L^2\to L^2$ bounds we have 
 for some $\eps_0(q) >0$
\Be\label{eqn:Mqq:Fourier}
\|\phi \widehat \sigma(2^j t\cdot) \|_{M^{q,q}} \lc C_q \min\{ (2^j t)^{\eps_0(q)}, (2^j t)^{-\eps_0(q) }\} , \quad 1<q<\infty,
\Ee
using either cancellation or decay,  and by Young's inequality we get the same bound for 
$\|\phi \widehat \sigma(2^j t\cdot) \|_{M^{p,q}}$ when $1\le p\le q$, $1 < q < \infty$. This takes care of the term $\ell=0$ in the condition \eqref{eqn:Bm}. To verify the remaining hypothesis of Theorem \ref{thm:sparsemult} it suffices to check that for $\ell>0$
 the condition
\Be\label{eqn:SRT-mult} 
\sup_{t>0} \sum_{j\in \bbZ } \| [\phi \widehat \sigma( 2^j t\cdot)] *\widehat \Psi_\ell\|_{M^{p,q}} \lc 2^{-\ell(d(1/p-1/q)+\eps) }
\Ee
is satisfied, as the condition \eqref{eqn:Bmcirc}  trivially follows by the assumption \eqref{eqn:SRTassu}.

Since $\widehat \sigma$ is smooth we have  for $2^j t\le 1$
\[ \|[\phi \widehat{\sigma}(2^j t\cdot)]*\widehat \Psi_\ell \|_{M^{r,s}} \lc C_N 2^{-\ell N}, \quad 2^j t\le 1,
\]
for $1\le r\le s\le \infty$  and therefore by interpolation with \eqref{eqn:Mqq:Fourier} and taking geometric means we see that there is an $\ep_1(r,s)$ such that  $\ep_1(r,s)>0$ if $1<r\le s<\infty$ and  
\Be\label{eqn:neg-terms}  \|[\phi \widehat{\sigma}(2^j t\cdot)]*\widehat \Psi_\ell \|_{M^{r,s} }\lc C_{N} 2^{-\ell d} (2^j t)^{\ep_1(r,s) }, \quad 2^j t\le 1.
\Ee

The contributions for $2^j t\ge 1$ are more interesting.  Since $\sigma$ is supported in $\{x: |x|\le 1\}$  we have the kernel estimate 
\[ |\cF^{-1} [\phi \widehat \sigma (2^j t\cdot)] (x)|\lc_N |x|^{-N} \quad \text{ for } |x|\ge 2^{j+1} t\]
and hence 
\Be \label{eqn:intermediatej} \| [\phi \widehat \sigma( 2^j t\cdot)] *\widehat \Psi_\ell \|_{M^{p,q}} \lc  2^{-\ell N} \quad \text { for } 2^\ell\ge 2^{j+4} t \ge 1.\Ee 
For $2^\ell \le 2^{j+4} t $ we do a rescaling argument to estimate
\[\| [\phi \widehat \sigma( 2^j t\cdot)] *\widehat \Psi_\ell\|_{M^{p_0,q}} \lc 
\|\widehat \sigma (2^j t\cdot) \|_{M^{p_0,q}} = (2^j t)^{-d(1/p_0-1/q) } \|\widehat \sigma\|_{M^{p_0,q} }
\]
and by assumption $\widehat \sigma\in M^{p_0,q}$. Interpolating with the $M^{q,q}$ estimate in \eqref{eqn:Mqq:Fourier} we get for $p_0<p\le q$
\Be\label{eqn:largej} 
\| [\phi \widehat \sigma( 2^j t\cdot)] *\widehat \Psi_\ell\|_{M^{p,q}} \lc (2^{j} t)^{- d(1/p-1/q) -\ep(p,q)}  \quad \text{ for }  2^{j+4}  t\ge 2^\ell.
\Ee
Combining \eqref{eqn:neg-terms}, \eqref{eqn:intermediatej} and \eqref{eqn:largej} and summing in $j$ 
we get
\eqref{eqn:SRT-mult} for a suitable $\eps=\eps(p,q)>0$. 

\subsection{Densities on spheres: Maximal singular integrals}\label{sec:anotherSRT} 

As discussed in \S\ref{sec:SingRadon}  the Corollary \ref{mainthm-cor}  covers classes of singular Radon transforms and also associated maximal operators for truncations. Here we will consider a natural singular integral variant of the spherical maximal function, and obtain a sparse domination inequality analogous to the one for spherical maximal functions with specific dilation sets in  \cite{AHRS, RoosSeeger}.
Let  $\sigma$ be the surface  measure on the unit sphere $\{x: |x|=1\}$ in $\bbR^d$ for $d\ge 2$ and  $\mu=\chi \sigma$ with a choice of smooth $\chi$ such that $$\int d\mu=0.$$
For every $t\in [1,2]$ we consider,  for fixed $t\in [1,2]$,   the prototypical singular Radon transform  as in  the previous section 
\Be\label{eqn:Htdef} \cS_t^{N_1,N_2} f=\sum_{j=N_1}^{N_2}\mu_{2^j t} * f , \qquad \cS_t f=\lim_{\substack{N_2\to \infty,\\N_1\to -\infty}} \cS_t^{N_1,N_2} f \Ee
and then form,  for  $E\subset [1,2]$,    the maximal function \Be \label{eqn:HEdef} \sS_E f(x)=\sup_{t\in E} |\cS_t f(x)|.\Ee  

For $0\le \beta\le \alpha\le 1$ define $\mathscr{R}(\beta,\alpha)\subset [0,1]^2$ as the union of the interior of the convex hull of the points 
\begin{gather*}
Q_1=(0,0), \qquad  Q_{2,\beta}=(\tfrac{d-1}{d-1+\beta}, \tfrac{d-1}{d-1+\beta}),\\ Q_{3,\beta}=(\tfrac{d-\beta}{d-\beta+1}, \tfrac{1}{d-\beta+1}), \qquad  Q_{4,\alpha}=(\tfrac{d(d-1)}{d^2+2\alpha-1}, \tfrac{d-1}{d^2+2\alpha-1}) 
\end{gather*}
with the open segment connecting $Q_1$ and $Q_{2,\beta}$.

\newcommand{\definecoords}{
    \def\ptsize{.1pt}
	\coordinate (Q1) at (0,0);
	\coordinate (Q2) at ( {(\d-1)/(\d-1+\b)}, {(\d-1)/(\d-1+\b)}  );
	\coordinate (Q3) at ( {\Qthreex},  { \Qthreey }  );
	\coordinate (Q4b) at ( { \Qfourx{\b}  }, { \Qfoury{\b} }  );
	\coordinate (Q4g) at ( { \Qfourx{\g}  }, { \Qfoury{\g} }  );
	\coordinate (R) at ( { \d*(\d-1)/ (\d*\d-1+\b) } , { (\d-1)/(\d*\d-1+\b) } );
	\coordinate (C1) at ( { (\Qfourx{\b}+\Qfourx{\g})/2 }, {(\Qfoury{\b}+\Qfoury{\g})/2}  ); 
	\coordinate (C2) at (Q4b);
}

\newcommand{\Qfourx}[1]{\d*(\d-1)/(\d*\d+2*#1-1)}
\newcommand{\Qfoury}[1]{(\d-1)/(\d*\d+2*#1-1)}
\def\Qthreex{(\d-\b)/( \d-\b+1)}
\def\Qthreey{1/(\d-\b+1)}
\newcommand{\drawauxlines}[2]{ 
	\draw (0,0) [->] -- (0,1) node [left] {$\frac1q$};
	\draw (0,0) [->] -- (1,0) node [below] {$\frac1p$};
	\draw [dashed,opacity=.3] (1,0) -- (0,1);
	\draw [dashed,opacity=.3] (#1) -- (1,{1/\d}); 
	\draw [dashed,opacity=.3] (#2) -- (1,1); 
	\draw [dashed,opacity=.1]
	(.5, 0) -- (.5, .5);
}

\newcommand{\drawQbg}{
	\fill (Q1) node [left] {$Q_1$} circle [radius=.02em];
	\fill (Q2) node [right] {$Q_{2,\beta}$} circle [radius=\ptsize];
	\fill (Q3) node [right] {$Q_{3,\beta}$} circle [radius=\ptsize];
	\fill (Q4g) node [below] {$Q_{4,\alpha}$} circle [radius=\ptsize];
	\fill [opacity=.2] (Q1) -- (Q2) -- (Q3) -- (Q4g) -- cycle;
	\draw [opacity=.6] (Q1) -- (Q2) -- (Q3) -- (Q4g) -- cycle;
}	

\begin{figure}[ht]
\begin{tikzpicture}[scale=4.25]
\def\d{3}
\def\b{.75}
\def\g{.9}
\definecoords
\drawauxlines{Q4g}{Q2}
\drawQbg
\end{tikzpicture}
\caption{The region $\mathscr R(\beta,\alpha)$ with  $\beta=0.75$, $\alpha=0.9$, $d=3$.}\label{quadrangle}
\end{figure}
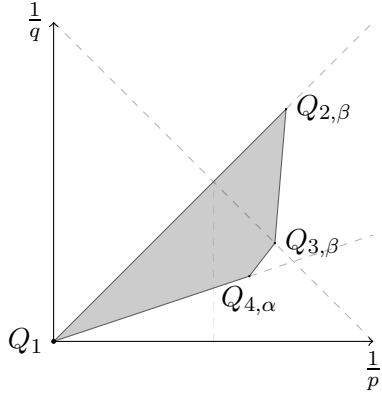

For $E\subset [1,2]$ denote by $\mathrm{dim}_\mathrm{M}\,E$ the upper Minkowski dimension of $E$ and by $\mathrm{dim}_{\mathrm{qA}}\,E$ the quasi-Assouad dimension of $E$ (see \cite{RoosSeeger} for definitions and background, and for a discussion of classes of sets $E$ for which the single-scaled $L^p\to L^q$  results described above are sharp). 

\begin{prop}\label{prop:maxsingintsparse}
Let $d\ge 2$, $E\subset [1,2]$ and $(1/p,1/q)\in \mathscr{R}(\beta,\alpha)$ with $\beta=\mathrm{dim}_\mathrm{M}\,E$, $\alpha=\mathrm{dim}_{\mathrm{qA}}\,E$. 
Then there is the $(p,q')$-sparse domination inequality
\[ %
|\inn{\sS_E f}{\om}| 
\leq  C \Lamaxga_{p,q'}(f,\om).\]
\end{prop}

The two-dimensional version of our operator models a maximal operator associated to a family of Hilbert transforms on curves considered in \cite{GuoRoosSeegerYung1,GuoRoosSeegerYung2} where  nonisotropic dilations are used (see also the previous papers
\cite{MarlettaRicci,GuoHickmanLieRoos} for  related problems).
In this  nonisotropic  case one could also consider more general situations, i.e. when $E$ is not a subset of $[1,2]$ (see also the prior work  \cite{MarlettaRicci} on maximal functions) but this involves multi-parameter structures for which sparse domination result are difficult and in some cases are proved to not hold 
\cite{BarronCondeOuRey}.

\begin{proof}[Proof  of Proposition \ref{prop:maxsingintsparse}] 
Using the density Lemma \ref{lem:density}  we may assume that $f\in C^\infty_c$.
It is then easy to see that for any bounded set $U\in \bbR^n$ we have $\mu_{2^j t}*f(x)=0$ for all $x\in U$, $t\in [1,2]$ and sufficiently large $j$. Moreover using the cancellation of $\mu$ and the smoothness of $f$ we see that $\mu_{2^jt} *f(x)= O(2^j)$ as $j\to -\infty$. Thus we see that for $f\in C^\infty_c$ the function   $\cS_t f$  is well defined and \[\lim_{\substack {N_2\to \infty, \\ N_1\to -\infty }} \sup_{t\in [1,2]} | \cS_t f-\cS_t^{N_1,N_2} f| =0,\] 
where the limit is uniform on compact sets. It is therefore sufficient to prove a sparse bound for the maximal function  $\sup_{t\in E} |S_t^{N_1,N_2} f|$ which is uniform in $N_1$, $N_2$. In what follows we will drop the superscript in $\cS_t^{N_1,N_2} $ but assume that we still working with a truncated sum depending on $N_1$, $N_2$.

To verify  conditions \eqref{p-q-rescaled}, \eqref{p-q-rescaled-reg} in Corollary \ref{mainthm-cor} we first note that for $(\tfrac 1p, \tfrac 1q) \in \sR (\beta, \alpha) $  there is $\eps(p,q) >0$ such that for $\la>2$  
\Be \label{eqn:SErescaled}\|\sup_{t\in E} |\mu_{t}*f| \|_{q} \lesssim_{p,q} \la^{-\eps(p,q) } 
\|f\|_p , \quad f\in \Eann(\la).
\Ee
This is coupled with an elementary $L^p\to L^q$ estimate with constant $O(1)$ estimate for functions with frequency support near the origin to yield \eqref{p-q-rescaled-reg-a} via Lemma \ref{lem:regbyFourier}; this also settles \eqref{p-q-rescaled-reg-a} by translation invariance. For inequality \eqref{eqn:SErescaled} we may refer to \cite[Cor. 2.2]{RoosSeeger}. 

It remains to verify \eqref{bdness-wt} and \eqref{bdness-rt} which follow by verifying the $L^p$ boundedness of $\sS_E$  for $(\tfrac 1p,\tfrac 1p)$ on the open interval $(Q_1Q_{2,\beta})$.
To accomplish this we make a further decomposition on the frequency side.
Let $\eta_\ell$ be as in \eqref{eqn:resofidmod}, \eqref{eqn:suppetaell}, and set 
$\eta_{\ell,j}= 2^{-jd}\eta_\ell(2^{-j}\cdot)$, so that  $\widehat {\eta_{0,j}}$ is supported where $|\xi|\lc 2^{-j}$ and $\widehat {\eta_{\ell,j} }$ is supported where $|\xi|\approx 2^{\ell-j}$.
Setting 
\[ \sS^\ell_E f(x)=\sup_{t\in E} \Big |\sum_{j=N_1}^{N_2} \mu_{2^j t}* \eta_{\ell,j} * f (x) \Big| 
\] 
it then suffices to show 
\Be\label{SEellppest} \|\sS^\ell_E f\|_p 
\lc_p 2^{-\ell\delta(p) } \|f\|_p
\Ee
with $\delta(p)>0$ for $(\tfrac 1p,\tfrac 1p)\in (Q_1Q_{2,\beta})$.
The estimate for $\ell=0$ reduces to standard singular integral theory; this uses the cancellation of $\mu$. Thus from now on we assume $\ell>0$.

We shall first discuss the case when either $d\ge 3$ or $d=2$, $\beta<1$ where we use  arguments as in \cite{SeegerWaingerWright1995}.
Because of $|\widehat \mu(\xi)|\lc \min \{|\xi|, |\xi|^{-(d-1)/2}\}$ we get 
\[ \sup_{\xi \in \widehat{\bbR}^d} \Big| 
\sum_{j=N_1}^{N_2} \widehat {\mu_{2^j t} }(\xi) \widehat {\eta_{\ell,j} }(\xi) \Big| \lc  2^{-\ell(d-1)/2},
\]
which implies an $L^2$ boundedness result for the operators $\cS_t^{\ell}$ with constant $O(2^{-\ell(d-1)/2})$, uniformly in $t\in [1,2]$.

We also have the $L^p$ boundedness result
\[ \Big\| \sum_{j=N_1}^{N_2} \mu_{2^j t}\ast \eta_{\ell,j}  \ast f   \Big\|_{p} \lesssim C_p
\|f\|_p, \quad 1<p<\infty\]
which is a consequence of results on isotropic singular Radon transforms as, say,  in \cite{DR86}. By interpolation we get for all $\eps>0$
\[ \Big\| \sum_{j=N_1}^{N_2} \mu_{2^j t}\ast\eta_{\ell,j}   \ast f  \Big \|_{p} \lesssim C_{\eps,p} 2^{\ell\eps(1-1/p)}
\min( 2^{-\ell \frac{d-1}{p}}, 2^{-\ell \frac{d-1}{p'}}) \|f\|_p, 
\quad 1<p<\infty. \]
The same estimate with $\mu_{2^jt}*\eta_{\ell,j} $ replaced by $2^{-\ell} 
2^{-jd} [\frac{d}{dt} \mu_t *\eta_\ell](2^{-j}\cdot) $ also holds. 
We cover the set $E$ with $O(2^{\ell(\beta+\eps)})$ intervals of length $2^{-\ell}$ and argue  as in 
\cite[p.119]{SeegerWaingerWright1995}  to  obtain 
\[ \|\sS_E^{\ell} f\|_p \lesssim_\varepsilon 2^{\ell(\frac{\beta}p +\varepsilon)} \min( 2^{-\ell \frac{d-1}{p}}, 2^{-\ell \frac{d-1}{p'}}) \|f\|_p .\]
 This gives   \eqref{SEellppest}, provided that  $d\ge 3$ or $d=2$, $\beta<1$.

For the case $d=2$, $\beta=1$, we need to show $L^p$ boundedness for $p>2$. By a Sobolev embedding argument this follows from the inequality
\begin{multline}\label{eq:lsm-sing-rad}
  \Big(\int_1^2 \Big\| \sum_{j=N_1}^{N_2} \mu_{2^j t}*\eta_{\ell,j} \ast  f \Big\|_p^p dt \Big)^{1/p} 
  + 2^{-\ell}
  \Big(\int_1^2 \Big\| \frac{\partial}{\partial t} \sum_{j=N_1}^{N_2} \mu_{2^j t}*\eta_{\ell,j} \ast  f \Big\|_p^p dt \Big)^{1/p} 
  \\
  \lc 2^{-\ell/p-\ell a(p)} \|f\|_p
\end{multline}
where $a(p)>0$ for $2<p< \infty$. 
By Littlewood--Paley theory we see that 
the bound for the first term in
 \eqref{eq:lsm-sing-rad}   reduces to 
\Be\label{eqn:locsm-sq} \Big(\int_1^2 \Big\| \Big(\sum_{j=N_1}^{N_2}  |\mu_{2^j t}*\eta_{\ell,j} \ast  f_j |^2\Big)^{1/2}  \Big\|_p^p dt \Big)^{1/p} \lc 
 2^{-\ell/p-\ell a(p)} \Big\|\Big(\sum_j|f_j|^2\Big)^{1/2} \Big\|_p
\Ee
for $2<p<\infty$. 
\eqref{eqn:locsm-sq}    is established by a local smoothing argument as in  \cite{GuoRoosSeegerYung1} (see in particular an isotropic version of  Corollary 3.6 of that paper). We thus have established the bound for the first term in 
\eqref{eq:lsm-sing-rad}, and the argument for the second term is analogous.
Finally from \eqref{eq:lsm-sing-rad} we obtain \eqref{SEellppest}  by another application of Littlewood--Paley theory  (applying the inequality to functions $f_j$ with $\widehat f_j$  supported where $|\xi|\approx 2^{\ell-j} $).
\end{proof}

\subsection{On radial Fourier multipliers}\label{sec:rad-mult}
We consider radial Fourier multipliers on $\widehat \bbR^d$ with $d\ge 2$,  of the form $m(\xi)=h(|\xi|)$ where $h$ satisfies the condition  $\sup_{t>0}\|\beta h(t\cdot)\|_{L^2_\alpha}<\infty$ for suitable $\alpha$;
here $L^2_\alpha$ is the usual Sobolev space on the real line and $\beta$ is any nontrivial
$C^\infty_c$ function with compact support in $(0,\infty)$. By duality we only need to consider the range $p\le 2$.

The inequality
\Be \label{rad-mult-global} \|h(|\cdot|)\|_{M^{p,q}} \lc \sup_{t>0} t^{d(\frac 1p-\frac 1q)}\|\beta h(t\cdot)\|_{L^2_\alpha} , \quad \alpha> d(1/q-1/2)\Ee
is known to hold for $1<p<\frac{2(d+1)}{d+3}$, $p\le q\le 2$
and one may conjecture that it holds  for $\frac{2(d+1)}{d+3}<p\le \frac{2d}{d+1}$ and $p\le q< \frac{d-1}{d+1}p'$. Indeed, as a straightforward consequence of the Stein--Tomas restriction theorem and Littlewood--Paley theory one gets for the endpoint  $p=\frac{2(d+1)}{d+3}$, $q=2$, $\alpha=0$, a complete characterization of  radial  Fourier multipliers in  $M^{p,2}$; namely
\[\|h(|\cdot|)\|_{M^{p,2} } \approx \sup_{t>0} t^{d(\frac 1p-\frac 12)}\Big(\int_{t}^{2t} |h(s)|^2 \frac{ds}{s}\Big)^{1/2},
\]
see \textit{e.g.} \cite{GarrigosSeeger2008}.  
The case $p=q=\tfrac{2d}{d-1} $ has been settled only in two dimensions in \cite{Carbery1983,CarberyGasperTrebels}, but remains open in three and higher dimensions. Note that as a special case one has the Bochner--Riesz conjecture when $h(s)= (1-s^2)_+^\la$. 
For partial $L^p\to L^p$ results in higher dimensions (via the connection  \cite{CarberyGasperTrebels} with Stein's square function)  we refer to 
\cite{ChristProc, Seegerthesis1986, LeeRogersSeeger2012,  Lee2018},
\cf. \S\ref{Steinsqfctp>2}  below.

We formulate sparse bounds for the multipliers satisfying 
\eqref{rad-mult-global}; in fact our hypotheses will involve the  single scale variant
\Be \label{rad-mult} \|g(|\cdot|)\|_{M^{p,q}} \le C(\alpha) \|g\|_{L^2_\alpha}, \; \alpha> d(1/q-1/2),\,\, \supp(g)\subset [1/2,2]. \Ee
Typically, the assumption \eqref{rad-mult} will be  applied to $g$  of the form $\beta h(t\cdot) $.

Theorem \ref{thm:sparsemult} leads to the following result.
\begin{prop}\label{cor:radial mult} Let $1<p\le q\le 2$ and let $T_h$ be the convolution operator with multiplier $h(|\cdot|)$.
Then

(i) Assume \eqref{rad-mult} holds
for a specific exponent pair  $(p,q)$ with 
$1<p\le \frac{2d}{d+1}$,  $p\le q\le  \min\{\frac{d-1}{d+1}p', 2\}$, and all   $\alpha> d(1/q-1/2)$. Then 
\Be \label{sparse-radial} \|T_h\|_{\Sp_\ga(p,q' )} \le C_b \sup_{t>0} \|\beta h(t\cdot)
\|_{L^2_b} , \quad b>d(1/p-1/2).
\Ee

(ii) In particular, 
\eqref{sparse-radial} holds true for 
 $1<p\le \tfrac{2(d+1)}{d+3}$, $p\le q\le 2$. 
 \end{prop}

\begin{proof}   We need to verify the assumptions of Theorem \ref{thm:sparsemult}. 
This amounts to veryfing the finiteness of the condition \eqref{eqn:Bm}. Setting $g= \beta h(t \cdot)$ and fixing $b>d(1/p-1/2)$ this follows from proving that for $\ell\ge 1$ we have %
\Be \label{eqn:decay-ine-ell}
\| g(|\cdot|)* \widehat{\Psi}_\ell\|_{M^{p,q} }\lc 2^{-\ell(\frac dp-\frac dq+\eps(b))}
\|g\|_{L^2_b}  \Ee
for some $\eps(b)>0$.

Let $\upsilon_0$ be supported in $\{s\in \bbR :|s|<1/2\} $ such that $\int \upsilon_0(s) ds=1$.
For $n\ge 1$ let $\upsilon_n(s)= \upsilon_0(2^ns)-\upsilon_0(2^{n-1}s)$, and define $g_n(s)= g*\upsilon_n(s) $.
By assumption \eqref{rad-mult}, we have 
$\|g_n(|\cdot|)\|_{M^{p,q}} \lc  \|g_n\|_{L^2_\alpha}$ for any $\alpha > d(1/q-1/2)$ and hence also
\Be \label{eqn:gn}\|g_n(|\cdot|)*\widehat \Psi_\ell\|_{M^{p,q} }
\lc \|g_n(|\cdot|)\|_{M^{p,q} } \lc \|g_n\|_{L^2_{\alpha} } \lc 2^{-n(b-\alpha) } \|g\|_{L^2_b}.
\Ee
For our fixed choice of $b>d(1/p-1/2)$, we choose $b_1<b$ such that $d(1/p-1/2)<b_1<b$, so that \eqref{eqn:gn} holds for the choice
 $\alpha=b_1-d(1/p-1/q) > d(1/q-1/2)$.

We let $\ep=(b-b_1)(2d)^{-1}$ and use  \eqref{eqn:gn} for $n\ge \ell (1+\ep)^{-1}$. Since 
$$\frac{d(\frac 1p-\frac 1q) + b-b_1}{1+\ep}  -d(\frac 1p-\frac 1q)= \frac{b-b_1-\ep d(\frac 1p-\frac 1q)}{1+\ep}\le \frac{b-b_1}{1+\ep}:=\eps(b)>0$$
we get
\[\sum_{n\ge \frac{\ell}{1+\ep}} 
\|g_n(|\cdot|)*\widehat \Psi_\ell\|_{M^{p,q} }
\lc 2^{-\frac{\ell}{1+\ep} (\frac dp-\frac dq+b-b_1)}\| g \|_{L^2_b} \lc 
2^{-\ell (\frac dp-\frac dq +\eps(b))} \| g \|_{L^2_b}\]
For $n\le \ell/(1+\ep)$ we observe that any derivative of order $k$ of $g_n(|\cdot|)$ is $O(2^{kn}\|g\|_1)$ and  an $N$-fold integration by parts gives  
$|\cF^{-1} [g_n(|\cdot|)](x)|\le C_N2^{(n-\ell)N}$ for $|x|\approx 2^\ell$, 
for all $N\in \bbN$. We  use this with  $N:=10d\frac{1+\ep}{\ep}$.
By Young's inequality 
\begin{align*}\sum_{n\le \frac{\ell}{1+\eps_1}}
\| \cF^{-1} [g_n(|\cdot|)*\widehat\Psi_\ell] \|_{M^{p,q} }
\lc
C_N  \ell2^{\ell (d-d(\frac 1p-\frac 1q))}2^{-\ell( \frac{\ep}{1+\ep}N)} \lc C(\ep)2^{-8d\ell}
\end{align*}
and \eqref{eqn:decay-ine-ell} is verified.
\end{proof}
As an example in the   above class of multiplier transformations we consider a multi-scale version of Bochner--Riesz operators. The Bochner--Riesz means of the Fourier integral are defined by
\begin{equation}\label{BRmeansdef}
\widehat{S^\la_t f}(\xi) = (1-|\xi|^2/t^2)^\la_+ \widehat{f}(\xi) 
\end{equation} 
and are conjectured to be bounded from $L^p\to L^q$ if $\la> d(1/q-1/2)-1/2$ and $1\le p\le q\le \min\{ \frac{d-1}{d+1}p',\, 2\}$, with operator norm
$O(t^{d(1/p-1/q)})$. One may reduce to $t=1 $ by scaling, and if $h_\la(s)= (1-s^2)_+^\la$ then $h_\la\in L^2_\nu$ for $\la>\nu+1/2$.
Therefore, Proposition \ref{cor:radial mult} immediate leads to sparse bounds for operators such as 
$\sum_{k=-\infty}^\infty \pm \big(S^\la_{2^k} - S^\la_{2^{k+1}} \big)$.
with uniform bounds in the choice of  the sequence of signs. After  a standard averaging argument using Rademacher functions this implies  sparse bounds for lacunary square functions
The vector-valued version of Theorem \ref{thm:sparsemult} leads to   sparse domination  for the lacunary square-function 
\Be \notag \label{lac-square-fct} \Big(\sum_{k\in \bbZ}| S^\la_{2^k} f- S^\la_{2^{k+1}} f|^2\Big )^{1/2}\Ee and consequently to sparse bounds for lacunary Bochner--Riesz maximal functions  $M_\la f=\sup_{k \in \bbZ} |S_{2^k} ^\la f|$. These  results can be  viewed as a natural  multi-scale generalization of the sparse domination results for Bochner--Riesz means in \cite{benea-bernicot-luque,LaceyMenaReguera}.
In this context, we remark that there are sharper endpoint sparse domination result for Bochner--Riesz means \cite{keslerlacey} which yield back some of the known  weak type $(p,p)$ endpoint bounds for \[\la=d(1/p-1/2)-1/2.\] However,  currently there is no sparse bound for analogous endpoints which involve multiple frequency scales. We  intend to return to this question in the future.

\subsection{Stein's square function} 
In \cite{Stein1958} Stein introduced the square function defined via Bochner--Riesz means by,
\begin{align*}G^\alpha f(x) &= \Big(\int_0^\infty\Big| \frac{\partial
S^\alpha_t f(x)}{\partial t} \Big|^2 t\,dt \Big)^{1/2}
\\&= c_\alpha \Big(\int_0^\infty |S^{\alpha-1}_t f(x)-S^\alpha_t f(x)|^2 \frac{dt}{t}\Big)^{1/2}
\end{align*}
in order to  establish pointwise convergence and strong summability results. %
Another important connection was established in  \cite{CarberyGasperTrebels}, namely that an $L^p$-boundedness result for $G^\alpha$ implies that the condition $\sup_{t>0} \|\beta h(t\cdot)\|_{L^2_\alpha}<\infty$
is sufficient for  $h(|\cdot|)\in M_p$. Moreover, $G^\alpha$ also controls maximal operators associated to radial Fourier multipliers \cite{CarberyEscorial}.

The expression $G^\alpha f(x) $  is almost everywhere equivalent to many alternative square functions, which can be obtained via versions of Plancherel's theorem with respect to the $t$-variable;
see the paper by Kaneko and Sunouchi \cite{KanekoSunouchi}.
We distinguish the cases $1<p\le 2$, in which by a result of Sunouchi
\cite{Sunouchi1967} we have $L^p$ boundedness for $\alpha>d(1/p-1/2)+1/2$,  and the more subtle case $2\le p<\infty$,
where $L^p(\bbR^d)$ boundedness for $d\ge 2$ is conjectured for $p>\frac{2d}{d-1}$ and $\alpha>d(1/2-1/p)$, and known if $d=2$ \cite{Carbery1983}. $L^p$ boundedness in the latter problem is closely related to the multiplier problem discussed in \S \ref{sec:rad-mult}; see \cite{ChristProc, Seegerthesis1986, LeeRogersSeeger2012, Lee2018} for partial results and \cite{LeeRogersSeeger-Stein-vol, LeeSeeger2015} for certain endpoint and weighted bounds. 

We recall some basic decompositions of the Bochner--Riesz means. One splits
\[(1-|\xi|^2)_+^{\alpha-1} -(1-|\xi|^2)^{\alpha}_+=\sum_{n\ge 0} 2^{-n(\alpha-1)} u_n(|\xi|), \]
where $u_0(0)=0$, the $u_n$ are smooth, and for $n\ge 1$ we have \[\supp(u_n)\subset (1-2^{-n+1}, 1-2^{-n-1})\] and $ |\tfrac{d^j}{d s^j} u_n(s)|\le C_j 2^{nj}$ for $j \in \bbN_0.$ 
Let $K_n= \cF^{-1}[u_n(|\cdot|)]$, $K_{n,s} =s^d K_n(s\cdot)$ and 
\[G_n f(x)= \Big(\int_0^\infty |K_{n,s}*f(x)|^2 \frac{ds}{s}\Big)^{1/2},\] so that
\[G^\alpha f(x)\lc \sum_{n=0}^{\infty} 2^{-n(\alpha-1)} G_n f(x).\]
We shall rely on the standard pointwise estimates obtained by stationary phase calculations, 
\Be \label{eqn:Knptw} |K_n(x)|\lc_N (1+|x|)^{-\frac{d+1}{2}} (1+2^{-n}|x|)^{-N}.
\Ee

\subsubsection{The case \texorpdfstring{$1<p\le 2$}{p < 2}}\label{Steinsqfctp<2}
A pointwise sparse domination result for $\alpha>(d+1)/2$ was proved by Carro and Domingo-Salazar \cite{CarroDomingo-Salazar}. For $1/2<\alpha \leq (d+1)/2$ we have  $L^p$ boundedness ($p\le 2$) only in the restricted range $\frac{2d}{2\alpha+2d-1}<p\le 2$ by Sunouchi's result which is sharp. 
Thus in this range  we are seeking  sparse domination results for the forms $\inn {G^\alpha f_1}{f_2}$.
  Theorem \ref{thm:sparsemult} yields the following.
\begin{prop} Let $d\ge 2$, $\frac 12<\alpha \le \frac{d+1}{2}$. Then for $\frac{2d}{2\alpha+2d-1}<p
\le 2$ we have the $(p,p)$-sparse domination inequality
\[|\inn{G^\alpha f}{\om}| \leq C \Lamaxga_{p,p}(f,\om).
\]
\end{prop}
\begin{proof}
The operators $G_n$ are defined through smooth kernels and therefore the result in \cite{CarroDomingo-Salazar} yields pointwise sparse bounds, with norms depending on $n$. This settles the case of small values of $n$. For large values of $n$, given $\eps>0$ we have to show
\Be\label{eqn:Gnsparse} |\inn{G_n f}{ \om}| \lc_\eps 2^{n (\frac dp-\frac d2-\frac 12+\eps)}   \,\La_{p,p}^{\gamma, *}(f,\om)
\Ee
since in the assumed range of $p$ we have $\alpha-1> d/p-d/2-1/2$ and therefore we can sum in $n$ to obtain the result for $G^\alpha$.
Let $\sH$ be the Hilbert space $L^2(\bbR^+, \frac{ds}{s})$.
By the linearization argument in \S\ref{sec:mainthm-part-two}, the inequality  \eqref{eqn:Gnsparse}  follows, for a scalar function $f_1$ and an $\sH$-valued function $f_2=\{f_{2,s}\}$,  from 
\Be \notag \label{eqn:Gnsparsevect} \Big|\iint K_{n,s} * f_1(x) f_{2,s}(x) \frac{ds}{s} dx \Big| \lc_\eps  2^{n (\frac dp-\frac d2-\frac 12+\eps)}  \,\La_{p,\bbC,p,\sH^*}^{\gamma, *}(f_1,f_2).
\Ee
By Theorem \ref{thm:sparsemult} this follows from 
\Be \label{eqn:pq-verification}
\sup_{t>0} \big\|[\beta(|\cdot|)u_n(t|\cdot|)]*\widehat \Psi_\ell\big\|_{M^{p,p'}_{\bbC,\sH^*} } 2^{\ell (\frac dp-\frac d{p'}+\ep_1)} \lc_\eps  2^{n (\frac dp-\frac d2-\frac 12+\eps)} 
\Ee
for some $\ep_1>0$.  To verify \eqref{eqn:pq-verification} we argue by interpolation and reduce to the cases $p=2$ and $p=1$.
It will be helpful to observe that for $1/8\le t/s\le 8$ we can replace the kernel $K_n$ on the left hand side of the inequality in \eqref{eqn:Knptw}
by $\cF^{-1} [\beta(|\cdot|)u_n(\frac ts|\cdot|)]$.
Thus
\Be \label{eqn:ptwise-vect}
\Big(\int| \cF^{-1} [ \beta(|\cdot|) u_n(\tfrac ts |\cdot|)](x) |^2 \tfrac{ds}{s}\Big)^{1/2} \lc_N (1+|x|)^{-\frac{d+1}{2}} (1+2^{-n}|x|)^{-N}.
\Ee 
Here we used that given $t$ the integrand is zero unless  $s\in [t/2, 2t]$. 
Hence for any $\ep_2>0$ (which we choose $\ll\min \{\eps,\ep_1\}$) we get 
\Be \label{eqn:Psi-error}
\sup_{t>0} \big\|[\beta(|\cdot|)u_n(t|\cdot|)]*\widehat \Psi_\ell\big\|_{M^{p,q}_{\bbC,\sH^*} } 
 \lc_{\ep_2,N} 2^{-\ell N} \quad 
 \text{ if } \ell>n(1+\ep_2) .
 \Ee
 If  $p=2$ we have also have, for $\ell\le n(1+\ep_2)$,
  \begin{align}  \label{eqn:p=2bd}
\big\|[\beta(|\cdot|)u_n(t|\cdot|)]*\widehat \Psi_\ell\big\|_{M^{2,2}_{\bbC,\sH^*} } 2^{\ell\ep_1}
 &\lc 2^{\ell\ep_1}\sup_\xi \Big(\int| \beta(|\xi|) u_n(\tfrac ts|\xi|)|^2 \tfrac{ds}s \Big)^{1/2}\\ \notag & \lc 2^{\ell\ep_1} 2^{-n/2}\lc 2^{n\ep_1(1+\ep_2)}2^{-n/2}.
 \end{align}
 Furthermore, for $p=1$ and $\ell\le n(1+\ep_2)$ we use \eqref{eqn:ptwise-vect} to see that
 \begin{align} \label{eqn:p=1bd}
&\big\|[\beta(|\cdot|)u_n(t|\cdot|)]*\widehat \Psi_\ell\big\|_{M^{1,\infty}_{\bbC,\sH^*} } 2^{\ell(d+\ep_1)}
 \\  \notag &\lc 2^{\ell(d+\ep_1)}  2^{-n(d+1)/2} \lc 2^{n(1+\ep_2)(d+\ep_1) - (d+1)/2} \lc 2^{n(d-1)/2+\eps}.
 \end{align}
 Combining \eqref{eqn:p=2bd}, \eqref{eqn:p=1bd} with \eqref{eqn:Psi-error} we obtain the cases of \eqref{eqn:pq-verification} for $p=1$ and $p=2$ and \eqref{eqn:pq-verification} follows by interpolation for $1\le p\le 2$.
\end{proof}

\subsubsection{The case \texorpdfstring{$2<p<\infty$}{p>2}} 
\label{Steinsqfctp>2}
The reduction to sparse bounds will be similar as in the case $p\le 2$, but the input information is more subtle.
 Instead of the  pointwise bounds \eqref{eqn:Knptw} we now use that 
 \Be\label{eqn:GnST}
 \Big\| \Big(\int_{1/8}^8 |K_{n,s}* f|^2\tfrac{ds}{s} \Big)^{1/2} \Big\|_p
 \lc 2^{n(\frac d2-\frac dr-1) } \|f\|_r,   
 \Ee
 for $2\le r\le p$, $p\ge \tfrac{2(d+1)}{d-1}$, which was proved using the Stein--Tomas restriction theorem \cite{ChristProc, Seegerthesis1986, LeeRogersSeeger-Stein-vol}.  We then obtain a satisfactory result for $\alpha \geq \frac{d}{d+1}$.
  
  \begin{prop} Let $d\ge 2$, $\frac d{d+1}\le \alpha \le \frac{d}{2}$. Then for $\frac{2d}{d-2\alpha}< p<\infty$ we have the $(2,p)$-sparse domination inequality
\[|\inn{G^\alpha f}{\om}| \leq C \Lamaxga_{2,p}(f,\om).\]
\end{prop}
  \begin{proof}
  Note that in the given $p$-range, $p>\frac{2(d+1)}{d-1}$  when $\alpha\ge \frac{d}{d+1}$. We use the notation in the proof of the preceding proposition. By linearization (see the argument in \S\ref{sec:mainthm-part-two}) 
  it suffices to prove
  \Be\notag \label{eqn:Gnsparsevectmod} \Big|\iint K_{n,s} * f_1(x) f_{2,s}(x) \frac{ds}{s} dx \Big| \lc_{\eps} 2^{n (\frac d2-\frac dp-1+\eps)}  \,\La_{2,\bbC,p,\sH^*}^{\gamma, *}(f_1,f_2)
\Ee
and by  Theorem \ref{thm:sparsemult} this follows, given $\eps>0$ from
\Be \label{eqn:pq-verificationmod}
 \big\|[\beta(|\cdot|)u_n(t|\cdot|)]*\widehat \Psi_\ell\big\|_{M^{2,p}_{\bbC,\sH^*} } 2^{\ell (\frac d2-\frac d{p}+\ep_1)} \lc_\eps 2^{n (\frac d2-\frac dp-1+\eps)} 
\Ee
for some $\ep_1>0$, uniformly in $t>0$.  For $\ell<n(1+\ep_2)$ the left hand side is bounded by 
\Be \label{eqn:auxexpression}
2^{n(1+\ep_2) (\frac d2-\frac d{p}+\ep_1) } \big\|\beta(|\cdot|)u_n(t|\cdot|)\big\|_{M^{2,p}_{\bbC,\sH^*} }  .\Ee Using \eqref{eqn:GnST} for $r=2$ 
we get
\[
 \Big\| \Big(\int_{t/8}^{8t} |\cF^{-1} [ \beta(|\cdot|) u_n(\tfrac ts|\cdot|) \widehat f] |^2 \tfrac{ds}{s} \Big)^{1/2} \Big\|_p \lc 2^{-n} \|f\|_2
 \]
 and thus the expression in \eqref{eqn:auxexpression}  is $O(2^{n(1+\ep_2) (\frac d2-\frac d{p}+\ep_1)-n }).$  Finally, we choose  $\ep_1,\ep_2\ll\eps$ and combine this  with the error estimates \eqref{eqn:Psi-error}. This completes the verification of  \eqref{eqn:pq-verificationmod}.
  \end{proof}

\appendix
\section{Facts about sparse domination}\label{basicsect}

{For completeness, we collect a number of auxiliary results, some of them well-known, about sparse domination.}

\subsection{Replacing simple functions} 

It is often convenient  to replace   
the spaces $\simpleone$ and 
$\simpledual$ in  the definition of the $\Sp_{\ga}(p_1,B_1, p_2, B_2^*)$ norms  by  other
suitable  test function 
classes 
 such as the spaces  of compactly supported $C^\infty$ functions or Schwartz functions. This is justified by the following Lemma.

\begin{lem}\label{lem:density} Suppose $1\le p_1<p_2'$ and $p_1<p<p_2'$ and let $T\in \Sp_\ga(p_1,p_2)$. Let  $\cV_1$ be a dense subspace of $L^p_{B_1}$  and 
$\cV_2$ be a dense subspace of $L^{p'}_{B_2^*}$. Then
\begin{multline}\label{newdefsparse}
\|T\|_{\Sp_\ga(p_1, B_1,p_2,B_2^*)}= \sup\Big\{
\frac {|\inn{Tf_1}{f_2}|}{\Lamaxga_{p_1,B_1,p_2,B_2^*} (f_1,f_2)}: f_i\in \cV_i,\,  f_i\neq 0,\, i=1,2\Big\}.
\end{multline}
\end{lem}

\begin{proof}

We first assume that $\cV_1=L^p_{B_1}$ and $\cV_2=L^{p'}_{B_2^*}$. In what follows we omit the reference to $B_1$, $B_2^*$.
 The right-hand side of \eqref{newdefsparse} dominates $\|T\|_{\Sp_\ga(p_1,p_2)}$, defined in \eqref{sparse norm}. In order to  verify the reverse inequality we have to show that given $\ep>0$ and given $f_1\in L^p$, $f_2\in L^{p'}$ we have the inequality
\Be \label{eqn:uptoeps}|\inn{Tf_1}{f_2}|\le (\|T\|_{\Sp_\ga(p_1,p_2)}+\ep) \, \Lamaxga_{p_1,p_2} (f_1,f_2).
\Ee
This is clear if one of the $f_i$ is zero almost everywhere. We may thus assume that
$\|f_1\|_p>0$, $\|f_2\|_{p'}>0$.
For any $\eps_1>0$ choose $g_1\in\simpleone$, $g_2\in \simpledual$ so that 
$\|f_1-g_1\|_p\le \ep_1$,  $\|f_2-g_2\|_{p'}\le \ep_1$, and also $\|g_1\|_p\le 2\|f_1\|_p$ ,
$\|g_2\|_{p'}\le 2\|f_2\|_{p'}$ and estimate, using the definition of $\|T \|_{\mathrm{Sp}_\gamma (p_1,p_2)}$,
\begin{align*}
&|\inn{Tf_1}{f_2} |\le 
|\inn{T[f_1-g_1]}{f_2} |+|\inn{Tf_1}{f_2-g_2}|+  |\inn{Tg_1}{g_2} |
\\
&\le \|T\|_{p-p} \big( \|f_1-g_1\|_p\|f_2\|_{p'} + \|f_1\|_p \|f_2-g_2\|_{p'} \big) + \|T\|_{\Sp_\ga(p_1,p_2)} \Lamaxga_{ p_1, p_2} (g_1,g_2) \\
& \le \|T\|_{p-p} \big( \ep_1 \|f_2\|_{p'} + \|f_1\|_p \ep_1 \big) + \|T\|_{\Sp_\ga(p_1,p_2)} \Lamaxga_{ p_1, p_2} (g_1,g_2).
\end{align*} 
Moreover, for $p_1<p<p_2'$, one has using (ii) in Lemma A.2. that
\begin{align*} 
&\Lamaxga_{p_1, p_2} (g_1,g_2) \le 
\Lamaxga_{ p_1, p_2} (g_1-f_1,g_2) +
\Lamaxga_{ p_1, p_2} (f_1,g_2-f_2)  +
\Lamaxga_{p_1, p_2} (f_1,f_2)
\\
 &\le C_1(p,p_1,p_2)(\|g_1-f_1\|_p \|g_2\|_{p'} +\|f_1\|_p \|g_2-f_2\|_{p'})+\Lamaxga_{p_1, p_2} (f_1,f_2)
 \\
 &\le C_1(p,p_1,p_2) (2\|f_2\|_{p'}\ep_1 +\|f_1\|_p\ep_1) + \Lamaxga_{ p_1, p_2} (f_1,f_2)
 \end{align*} and thus 
  \[|\inn{Tf_1}{f_2} |\le 
 \|T\|_{\Sp_\ga(p_1,p_2)} \, \Lamaxga_{ p_1, p_2} (f_1,f_2)  + \cE,\]
  with $\cE\le C(f_1,f_2, p,p_1,p_2,T) \ep_1$.
Choosing a suitable $\ep_1$ depending on $\ep$ we obtain the assertion \eqref{eqn:uptoeps}, for the case $\cV_1=L^p$, $\cV_2=L^p$.

In the general case we replace
 the couple  of pairs
 $(\simpleone, \simpledual)$  and $(L^p,L^{p'})$ by the couple of pairs 
  $(\cV_1,\cV_2)$ and $(L^p,L^{p'})$ and see that a repetition of the above arguments settles this case as well. 
\end{proof}

\subsection{The Hardy--Littlewood maximal function}\label{subsec:sparse HL}

 It is a well-known fact that the Hardy--Littlewood maximal operator, denoted by $\cM$, satisfies a sparse domination inequality. We have not been able to 
 identify the  original reference for this fact and  refer 
to Lerner's expository lecture \cite{Lerner-ElEscorial} instead.  This constitutes a 
first nontrivial example for sparse domination and we include
a  standard proof for completeness.
\begin{lem}
Let $f \in L^1_{loc}(\bbR^d)$. Then there exist
$\gamma$-sparse
families $\fS_i(f)$, $i=1,\dots, 3^d$,  such that
$$
\cM f(x) \leq 2^d(1-\gamma)^{-1}  \sum_{i=1}^{3^d} \sum_{Q \in \fS_i(f)} \jp{f}_{Q,1} \bbone_Q(x). 
$$
\end{lem}

\begin{proof}
Let $\fD$ be a dyadic lattice and let $\cM^{\fD}$ denote the dyadic maximal function associated to $\fD$, that is, $\cM^{\fD}f(x):=\sup_{\substack{Q \ni x\\Q \in \fD}} \jp{f}_{Q}$. Fix $a \in \bbR$ a constant to be determined later, and for each $k \in \bbZ$, define the sets
$$
\Omega_k:=\{ x \in \bbR^d : \cM^\fD f(x) > a^k \} = \bigcup_{Q_{k}^j \in \fD } Q_k^j
$$
where $\{Q_k^j\}_j \subseteq \fD$ are the maximal disjoint dyadic cubes for $\Omega_k$, that is,
\begin{equation}\label{eq:maximal cubes}
a^k \leq \jp{f}_{Q_k^j} \leq a^k 2^d.
\end{equation}
Define the sets $E_k^j:= Q_k^j \backslash \Omega_{k+1}$, and note that the family of sets $\{E_k^j\}_{k,j}$ is pairwise disjoint and $|E_k^j| > (1-\tfrac{2^d}{a})|Q_k^j|$; for the last claim, note that
\begin{align*}
|Q_k^j \cap \Omega_{k+1}| &= \sum_{i} |Q_k^j \cap Q_{k+1, i}| = \sum_{i \, :\, Q_{k+1}^{i} \subset Q_k^j} |Q_{k+1}^i| \\
& < \sum_{i\, :\,  Q_{k+1}^{i} \subset Q_k^j} \frac{1}{a^{k+1}} \int_{Q_{k+1}^i} |f| \leq \frac{1}{a^{k+1}} \int_{Q^j_k} |f| 
\leq \frac{2^d}{a} |Q_{k}^j|, 
\end{align*}
using  the disjointness of the cubes $Q^i_{k+1}$ and \eqref{eq:maximal cubes} for $Q_{k+1}^i$ and $Q_k^j$. 
Thus, choosing $a=2^{d}(1-\gamma)^{-1}$, $\fS(f):=\{Q^j_k\}_{k,j}$ is a $\gamma$-sparse family, and moreover
\begin{align*}
    \cM^{\fD} f(x)&= \sum_{k \in \bbZ} \cM^{\fD} f(x) \bbone_{\Omega_k \backslash \Omega_{k+1}}(x)  \leq \sum_{k \in \bbZ} \sum_{Q_j^k} a^{k+1} \bbone_{E_k^j}(x) \\
    & \leq a \sum_{k, j} \jp{f}_{Q_k^j} \bbone_{Q_{k}^j}(x) = a \sum_{Q \in \fS(f)}  \jp{f}_{Q} \bbone_{Q}(x)
\end{align*}
by \eqref{eq:maximal cubes}. Finally, the result for the maximal function $\cM$ follows from the $3^d$-trick (see \cite[Theorem 3.1]{lerner-nazarov}), which ensures that there exist $3^d$ dyadic lattices $\fD_i$, $i=1, \dots, 3^d$ such that
\[
\cM f(x) \leq \sum_{i=1}^{3^d} \cM^{\fD_{i}} f(x). \qedhere
\]
\end{proof}

\begin{remarka} If $f$ has values in a  Banach space $B$, the same argument applies to $f\mapsto \cM(|f|_B)$. However, there are more interesting   vector-valued extensions such as  in the Fefferman--Stein theorem \cite{FeffermanStein}, 
and corresponding general sparse domination results with additional hypotheses on the Banach space are discussed in a paper by H\"anninen and Lorist \cite{HanninenLorist}.%
\end{remarka}

\subsection{Operators associated with dilates of Schwartz functions} 
 It is convenient in many applications to observe that maximal functions and variation operators generated by convolution operators with Schwarz functions satisfy sparse bounds. We choose to deduce the variational statements as a consequence of our Theorems in \S \ref{sec:squarefct-etc}, but it could also be based \textit{e.g.} on \cite{DiPlinioHytonenLi}. For the definition of the dyadic and short variation operators we refer to Remark \ref{rem:longshort}; here  $V^r_{\mathrm{sh}}$ is understood with $E=(0,\infty)$.

\begin{lem}\label{lem:nonlocalerror} Let $K\in C^2(\bbR^d)$ be a convolution kernel satisfying, for all multiindices $\alpha \in \bbN_0^d$ with $\sum_i |\alpha_i|\le 2$, 
\Be \notag \label{eqn:kappaest} 
|\partial^{\alpha} K(x)|\le (1+|x|)^{-d-2}.
\Ee Let  $K_t(x)=t^{-d}K(t^{-1}x)$ and let $\cK_tf=K_t* f(x)$. 
Then for $1<p\leq q<\infty$
\begin{align*} & |\inn{\sup_{t>0} |\cK_tf| }{\om} | \lc \Lamaxga_{p,q'} (f,\om)
\\
&|\inn{V^r_\dyad \cK f}{\om} | \lc \Lamaxga_{p,q'} (f,\om), \text{ $2< r < \infty,$}
\\
&| \inn{ V^r_{\mathrm{sh}} \cK f}{\om} | \lc \Lamaxga_{p,q'} (f,\om), \text{ $2 \leq r < \infty.$}
\end{align*}
\end{lem}

\begin{proof} Since $\sup_{t>0} |\cK_t*f|$ is pointwise dominated by the Hardy-Littlewood maximal function  the sparse bound for 
$\inn{\sup_{t>0} |\cK_tf| }{\om}$
can be directly deduced from the sparse bound  for the Hardy--Littlewood maximal function $\cM$ in \S\ref{subsec:sparse HL}.

For the variation norm inequalities we decompose  $\cK=\sum_{n=0}^\infty \cK^n$ where $\cK^n$ denotes convolution with ${K^n}:=K\Psi_n $
(here $\Psi_n$ is supported where $|x|\approx 2^{n}$  when $n>0$, see \eqref{eqn:defofPsiell}). We can form the long and short variation operators with respect to the family of  operators $\{\cK^n_t\}_{t>0}$ where $\cK_t^n$ denotes convolution with $K^n_t:=t^{-d} K^n(t^{-1}\cdot)$.  
Using the pointwise bound  on $\nabla K$ and results in \cite{JonesKaufmanRosenblattWierdl} or \cite{JonesSeegerWright} we  have 
\begin{align*}
\| V^r_\dyad \cK^n  f \|_p & \lc_p  2^{-n} \|f\|_p, \quad 1<p\le \infty, \,\, r>2,
\\
\| V^r_{\mathrm{sh}} \cK^n  f \|_p & \lc_p  2^{-n} \|f\|_p, \quad 1<p\le \infty, \,\, r\ge 2.
\end{align*}
The kernel $K^n_t$ is supported in $\{x:|x|\le 2^{n-1}t\}$ and
from our assumptions it is easy to see that the  rescaled estimate
\[ \big \| V^r_{[1,2]}  \{2^{nd}K^n_t(2^n\cdot)*f\}  \big\|_q\lc 2^{-n} \|f\|_p, \quad 1\le p\le q\le \infty,\,  1\le r\le\infty,
\] holds. Moreover, using the bound for $\nabla K$  and $\nabla^2K$ we also get
\[ \big \| V^r_{[1,2]}  \{2^{nd}K^n_t(2^n\cdot)*\Delta_hf\}  \big\|_q\lc 2^{n(1/r-1)}|h|^{1/r}
\|f\|_p\] for $1\le p\le q\le \infty$, $1<r\le\infty$.
Applications of Theorem \ref{thm:Vr} and  Theorem \ref{thm:ellr} (for $\ell^r$-sums and with the choice of $B_2$ a subspace of $V^r_{[1,2]}$ of large finite dimension) together with the monotone convergence theorem  yield 
\begin{align*} 
| \inn{V^r_\dyad \cK^n f}{\om}| & \lc 2^{-n}\Lamaxga_{p,q'} (f,\om), \text{ $r>2,$}
\\
|\inn{V^r_{\mathrm{sh}} \cK^n f}{\om} | & \lc 2^{-n(1-1/r)}\Lamaxga_{p,q'} (f,\om), \text{ $r\ge2.$}
\end{align*}
    The proof is completed by summation in $n$.
\end{proof}

\section{Sparse domination: Cases where $p=1$ or $q=\infty$}
\label{appendixB}
Here we describe analogues of our main result Theorem \ref{mainthm}  which cover  cases where $p=1$ or $q=\infty$; we refer to Remark (iv) following the statement of Theorem \ref{mainthm} for an explanation of why these cases need to be treated separately.
We formulate three different results, one for $p=1$, $q<\infty$, one for $q=\infty$, $p>1$, and one for  $p=1$, $q=\infty$. This allows us to recover the classical case of Calder\'on--Zygmund operators, although we do not claim universality of sparse-domination results here: for example, we do not recover the sparse domination for Carleson-type operators from \cite{DPL14, Beltran2018, Beltran-Thesis}, neither the works for $p=1$ by Conde-Alonso, Culiuc, Di Plinio and Ou \cite{conde-alonso-etal} and by Lerner \cite{LernerRev2019} which also treat results on rough singular integrals, nor the works for $q=\infty$ which can often be upgraded to stronger pointwise sparse domination results of the type  \eqref{pointwisesparse} (see in particular 
\cite{LeNew}, \cite{LernerOmbrosi}, \cite{Lorist}). 

We will sketch the proofs of our results, indicating only what modifications need to be made compared to the proof of Theorem \ref{mainthm}.  Theorems \ref{thmB1}, \ref{thmB2} and \ref{thmB3} below have applications to  maximal operators, square functions and long variation operators (as formulated in \S\ref{sec:squarefct-etc}) similar to those of Theorem \ref{mainthm}. We leave the details to  the interested reader.

\subsection{The case $p=1, q< \infty$}

If $p=1$, one can drop the condition of weak type $(1,1)$. 
Our variant of Theorem \ref{mainthm} is then as  follows. 
\begin{thm}\label{thmB1}
 Let $1 \leq q<\infty$. Let $\{T_j\}_{j\in\bbZ}$ be a family of operators in $\mathrm{Op}_{B_1,B_2}$ such that
\begin{enumerate}
    \item[$\circ$] the support condition \eqref{support-assu} holds,
    \item[$\circ$] the restricted strong type $(q,q)$ condition \eqref{bdness-rt} holds,
    \item[$\circ$] the single scale $(1,q)$ condition \eqref{p-q-rescaled} holds,
    \item[$\circ$] the single scale $\varepsilon$-regularity conditions \eqref{p-q-rescaled-reg-a}, \eqref{p-q-rescaled-reg-b} hold with $p=1$.
\end{enumerate}
Define
\Be\notag\mathcal{C}= A(q)+A_\circ(1,q)  \log \big(2+
\tfrac{B}  {A_\circ(1,q)}\big).\Ee
Then,
for all integers $N_1, N_2$ with $N_1\le N_2$,
\Be\notag
\Big\|\sum_{j=N_1}^{N_2}T_j\Big \|_{\Sp_\ga(1,{B_1},{q'},{B_2^*})} \lesssim_{q,\eps,\gamma,d} \mathcal{C}.
\Ee
\end{thm}

\begin{proof}
We argue as in the proof of Theorem \ref{mainthm}, with the decomposition \eqref{first splitting} and the bound \eqref{good-fct}. The terms \eqref{termI}, \eqref{termII} are handled exactly as in the proof of Theorem \ref{mainthm}. Consider the splitting of $III$ as in \eqref{eq:IIIsplitting}. The terms
\eqref{eq:IIIthree-new} and \eqref{eq:IIIfour-new} are estimated as in the proof of Theorem \ref{mainthm}. We are thus only left with estimating the term
\begin{equation*}
III_1 + III_2
= \sum_{N_1\le j\le N_2} \sum_{\substack{W\in \cW \\ L(W)<j}} 
\inn{T_j b_{1,W} }{g_{2}}.
\end{equation*}
The argument in the proof of Theorem \ref{mainthm} does no longer work; recall that for $p>1$ these terms were bounded immediately via the weak type $(p,p)$ condition \eqref{bdness-wt} and the duality of $L^{p,\infty}_{B_2^{**}}$ and $L^{p',1}_{B_2^*}$.
Instead, here we will bound $III_1 + III_2$ using \eqref{p-q-rescaled} and the regularity condition \eqref{p-q-rescaled-reg-a}, close in spirit to the bounds of the terms $IV_1$ and $IV_2$ (defined in \eqref{eq:IVi def}) in the proof of Theorem \ref{mainthm}.

We let $0<\varepsilon'<\min \{1/q',\varepsilon\}$ and $\ell >0$ be as in  \eqref{elldef}, that is,
\[2^\ell <\frac{100}{\ep'}\big(2+ \frac{B}{A_\circ(1,q)}\big) \le 2^{\ell+1}.\] 
Let $\fR_j$ be the collection of dyadic subcubes of $Q_0$ of side length $2^j$. 
We tile $Q_0$ into such cubes  and write
 \begin{align*} 
III_1 + III_2&=\sum_s \widetilde {III}_s,
\end{align*}
where 
\begin{equation}\label{IIIs def APP}
\widetilde {III}_s=\sum_{N_1\le j\le N_2} \sum_{R\in \fR_j} 
\biginn{ \sum_{\substack{W\subset R\\
 L(W)=j-s}}
   T_j b_{1,W} } 
{
g_{2} \bbone_{\tr{R}} }.
\end{equation}
We first note 
\begin{subequations}
\begin{align} \label{TjAPP}
&\|T_j\|_{L^1_{B_1} \to L^q_{B_2}} \lc A_\circ(1,q)  2^{-jd(1-\frac 1q)}, 
\\ \label{TjregaAPP}
&\|T_j(\mathrm I -\bbE_{s-j})\|_{L^1_{B_1} \to L^q_{B_2}} \lc_{\varepsilon} B 2^{-jd(1-\frac 1q)} 2^{-\eps's}.
\end{align}
\end{subequations}
where
\eqref{TjAPP} follows from the single scale $(1,q)$ condition \eqref{p-q-rescaled} and 
\eqref{TjregaAPP} follows from the single scale $\varepsilon$-regularity condition \eqref{p-q-rescaled-reg-a} and Corollary \ref{cor:reg}.

For $L(W)=j-s$, we have $b_{1,W} =(I-\bbE_{s-j}) f_{1,W}$.
Let $R\in \fR_j$.
By \eqref{TjAPP} we get 
\begin{subequations}
\begin{align}
\notag &\Big|\biginn{ \sum_{\substack{W\subset R\\
 L(W)=j-s}}
   T_j b_{1,W} } 
{
g_2 \bbone_{\tr{R}} } \Big|
\\ \notag
&\le\Big\| T_j  \sum_{\substack{W\subset R\\
 L(W)=j-s}} b_{1,W} \Big\|_{L^q_{B_2}}
 \big\|
g_2 \bbone_{\tr{R}} \big \|_{L^{q'}_{B_2^*}}
\\ \notag
&\lc
A_\circ(1,q)
|R|^{-(1-\frac 1q)}
\sum_{\substack{W\subset R\\
 L(W)=j-s}}\| b_{1,W} \|_{L^1_{B_1}} 
 |R|^{1/q'}\jp{f_2}_{\tr{Q_0},q'} \\ \label{requires p=1}
 & \lesssim  A_\circ(1,q)  \sum_{\substack{W\subset R\\
 L(W)=j-s}}| W | \jp{f_1}_{Q_0,1}  \jp{f_2}_{\tr{Q_0},q'},
  \end{align}
  and by
  \eqref{TjregaAPP}, 
\begin{align}
\notag &\Big|\biginn{ \sum_{\substack{W\subset R\\
 L(W)=j-s}}
   T_j b_{1,W} } 
{
g_2 \bbone_{\tr{R}} } \Big|
\\ \notag
&\le\Big\| T_j (I-\bbE_{s-j}) \sum_{\substack{W\subset R\\
 L(W)=j-s}} f_{1,W} \Big\|_{L^q_{B_2}}
 \big\|
g_2 \bbone_{\tr{R}} \big \|_{L^{q'}_{B_2^*}}
\\ \notag
&\lc_{\varepsilon}
B 2^{-\eps' s}
|R|^{-(1-\frac 1q)}
\sum_{\substack{W\subset R\\
 L(W)=j-s}}\| f_{1,W} \|_{L^1_{B_1}} 
 |R|^{1/q'}\jp{f_2}_{\tr{Q_0},q'} \\ \label{requires p=1wcanc}
 & \lesssim  B 2^{-\varepsilon' s}  \sum_{\substack{W\subset R\\
 L(W)=j-s}}| W | \jp{f_1}_{Q_0,1}  \jp{f_2}_{\tr{Q_0},q'}   .
  \end{align}
  \end{subequations}
  Note that in obtaining the above bounds we have used \eqref{upperboundsWh} and \eqref{Linfty-g}.
  
  In the above definition \eqref{IIIs def APP} for $\widetilde {III}_s$  we  write $j=sn+i$ with $i=1,\dots, s$ so that  
  \begin{multline*} 
  |\widetilde {III}_s|\lc_{\varepsilon} 
  \min\{A_\circ(1,q),  B 2^{-\eps' s}\} 
  \\
  \sum_{i=1}^{s} 
  \sum_{\substack{n\in \bbZ\\ sn+i\in [N_1,N_2]}}
  \sum_{R\in \fR_{sn+i}} \sum_{\substack{W\subset R\\
 L(W)=s n + i-s}}| W |  \jp{f_1}_{Q_0,p}  \jp{f_2}_{\tr{Q_0},q'}  .
\end{multline*}
Now interchange the order of summation; 
here consider for  fixed $W\in \cW$  the set  of all triples $(R,n,i)$ such that $L(W)= s(n-1)+i$, $R\in \fR_{sn+i}$ and $W\subset \tr{R}$, and observe that the cardinality of this set is 1. Combining this with the above estimates and summing over the disjoint cubes $W\in \cW$ we obtain the bound
\begin{align*} 
|III_1+III_2|&\lc_{d,\ga,\varepsilon} 
\jp{f_1}_{Q_0,1} 
    \jp{f_2}_{\tr{Q_0},q'}  \sum_{W \in \cW} |W| 
    \sum_{s=1}^\infty \min\{A_\circ(1,q),  B 2^{-\eps' s}\}  \\
    & \lesssim A_\circ(1,q) \log\big(2+ \frac{B}{A_\circ(1,q)} \big) |Q_0| \jp{f_1}_{Q_0,1} 
    \jp{f_2}_{\tr{Q_0},q'},
  \end{align*}
  as desired.
 \end{proof}
 
 In the spirit of Section \ref{sec:nec}, it is possible to deduce that the sparse bound in Theorem \ref{thmB1} implies that the multi-scale sums $\sum_{j=N_1}^{N_2} T_j$ are of weak-type $(1,1)$. The proof of this fact is slightly different than the one given in Section \ref{sec:nec} for $p>1$, as it cannot rely on the duality between $L^{p,\infty}$ and $L^{p,1}$. We refer to the reader to \cite[Appendix B]{conde-alonso-etal} for details.

\subsection{The case $p>1$, $q=\infty$} If $q=\infty$ one can  drop the restricted strong type $(q,q)$ condition \eqref{bdness-rt}. Our variant of Theorem \ref{mainthm} is then as follows.

\begin{thm}\label{thmB2}
 Let $1<p\le \infty$. Let $\{T_j\}_{j\in\bbZ}$ be a family of operators in $\mathrm{Op}_{B_1,B_2}$ such that
\begin{enumerate}
    \item[$\circ$] the support condition \eqref{support-assu} holds,
    \item[$\circ$] the weak type $(p,p)$ condition \eqref{bdness-wt} holds,
    \item[$\circ$] the single scale $(p,\infty)$ condition \eqref{p-q-rescaled} holds,
    \item[$\circ$] the single scale $\varepsilon$-regularity conditions \eqref{p-q-rescaled-reg-a}, \eqref{p-q-rescaled-reg-b} hold with $q=\infty$.
\end{enumerate}
Define
\Be\notag\mathcal{C}= A(p)+A_\circ(p,\infty)  \log \big(2+
\tfrac{B}  {A_\circ(p,\infty)}\big).\Ee
Then,
for all integers $N_1, N_2$ with $N_1\le N_2$,
\Be\notag
\Big\|\sum_{j=N_1}^{N_2}T_j\Big \|_{\Sp_\ga(p,{B_1},{1},{B_2^*})} \lesssim_{p,\eps,\gamma,d} \mathcal{C}.
\Ee 
\end{thm}

\begin{proof} Again we argue as in the proof of Theorem \ref{mainthm} and describe the main induction step. Using the previous notations we now decompose
\Be
\inn{Sf_1}{f_2} = \mathrm I + \mathrm {II}+\mathrm{III}+\mathrm{IV}
\Ee
where
\begin{align*}
    \mathrm{I}&=   \sum_{W \in \cW} \inn{S_W f_1}{f_2}
    \\
    \mathrm{II}&= \inn{(S-\sum_{W \in \cW} S_W)f_1}{g_2}
    \\
    \mathrm{III}&= \inn{(S-\sum_{W \in \cW} S_W)g_1}{b_2}
    \\
    \mathrm{IV}&= \inn{(S-\sum_{W \in \cW} S_W)b_1}{b_2}.
\end{align*}
Note that the numbering here is slightly different from the one in the proof of Theorem \ref{mainthm}.
We deal with the term $\mathrm I $ using the induction hypothesis as in the proof of Theorem \ref{mainthm} and, using the argument therein, it suffices to show that the terms
$\mathrm{II}$, $\mathrm{III}$ and $\mathrm{IV}$ are bounded by $ c\, \cC |Q_0| \jp{f_1}_{Q_0,p} \jp{f_2}_{3Q_0,1}  $.

We first consider $\mathrm{II}=\mathrm{II}_1-\mathrm{II}_2$ where 
\[ \mathrm{II}_1=\inn{Sf_1}{g_2}, \qquad \mathrm{II}_2=\sum_W\inn{ S_Wf_1}{g_2}.\]
Here we use the weak type $(p,p)$ assumption \eqref{bdness-wt} for $p>1$ and \eqref{ri-g} for $r_2=p'<\infty$ to get
\begin{align*}
\inn{Sf_1}{g_2}&
\le \| Sf_1 \|_{L^{p,\infty}_{B_2^{**}}} \| g_2 \bbone_{3Q_0} \|_{L^{p',1}_{B_2^*}} 
\\ &\lc A(p) \|f_1\|_p \jp{f_2}_{3Q_0,1} |Q_0|^{1-1/p }
\\ & \lc A(p)  |Q_0| \jp{f_1}_{p,Q_0}  \jp{f_2}_{3Q_0,1} 
\end{align*}
and 
\begin{align*}
\sum_{W \in \cW}\inn{S_Wf_1}{g_2}&
\le \sum_W \|S_W[f_1 \bbone_W] \|_{L^{p,\infty}_{B_2^{**}}} \| g_2 \|_{L^\infty_{B_2^*}} \| \bbone_{3W} \|_{L^{p',1}_{B_2^*}} 
\\
&\lc A(p)\sum_{W \in \cW} \|f_1\bbone_W\|_p \jp{f_2}_{3Q_0,1} |W|^{1-1/p }
\\&\lc A(p) \sum_{W \in \cW} |W| \jp{f_1}_{Q_0,p}  \jp{f_2}_{3Q_0,1}
\end{align*}
and hence, by the disjointness of the cubes $W \in \cW$,
\Be\label{rmII} |\mathrm {II}|\le |\mathrm{II}_1|+|\mathrm{II}_2|\lc A(p) 
|Q_0| \jp{f_1}_{p,Q_0}  \jp{f_2}_{3Q_0,1} .
\Ee

The last term $\mathrm{IV}$ corresponds exactly to the sum $III_3+III_4$ in the proof of Theorem \ref{mainthm}, defined in \eqref{eq:IIIthree-new}, \eqref{eq:IIIfour-new}, and it is therefore treated in the same way; here the weak type and restricted strong type assumptions are not used. In particular, we obtain
\Be\label{rmIV} |\mathrm {IV}|\lc A_\circ(p,\infty) \log \big(2+\frac{B}{A_\circ(p,\infty)}\big) 
|Q_0| \jp{f_1}_{p,Q_0}  \jp{f_2}_{3Q_0,1}.
\Ee

It remains to bound the term $\mathrm{III}$.
By the definition of $g_1$ and $S_W$ we have
\begin{equation*}
    \sum_{W \in \cW} S_W g_1 = \sum_{W \in \cW} \sum_{j=N_1}^{L(W)} T_j [ \av_W[f_1] \bbone_W] = \sum_{j=N_1}^{N_2} \sum_{\substack{W \in \cW \\ j \leq L(W)}} T_j[\av_W[f_1] \bbone_W] 
\end{equation*}
and thus we may 
split  $\mathrm{III}=\mathrm{III}_1+\mathrm{III}_2$ where
\begin{align*}
    \mathrm{III}_1 & = \inn{S[g_1 \bbone_{\Omega^\complement}]}{b_2} \\
    \mathrm{III}_2 & =  \sum_{j=N_1}^{N_2} \sum_{\substack{W \in \cW \\ L(W) < j}} \inn{T_j[\av_W[f_1] \bbone_W]}{b_2}.
\end{align*}
Let $\fR_j$ be the collection of dyadic subcubes of $Q_0$ of side length $2^j$. 
We tile $Q_0$ into such cubes  and write 
\begin{align*}
    \mathrm{III}_1 & = \sum_{j=N_1}^{N_2} \sum_{R \in \fR_j}\inn{T_j[g_1 \bbone_{\Omega^\complement \cap R}]}{b_2}  = \sum_{j=N_1}^{N_2} \sum_{R \in \fR_j}\inn{T_j[g_1 \bbone_{\Omega^\complement \cap R}]}{\sum_{W' \in \cW}b_{2,W'} \bbone_{3R}}.
\end{align*}
Next, note that in order to  have  $\inn{T_j[g_1 \bbone_{\Omega^\complement \cap R}]}{\sum_{W' \in \cW}b_{2,W'} \bbone_{3R}} \neq 0$ we must have  that $\Omega^\complement  \cap R \neq \emptyset$ and $W'\cap 3R \neq \emptyset$. As $W' \in \cW$, the above implies
\begin{equation*}
    5\,\diam (W') \leq \dist (W', \Omega^\complement) \leq 3 \,\diam (R)
\end{equation*}
and therefore $L(W') < j$. Thus, 
\begin{equation}\label{rmIII-1}
    \mathrm{III}_1=\sum_{j=N_1}^{N_2} \sum_{R \in \fR_j}\inn{T_j[g_1 \bbone_{\Omega^\complement \cap R}]}{\sum_{\substack{W' \in \cW \\ L(W') < j}}b_{2,W'} \bbone_{3R}}.
\end{equation}
We next decompose  $\mathrm{III}_2=\mathrm{III}_{2,1} + \mathrm{III}_{2,2}$, where
\begin{subequations}
\begin{align} \label{rmIII-21}
\mathrm{III}_{2,1}&:=\sum_{j=N_1}^{N_2} \biginn{T_j\big[ \sum_{\substack{W \in \cW \\ L(W) < j}} \av_W[f_1] \bbone_W\big]}{\sum_{\substack{W' \in \cW \\ L(W') \geq j}}b_{2,W'}} ,\\ \label{rmIII-22}
\mathrm{III}_{2,2}&:=\sum_{j=N_1}^{N_2} \sum_{\substack{W \in \cW \\ L(W) < j}} \biginn{T_j[\av_W[f_1] \bbone_W]}{\sum_{\substack{W' \in \cW \\ L(W') < j}}b_{2,W'}}.
\end{align}
\end{subequations}

The term $\mathrm{III}_{2,1}$ can be treated as in the estimation of the term $III_3$ in the proof of Theorem \ref{mainthm}, defined in \eqref{eq:IIIthree-new}, as  cancellation does not play a role in this argument. The geometry expressed in \eqref{restriction on j} is crucial, i.e. we have 
likewise 
\begin{equation}\label{restriction on j B}
\def\arraystretch{1.4}
\left.
\begin{array}{l}
      \inn{T_j [\av_W f_1 \bbone_W] }{b_{2,W'}} \neq 0 \\ 
L(W)<j\le L(W')
\end{array}\right\} \quad  \implies  \quad
j\le L(W')\le L(W)+2\le j+2. 
\end{equation}
This implies
\begin{align*}
&|\mathrm{III}_{2,1}| \le \sum_{N_1 \leq j \leq N_2}\sum_{\substack{W'\in \cW:\\ j\le L(W')\le j+2}}\,\, \sum_{\substack{W\in \cW:\\L(W')-2\le L(W)\le j\\ W\subset 3W'}}|\inn{T_j [\av_W [f_1]\bbone_W] }{b_{2,W'}}|
\\ &\le A_\circ(p,\infty) 
\sum_{N_1 \leq j \leq N_2} 2^{-jd/p} 
\sum_{\substack{W,W'\in \cW:\,W\subset 3W'\\j\le L(W')\le j+2\\
L(W')-2\le L(W)\le j}} 
\|\av_W [f_1] \bbone_W  \|_{L^p_{B_1}}\|b_{2,W'}\|_{L^{1}_{B_2^*}}
\\
&\lc  A_\circ(p,\infty) \jp{f_1}_{Q_0,p} \jp {f_2}_{\tr{Q_0},1} 
\\&\qquad\qquad \times 
\sum_{N_1 \leq j \leq N_2}
\sum_{\substack{W,W':\,W\subset 3W'\\j\le L(W')\le j+2\\
L(W')-2\le L(W)\le j}} 
2^{-jd/p} |W|^{1/p} |W'|
\\
&
\lc A_\circ(p,\infty)
\jp{f_1}_{Q_0,p} \jp {f_2}_{\tr{Q_0},1} \sum_{\substack W'\in \cW} 
|W'| 
\lc 
A_\circ(p,\infty) |Q_0|
\jp{f_1}_{Q_0,p} \jp {f_2}_{\tr{Q_0},1}. 
\end{align*}

The terms  $\mathrm{III}_1$ and $\mathrm{III}_{2,2}$ can be treated in a similar way as in the estimation of the terms  $IV_1, IV_3$ (defined in \eqref{eq:IVi def}) in the proof of Theorem \ref{mainthm}. 
Let $0 < \varepsilon' < \min\{1/p,\varepsilon\}$ and $\ell>0$ be as in \eqref{elldef}. Then we split
\[ \mathrm{III}_1= \mathrm{III}_{1}^{\mathrm{lg}}  +\mathrm{III}_{1}^{\mathrm{sm}}, \qquad 
\mathrm{III}_{2,2}=\mathrm{III}_{2,2}^{\mathrm{lg}} +\mathrm{III}_{2,2}^{\mathrm{sm}} \]
where
\begin{align*}
    \mathrm{III}_1^{\mathrm{lg}} &=\sum_{j=N_1}^{N_2} \sum_{R \in \fR_j} \sum_{s=1}^\ell \biginn{T_j[g_1 \bbone_{\Omega^\complement \cap R}]}{\sum_{\substack{W' \in \cW \\ L(W') =j-s}}b_{2,W'} \bbone_{3R}}
    \\
    \mathrm{III}_1^{\mathrm{sm}} &=\sum_{j=N_1}^{N_2} \sum_{R \in \fR_j} \sum_{s=\ell+1}^\infty \biginn{g_1 \bbone_{\Omega^\complement \cap R}}{T_j^*\big[\sum_{\substack{W' \in \cW \\ L(W') =j-s}}b_{2,W'} \bbone_{3R}\big]}
\end{align*}
and 
\begin{align*}
    \mathrm{III}_{2,2}^{\mathrm{lg}}&=\sum_{j=N_1}^{N_2} \sum_{R \in \fR_j} \sum_{s=1}^\ell  \inn{\sum_{\substack{W \subset R \\ L(W) < j}} T_j[\av_W[f_1] \bbone_W]}{ \sum_{\substack{W' \in \cW \\ L(W') = j-s}}b_{2,W'} \bbone_{3R}},
    \\
     \mathrm{III}_{2,2}^{\mathrm{sm}}&=\sum_{j=N_1}^{N_2} \sum_{R \in \fR_j} \sum_{s=\ell+1}^\infty  \inn{\sum_{\substack{W \subset R \\ L(W) < j}} \av_W[f_1] \bbone_W}{ T_j^* \big[\sum_{\substack{W' \in \cW \\ L(W') = j-s}}b_{2,W'} \bbone_{3R} \big]}.
\end{align*}
Observe that the terms $\mathrm{III}_1^{\mathrm{sm}} $,
$\mathrm{III}_{2,2}^{\mathrm{sm}}$ involve very small cubes $W'$ for which the cancellation of $b_{2,W'}$ can be most effectively used.  The terms  $\mathrm{III}_1^{\mathrm{lg}} $,
$\mathrm{III}_{2,2}^{\mathrm{lg}}$ involve larger cubes;  for these terms it is more effective to  use  the single scale $(p,q)$ conditions \eqref{p-q-rescaled}.

We note  that the terms $ \mathrm{III}_1^{\mathrm{lg}}$ and  $\mathrm{III}_{2,2}^{\mathrm{lg}}$ behave very similarly, and also the terms 
$ \mathrm{III}_1^{\mathrm{sm}}$ and  $\mathrm{III}_{2,2}^{\mathrm{sm}}$.

Indeed, if $h_R$, $\widetilde h_R$  denote either of the first functions on the bilinear form,
\[ h_R(x)= g_1(x) \bbone_{\Omega^\complement\cap R}, \quad \widetilde h_R(x) =
\sum_{\substack{W\subset R\\L(W)<j}} \av_W[f_1] \bbone_W(x),\]
then it follows from the definition of $\Omega^\complement$ and the disjointness of the $W \in \cW$ that $h_R$, $\widetilde h_R$ share the relevant property 
\begin{equation*}
    \| h_R\|_{L^r_{B_1}} ,\,\|\widetilde h_R\|_{L^r_{B_1}}  \leq |R|^{1/r} \jp{f_1}_{Q_0,p}, \quad 1 \leq r \leq \infty,
\end{equation*}
which we will use with $r=p$.

By the above considerations, the hypothesis \eqref{p-q-rescaled} and \eqref{upperboundsWh} we have \begin{align*} \mathrm{III}_1^{\mathrm{lg}} &=
 \sum_{j=N_1}^{N_2} \sum_{R \in \fR_j}  \sum_{s=1}^\ell  \inn{T_j h_R}{ \sum_{\substack{W' \in \cW \\ L(W') = j-s_2}}b_{2,W'} \bbone_{3R} } \\
    & \lesssim\sum_{s=1}^\ell \sum_{j=N_1}^{N_2} \sum_{R \in \fR_j}
    A_\circ(p,\infty) 2^{-jd/p}\|h_R\|_{L^p_{B_1}} 
\Big\| \sum_{\substack{W' \in \cW \\ L(W') = j-s}}b_{2,W'} \bbone_{3R} \Big\|_{L^{1}_{B_2^*}}
\\
&\lc \sum_{s=1}^\ell  A_\circ(p,\infty) \jp{f_1}_{Q_0,p} \sum_{j=N_1}^{N_2}
 \sum_{R \in \fR_j}\sum_{\substack{W' \subset 3R \\ L(W') = j-s}}\|b_{2,W'}  \|_{L^{1}_{B_2^*}}
 \\
 &\lc  A_\circ(p,\infty)  \jp{f_1}_{Q_0,p} \jp{f_2}_{3Q_0,1} \sum_{s=1}^\ell \sum_{j=N_1}^{N_2}
 \sum_{\substack{W' \in \cW \\ L(W') = j-s}}|W'|
\\&\lc  \ell A_\circ(p,\infty) 
|Q_0| \jp{f_1}_{Q_0,p} \jp{f_2}_{3Q_0,1} 
\end{align*} 
and hence, by the definition of $\ell$,
\[ \mathrm{III}_{1}^{\mathrm{lg}}
\lc A_\circ(p,\infty) 
\log (2+\frac{B}{A_\circ(p,\infty)}) 
|Q_0| \jp{f_1}_{Q_0,p} \jp{f_2}_{3Q_0,1} .
\] Similarly we show (after replacing $h_R$ with $\widetilde h_R$ in the above calculation)
 \[
 \mathrm{III}_{2,2} ^{\mathrm{lg}} 
 \lc  A_\circ(p,\infty) \log (2+\frac{B}{A_\circ(p,\infty)}) |Q_0| \jp{f_1}_{Q_0,p} \jp{f_2}_{3Q_0,1} .\]
 
 For the estimation of $\mathrm{III}_1^{\mathrm{sm}}$, $\mathrm{III}_{2,2} ^{\mathrm{sm}}$ we use the $\eps$-regularity property \eqref{p-q-rescaled-reg-b} and Corollary \ref{cor:reg} to get
 \begin{align} \label{TjregbAPP}
\|T_j^*(\mathrm I -\bbE_{s-j})\|_{L^{1}_{B_1} \to L^{p'}_{B_2}} &\lc_{\varepsilon} B 2^{-jd/p} 2^{-\eps s}.
\end{align}
Moreover we use the formula  $b_{2,W'}=(I- \mathbb{E}_{s-j})f_{2,W'}$,  valid for 
$L(W')=j-s$.
Thus, via H\"older's inequality
\begin{align*}
    &\sum_{j=N_1}^{N_2} \sum_{R \in \fR_j}  \sum_{s=\ell+1}^\infty  \inn{h_R}{ T_j^* \big[\sum_{\substack{W' \in \cW \\ L(W') = j-s}}b_{2,W'} \bbone_{3R} \big]} \\
    &\lesssim \sum_{j=N_1}^{N_2} \sum_{s=\ell+1} 
    \|h_R\|_{L^p_{B_1}} \|T^*_j(I-\bbE_{s-j} )\|_{L^1_{B_2^*}\to L^{p'}_{B_1}}
     \Big\| \sum_{\substack{W' \in \cW \\ L(W') = j-s}}f_{2,W'} \bbone_{3R} \Big\|_{L^{1}_{B_2^*}}
    \\
    & \lesssim\sum_{j=N_1}^{N_2} \sum_{R \in \fR_j} \sum_{s=\ell+1}^\infty |R|^{1/p} \jp{f_1}_{Q_0,p} B 2^{-\varepsilon' s} 2^{-jd/p} \Big\| \sum_{\substack{W' \in \cW \\ L(W') = j-s_2}}f_{2,W'} \bbone_{3R} \Big\|_{L^{1}_{B_2^*}} 
    \\
    & \lesssim \sum_{j=N_1}^{N_2}  \sum_{s=\ell+1}^\infty B 2^{-\varepsilon' s} \jp{f_1}_{Q_0,p} \sum_{R \in \fR_j}\sum_{\substack{W' \subseteq 3R\\ L(W')=j-s}} \| f_{2,W'}  \|_{L^{1}_{B_2^*}}  \\
    & \lesssim \sum_{j=N_1}^{N_2}  \sum_{s=\ell+1}^\infty B 2^{-\varepsilon' s}  \jp{f_1}_{Q_0,p} \jp{f_2}_{3Q_0,1} 
    \sum_{R \in \fR_j}\sum_{\substack{W' \subseteq 3R\\ L(W')=j-s}} |W'| .
\end{align*}
We sum in $W'$ and then use $\sum_{s=\ell+1}^\infty B 2^{-\eps' s} \lc A_\circ(p,\infty) $ to obtain 
\[ |\mathrm{III}_1^{\mathrm{sm}} |\lc A_\circ(p,\infty)
|Q_0|\jp{f_1}_{Q_0,p} \jp{f_2}_{3Q_0,1} .\]
In exactly the same way (replacing $h_R$ by $\widetilde h_R$) we obtain
\[ |\mathrm{III}_{2,2}^{\mathrm{sm}} |\lc A_\circ(p,\infty)
|Q_0|\jp{f_1}_{Q_0,p} \jp{f_2}_{3Q_0,1} .\]
This concludes the proof.
\end{proof}

\subsection{The case $p=1$ and $q=\infty$}

In this case we can get rid of both the weak $(p,p)$ and restricted strong type $(q,q)$ hypotheses, but we shall still assume either a weak-type estimate $(r,r)$ or restricted strong type $(r,r)$ for some $1 < r <\infty$. 

\begin{thm}\label{thmB3}
Let $\{T_j\}_{j\in\bbZ}$ be a family of operators in $\mathrm{Op}_{B_1,B_2}$ such that
\begin{enumerate}
    \item[$\circ$] the support condition \eqref{support-assu} holds,
    \item[$\circ$] there exists $r\in (1,\infty)$ so that either the weak type $(r,r)$ condition \eqref{bdness-wt} holds or  the restricted strong type $(r,r)$ condition \eqref{bdness-rt} holds,
    \item[$\circ$] the single scale $(1,\infty)$ condition \eqref{p-q-rescaled} holds,
    \item[$\circ$] the single scale $\varepsilon$-regularity conditions \eqref{p-q-rescaled-reg-a}, \eqref{p-q-rescaled-reg-b} hold with $p=1$ and $q=\infty$.
\end{enumerate}
Define
\Be\notag\mathcal{C}= A(r)+A_\circ(1,\infty)  \log \big(2+
\tfrac{B}  {A_\circ(1,\infty)}\big).\Ee
Then,
for all integers $N_1, N_2$ with $N_1\le N_2$,
\Be\notag
\Big\|\sum_{j=N_1}^{N_2}T_j\Big \|_{\Sp_\ga(1,{B_1},{1},{B_2^*})} \lesssim_{r,\eps,\gamma,d} \mathcal{C}.
\Ee
\end{thm}

\begin{proof}[Sketch of proof] 
We use the terminology in the proofs of Theorems \ref{thmB1} and \ref{thmB2}. An examination of the proofs   reveals that it only remains to establish  the inequality
\begin{equation} \label{goodgood}
    |\inn{(S-\sum_{W \in \cW} S_W) g_1}{g_2}|\lc A(r)  |Q_0| \jp{f_1}_{Q_0,1} \jp{f_2}_{3Q_0,1},
\end{equation}
either under the restricted strong type $(r,r)$ assumption \eqref{bdness-rt}, 
or under 
the weak  type $(r,r)$ assumption \eqref{bdness-wt}. Here we will strongly use \eqref{Linfty-g} for both $g_1$ and $g_2$.

We first verify \eqref{goodgood} assuming  \eqref{bdness-rt}. 
By H\"older's inequality,
\begin{align*}
|\inn{S g_1}{g_2}|&\lesssim \| S g_1 \|_{L^r_{B_1}} \| g_2 \|_{L^{r'}_{B_2^*}} \lesssim A(r)\| g_1 \|_{L^{r,1}_{B_1}} |Q_0|^{1/r'} \|g_2\|_{L^\infty_{B_2^*}}\\
& \lesssim A(r)  |Q_0| \jp{f_1}_{Q_0,p} \jp{f_2}_{3Q_0,1}.
\end{align*}
Moreover for each $W\in \cW$
\begin{align*}
|\inn{S_W g_1}{g_2}|&= 
|\inn{S_W [g_1\bbone_W]}{g_2\bbone_{3W} }|\le 
\| S_W [g_1\bbone_W] \|_{L^r_{B_1}} \| g_2\bbone_{3W}  \|_{L^{r'}_{B_2^*}} 
\\&\lesssim A(r)\| g_1 \bbone_W\|_{L^{r,1}_{B_1}} |W|^{1/r'} \|g_2\bbone_{3W}\|_{L^\infty_{B_2^*}}
\lesssim A(r)  |W| \jp{f_1}_{Q_0,p} \jp{f_2}_{3Q_0,1},
\end{align*} 
and by summing over the disjoint cubes $W\in \cW$ we obtain
\[\sum_{W\in \cW} \big|\inn{ S_W g_1}{g_2}\big|\lc
A(r)  |Q_0| \jp{f_1}_{Q_0,p} \jp{f_2}_{3Q_0,1}.\]
Combining the two bounds yields \eqref{goodgood} (under the assumption \eqref{bdness-rt}). 

We now  verify \eqref{goodgood} assuming  \eqref{bdness-wt}. First, by H\"older's inequality for Lorentz spaces, 
\begin{align*}
|\inn{S g_1}{g_2}|\lesssim \| S g_1 \|_{L^{r,\infty}_{B_1}} \| g_2 \|_{L^{r',1}_{B_2^*}} &\lesssim A(r) \| g_1 \|_{L^{r}_{B_1}} |Q_0|^{1/r'}
\|g_2\|_{L^\infty_{B_2^*}}
\\&\lesssim |Q_0| \jp{f_1}_{Q_0,p} \jp{f_2}_{3Q_0,1}.
\end{align*}
Similarly,  for all $W\in \cW$, 
\[
|\inn{S_W g_1}{g_2}|\lesssim |W| \jp{f_1}_{Q_0,p} \jp{f_2}_{3Q_0,1}\]
and then after summation 
\[\sum_{W \in \cW}
\big|\inn{S_W g_1}{g_2}|\lesssim |Q_0| \jp{f_1}_{Q_0,p} \jp{f_2}_{3Q_0,1}.\]
This yields \eqref{goodgood} (under the assumption \eqref{bdness-wt}). 
\end{proof}

\section{Facts about Fourier multipliers}\label{app:multipliers}
For completeness, we provide proofs of the facts stated in the remark after the definition of the $\mathcal{B}[m]$. The proofs will be given for scalar multipliers but they carry over  to the setting with  $\sL(\sH_1,\sH_2)$-valued  multipliers. We start with the following simple observations.

\begin{lem}\label{lem:two-spatial-loc}
Let $\Psi\in C^\infty_c(\bbR^d)$ be supported in $\{x \in \bbR^d:1/2<|x|<2\}$.
Let $\Phi\in C^\infty_c(\bbR^d)$ be supported in $\{x \in \bbR^d:|x|<2\}$. 
Let $N>d$ and $\ka$ be such that \[\sup_{x \in \bbR^d}(1+|x|)^N|\ka(x)|\le 1.\] 
Then the following holds.

(i) Let $1\le \rho\le R/8$. Then
\[ \big \|\big[ \widehat \ka (m*R^d\widehat \Psi(R\cdot))\big] *\rho^d \widehat \Phi(\rho\cdot) \big\|_{M^{p,q} }
\lc R^{d-N} \big\|m*R^d\widehat \Psi(R\cdot)\big\|_{M^{p,q}} \]

(ii) Let $1\le \rho \le R/8$. Then
\[ \big \|\big[ \widehat\ka  (m*\rho^d\widehat \Phi(\rho\cdot))\big] *R^d \widehat \Psi(R\cdot) \big\|_{M^{p,q} }
\lc R^{d-N} \big\|m*\rho^d\widehat \Phi(\rho\cdot)\big\|_{M^{p,q}}  \]
\end{lem}

\begin{proof} 
Let $K=\cF^{-1}[m]$ and set $\|K\|_{cv(p,q)}:= \|\widehat K\|_{M^{p,q}}$.
The expression in (i) is  equal to
\[ \Big\| \Phi(\rho^{-1}\cdot) \int \ka(y) K(\cdot -y) \Psi(R^{-1}(\cdot-y)) \, dy\Big\|_{cv(p,q)}.
\] Observe that by the support properties of $\Phi$, $\Psi$  the integral in $y$ is extended  over
$R/2-2\rho\le |y|\le 2R+2\rho$, hence $|y|\in (R/4,4R)$. Thus the displayed quantity is bounded by
\begin{align*}&\int_{R/4\le|y|\le 4R}|\ka (y)| \|K(\cdot-y)\Psi(R^{-1}(\cdot-y))\|_{cv(p,q)} \, dy
\\ &\le \int_{R/4\le|y|\le 4R}|\ka (y)| \, dy\,\|m*R^d \widehat \Psi(R\cdot)\|_{M^{p,q}}  
\end{align*} and the desired bound follows from the hypothesis on $\ka$.  Part (ii) is proved in the same way.
\end{proof}

\begin{lem} \label{lem:multiplication-by-smooth} Let $\Psi_n$, $n\ge 0$, be as in \S\ref{sec:more-precise}. Let $N>d$ and let  $\chi$ be such that
$\|\partial^\alpha_\xi\chi\|_1\le A$ for all $\alpha \in \bbN_0$ such that $|\alpha|\le N$. Let $h\in L^1(\widehat{\bbR}^d)$ be supported in $\{\xi \in \widehat{\bbR}^d:1/2\le |\xi|\le 2\}$. Then
\[ \|(h\chi)*\widehat \Psi_\ell\|_{M^{p,q} } \lc A\sum_{n=0}^\infty C_{N-d}(n,\ell) \|h*\widehat\Psi_n\|_{M^{p,q}}\]
for any $\ell \geq 0$, where
\Be\label{eqn:Cnl} C_{N_1}(n,\ell) := \begin{cases} 
1&\text{ if } \ell-5\le n\le \ell+5,
\\
2^{-\ell N_1}  &\text{if } 0\le n<\ell-5,
\\
2^{-n N_1} &\text{ if }  \ell+5<n.
\end{cases}
\Ee 
\end{lem}
\begin{proof}
We write $(h\chi)*\widehat{\Psi}_\ell=  \sum_{n=0}^\infty [(h*\widehat{\Psi}_n)\chi]* \widehat{ \Psi}_\ell$.
The result then follows by noting that $|\cF^{-1}[\chi](x)|\lc (1+|x|)^{-N}$
and an application of Lemma \ref{lem:two-spatial-loc}.
\end{proof}

\subsection{Multiplication by smooth symbols}\label{sec:bmsmooth-pf}

The above observations can be applied to show that the space defined by the finiteness of $\cB[m]$ in \eqref{eqn:Bm} is invariant under multiplication with multipliers satisfying a standard symbol  of order $0$ assumption. There is of course also a corresponding similar and immediate statement for $\cB_\circ[m]$. 

\begin{lem}\label{lem:mult-by-symbols0}
Let $a \in C^{\infty}(\widehat{\bbR}^d)$. Then
$$\cB[am]\lc \cB[m] \sum_{|\alpha|\le 2d+1} \sup_{\xi \in \widehat{\bbR}^d} |\xi|^{|\alpha|} |\partial^\alpha a(\xi)|,$$
where $\alpha \in \bbN_0$. Consequently, if $|\partial^\alpha a(\xi)| \lesssim_\alpha (1+|\xi|)^{-|\alpha|}$ for all $\xi \in \widehat{\bbR}^d$ and all $\alpha \in \bbN_0^d$, we have  $\cB[am]\lc \cB[m]$.
\end{lem}
\begin{proof}
Let $\widetilde \phi\in C^\infty_c(\widehat{\bbR}^d)$ be supported in $\{\xi \in \widehat{\bbR}^d: 1/4\leq |\xi|\le 4\}$ and such that $\widetilde\phi(\xi)=1$ for $1/2\le|\xi|\le 2.$ Let $a^t(\xi)= \widetilde \phi(\xi) a(t\xi)$. %
\[  [\vphi a(t\cdot)m(t\cdot)]*\widehat{\Psi}_\ell= \sum_{n=0}^\infty 
\big[\big([\vphi m(t\cdot)]*\widehat {\Psi}_n\big) a^t\big]\,*\, \widehat {\Psi}_\ell .\]
We have
$\sum_{|\alpha|\le 2d+1}| \partial^\alpha a^t (\xi)| \lc 1$, uniformly in $t$. 
By Lemma \ref{lem:multiplication-by-smooth} with $N_1=2d+1$,
\begin{align*}
   & \sum_{\ell=0}^\infty (1+\ell)2^{\ell d(1/p-1/q)} \|[\phi a(t\cdot)m(t\cdot)]*\widehat{\Psi}_\ell\|_{M^{p,q}}
   \\& \lc
   \sum_{\ell=0}^\infty \sum_{n=0}^\infty C_{d+1}(n,\ell) (1+\ell)2^{\ell d(1/p-1/q)} \|[\phi m(t\cdot)]*\widehat{\Psi}_n\|_{M^{p,q}}
   \\&\lc \sum_{n=0}^\infty  (1+n)2^{n d(1/p-1/q)} \|[\vphi m(t\cdot)]*\widehat{\Psi}_n\|_{M^{p,q}}
\end{align*}
where  in the last line we used that   
\[ \sup_{n \geq 0} \sum_{\ell=0}^\infty \frac{1+\ell}{1+n} 2^{(\ell-n)d(\frac 1p-\frac 1q)} C_{d+1} (n,\ell) <\infty. \qedhere\]
\end{proof}

\subsection{\texorpdfstring{Independence of $\phi, \Psi$ in the finiteness of $\mathcal{B}[m]$}{Independence of bump functions}}
\label{sec:finiteness Bm}

The previous argument in Lemma \ref{lem:mult-by-symbols0} can also be used to show that the space defined by the finiteness of $\cB[m]$, $\cB_\circ[m]$  is independent of the specific choices of $\phi$, $\Psi$ in \S\ref{sec:more-precise}. We only give the argument for $\cB[m]$ and  a similar  reasoning applies to $\cB_\circ[m]$. 
\begin{lem} Denote the left hand side of \eqref{eqn:Bm} by $\cB[m,\phi,\Psi]$. 
Given two choices of $(\phi, \Psi)$ and 
$(\widetilde \phi, \widetilde\Psi)$ with the  specifications in the first paragraph of \S\ref{sec:more-precise}, there is a constant $C=C(\phi, \Psi)>1$ such that
\[ C^{-1} \cB[m,\phi, \Psi]\le \cB[m,\widetilde \phi, \widetilde \Psi]\le C\cB[m,\phi, \Psi].\]
\end{lem}
\begin{proof}
We show the second inequality.
Note that $\int_0^\infty |\phi (s\xi)|^2 \frac{ds}{s} \ge c>0$ for $\xi\neq 0$. Let $\beta^s$ be defined by
\[\widehat {\beta^s}(\xi)  = \frac{\widetilde\phi(s^{-1}\xi) \phi(\xi)}{\int_0^\infty |\phi(\sigma s^{-1}\xi)|^2 \frac{d\sigma}{\sigma}}. \]
We then have, in view of the support conditions on $\phi$  and $\widetilde{\phi}$,
\[\tilde \phi(\xi)=\int_{1/4}^4 \widehat{\beta^s}(s\xi) \phi(s\xi) \frac{ds}{s}\]
and hence 
\begin{align*} 
&\big\|\widetilde \phi m(t\cdot)*\widehat{\widetilde\Psi}_\ell\big\|_{M^{p,q}}
\le \int_{1/4}^4\big\|
[\widehat {\beta^s}(s\cdot)\phi(s\cdot)m(t\cdot)] *\widehat{\widetilde\Psi}_\ell\big\|_{M^{p,q}}\frac{ds}{s}
\\
&= \int_{1/4}^4 s^{d(\frac 1p-\frac 1q)}\big\| 
[\widehat{\beta^s}\phi(\cdot)m(ts^{-1} \cdot)] *s^{-d}\widehat{\widetilde\Psi}_\ell(s^{-1}\cdot)\big\|_{M^{p,q}}\frac{ds}{s}
\\&\lc \sum_{n=0}^\infty \int_{1/4}^4
\big\| \big(\widehat {\beta^s} \big(
[\phi(\cdot)m(ts^{-1} \cdot)]*\widehat{\Psi}_n\big)\big) *s^{-d}\widehat{\widetilde\Psi}_\ell(s^{-1}\cdot)\big\|_{M^{p,q}}\frac{ds}{s}.
\end{align*}
By Lemma \ref{lem:multiplication-by-smooth} this is dominated by
\[C_{N-d}(n,\ell) \int_{1/4}^4\big\|[\phi(\cdot)m(ts^{-1} \cdot)]*\widehat{\Psi}_n\big\|_{M^{p,q}}\,\frac{ds}{s},\]
where $C_{N-d} (n,\ell)$ is as in \eqref{eqn:Cnl}.
It is now easy to see that for  $N\ge 2d+1$
\begin{align*}
 &\sum_{\ell \geq 0} 2^{\ell d(\frac 1p-\frac 1q)}(1+\ell)   \big\|\widetilde \phi m(t\cdot)*\widehat{\widetilde\Psi}_\ell\big\|_{M^{p,q}}
\\&\lc \int_{1/4}^4
    \sum_{\ell\ge 0} \sum_{n\ge 0} C_{N-d}(n,\ell)2^{\ell d(\frac 1p-\frac 1q)}(1+\ell) \big\|[\phi(\cdot)m(ts^{-1} \cdot)]*\widehat{\Psi}_n\big\|_{M^{p,q}} \frac{ds}{s}
    \\
    &\lc \sup_\tau \sum_{n=0}^\infty 2^{nd(\frac 1p-\frac 1q)} \big\|[\phi(\cdot)m(\tau \cdot)]*\widehat{\Psi}_n\big\|_{M^{p,q}}.
\end{align*}
This establishes the inequality $\cB[m,\widetilde \phi, \widetilde \Psi]\le C\cB[m,\varphi, \Psi]$
and the converse follows by interchanging the roles of $(\phi, \Psi)$ and 
$(\widetilde \phi, \widetilde \Psi)$.
\end{proof}


\providecommand{\bysame}{\leavevmode\hbox to3em{\hrulefill}\thinspace}
\providecommand{\MR}{\relax\ifhmode\unskip\space\fi MR }
\providecommand{\MRhref}[2]{%
	\href{http://www.ams.org/mathscinet-getitem?mr=#1}{#2}
}
\providecommand{\href}[2]{#2}


\end{document}